\documentclass[11pt]{amsart}
\usepackage{graphicx}
\usepackage{pgf,tikz,pgfplots}
\usepackage{mathrsfs}
\usepackage{mathpple}
\usepackage{tikz-cd}
\usepackage{amsmath}
\usepackage{tikz}
\usepackage{mathdots}
\usepackage{yhmath}
\usepackage{cancel}
\usepackage{color}
\usepackage{siunitx}
\usepackage{array}
\usepackage{multirow}
\usepackage{amssymb}
\usepackage{gensymb}
\usepackage{tabularx}
\usepackage{calc}
\usepackage{booktabs}
\usetikzlibrary{fadings}
\usetikzlibrary{patterns}
\usetikzlibrary{shadows.blur}
\usetikzlibrary{shapes}

\usepackage{hyperref,xcolor}
\definecolor{wine-stain}{rgb}{0.5,0,0}
\hypersetup{
  colorlinks,
  linkcolor=wine-stain,
  linktoc=all
}
\usepackage{accents}

\newcommand{\vardbtilde}[1]{\tilde{\raisebox{0pt}[1\height]{$\tilde{#1}$}}}
\usepackage{latexsym,bm,amsmath,amssymb}

\usetikzlibrary{arrows}
\newtheorem{thm}{Theorem}[section]
\newtheorem{cor}[thm]{Corollary}
\newtheorem{lem}[thm]{Lemma}
\newtheorem{prop}[thm]{Proposition}
\theoremstyle{definition}
\newtheorem{exa}[thm]{Example}
\newtheorem{defn}[thm]{Definition}

\theoremstyle{remark}
\newtheorem{rem}[thm]{Remark}
\numberwithin{equation}{section}

\newcommand{\BVk}{O_{\mathrm{BV},\mathbf{k}}}

\newcommand{\hoco}{\mathrm{hocolim}}

\newcommand{\equiQ}{{\textsf{Q}}}
\newcommand{\shift}{[1]}

\begin{document}

\title[]{ Trace map on chiral Weyl algebras}%
 \author{Zhengping Gui}

  \address{
Z. Gui: International Centre for Theoretical Physics, Trieste, Italy;
}
\email{zgui@ictp.it}

\thanks{}%

\begin{abstract}
We construct a trace map on the chiral homology of chiral Weyl algebra for any smooth Riemann surface. Our trace map can be viewed as a chiral version of the deformed HKR quasi-isomorphism. This also provides a mathematical rigorous construction of correlation function for symplectic bosons in physics. We calculate some examples of trace maps with one insertion and find they are closely related to the variation of analytic torsion for holomorphic bundles on Riemann surfaces.
\end{abstract}
\maketitle
\tableofcontents


\section{Introduction}
The Witten genus, originally introduced by Witten \cite{witten1987elliptic} during his investigation of two-dimensional nonlinear $\sigma$-models, has had a significant impact on various fields of mathematics (see for instance \cite{hirzebruch1992manifolds,segal1987elliptic}). In the context of this type of two-dimensional quantum field theory, the partition function is constructed through a path integral over the mapping space $\mathrm{Map}(E_{\tau},M)$ (maps from an elliptic curve $E_{\tau}$ to a smooth manifold $M$) and subsequently leads to the emergence of the Witten genus, employing localization techniques. It is natural to study the higher genus analog of the partition function. This is studied in  \cite{alvarez2002beyond}, where researchers have uncovered new cobordism invariants in the form of sections of certain line bundles defined over the Teichmüller space.

There exists an alternative mathematical approach to the Witten genus given by Costello \cite{costello2010geometric,costellogeometric}, which is a geometric construction based on his work on (Batalin-Vilkovisky) BV quantization and renormalization of effective field theories. The quantum field theory model he used is the so-called $\beta\gamma$-system which is a two-dimensional conformal field theory.  By applying the machinery of BV quantization to this theory, he derives the Witten genus as the resulting partition function. Since two-dimensional chiral conformal field theories can be studied mathematically using the language of chiral algebras developed by Beilinson and Drinfeld \cite{beilinson2004chiral}. It becomes natural to explore the relationship between this construction of the Witten genus and the notion of chiral homology—derived invariants arising from chiral algebras on Riemann surfaces.  This is pursued in \cite{gui2021elliptic}, combining Costello's work and the formalism of chiral algebras.

In \cite{gui2021elliptic}, a trace map is constructed within the context of chiral homology for the $\beta\gamma-bc$-system (a special case of chiral Weyl algebras) on an elliptic curve by employing the BV formalism. This construction can be regarded as a mathematically rigorous definition of the physical partition function with multiple non-local operator insertions. In this paper, we extend this construction to chiral Weyl algebras on arbitrary smooth Riemann surfaces. Since it includes the case of the higher genus $\beta\gamma$-system with a linear target, we can view this as the first step for constructing the higher-genus analog of the Witten genus through the framework of chiral algebras.

Let $X$ be a smooth Riemann surface and $E$ be a super holomorphic Hermitian vector bundle on $X$ equipped with an even symplectic pairing $\langle-,-\rangle:E\otimes E\rightarrow \omega_X$. Here $\omega_X$ is the canonical bundle of $X$. Consider the two-dimensional conformal field theory with the chiral Lagrangian $\int_X\langle \phi,\bar\partial{\phi}\rangle, \phi\in \Omega^{0,\bullet}(X,E)$. The algebraic structure of the quantum observables in this theory is captured by a chiral algebra $\mathcal{A}_E$. The chiral algebra $\mathcal{A}_E$ is called the chiral Weyl algebra which can be thought of as a generalization of $\beta\gamma-bc$-system chiral algebra. The algebra $O_{\mathrm{BV}}$ of zero modes is the polynomial functions on the harmonic fields $\mathbb{H}(X,E)\subset \Omega^{0,\bullet}(X,E)$ which is a BV algebra with a BV operator $\Delta_{\mathrm{BV}}$. Its BV structure comes from the (-1)-symplectic pairing on $\mathbb{H}(X,E)$ given by $\int_X\langle-,-\rangle$. Our main result is the following theorem.

\begin{thm}[=Theorem \ref{MainTheorem510} ]\label{MainThmIntro}
    There is a quasi-isomorphism
    $$
    \mathrm{Tr}_{\mathcal{A}_E}:(\tilde{C}^{\mathrm{ch}}(X,\mathcal{A}_E)_{\mathcal{Q}},d^{\mathrm{ch}}_{\mathcal{A}_E})\rightarrow (O_{\mathrm{BV}},-\Delta_{\mathrm{BV}}).
    $$
    Here $(\tilde{C}^{\mathrm{ch}}(X,\mathcal{A}_E)_{\mathcal{Q}},d^{\mathrm{ch}}_{\mathcal{A}_E})$ is a chiral chain complex constructed by Beilinson and Drinfeld \cite{beilinson2004chiral}. The chiral homology of $\mathcal{A}_E$ is defined to the homology groups of this complex.
\end{thm}

When $X$ is an elliptic curve and E is a trivial vector bundle, the theorem mentioned above has been established in \cite{gui2021elliptic}. The main difficulty in carrying out this construction to general Riemann surfaces is that the chiral algebra becomes a nontrivial vector bundle with a complicated coordinate change formula. The Feynman diagram needs to be constructed with more care. Our construction is based on the relation between Lie* algebra and its chiral envelope. This construction has the potential to be extended to higher-dimensional chiral algebras introduced in \cite{francis2012chiral}.

As suggested in \cite{gui2021elliptic}, at least from the algebraic index theory point of view, the chiral chain complex of a chiral algebra plays a role similar to the Hochschild chain complex of an associative algebra. Parallel to the fact that the usual Weyl algebra arises as a deformed commutative polynomial algebra, the chiral Weyl algebra $\mathcal{A}_E$ is a deformation of the commutative chiral algebra $\mathrm{Sym}\  E_{\mathcal{D}}$ where $ E_{\mathcal{D}}=E\otimes_{\mathcal{O}_X}\mathcal{D}_X$. Our trace map in Theorem \ref{MainThmIntro} can be viewed as a deformed HKR quasi-isomorphism. Let us explain this in a bit more detail. In \cite[pp342, Section 4.6.2]{beilinson2004chiral}, Beilinson and Drinfeld prove that there is an isomorphism in a certain derived category that closely resembles the HKR quasi-isomorphism
$$
\mathrm{HKR}_{\mathrm{ch}}:\mathrm{Sym}\left(R\Gamma_{\mathrm{DR}}(X,E_{\mathcal{D}})[1]\right)\xrightarrow{\sim} C^{\mathrm{ch}}\left(X,\mathrm{Sym}\ (E_{\mathcal{D}})\right).
$$
Here $R\Gamma_{\mathrm{DR}}(X,-)$ is a complex computing the de Rham cohomology of a right $\mathcal{D}_X$-module and $C^{\mathrm{ch}}(X,-)$ is a functorial chain complex that computes the chiral homology (for example, we can take $C^{\mathrm{ch}}(X,\mathcal{A}_E)$ to be $(\tilde{C}^{\mathrm{ch}}(X,\mathcal{A}_E)_{\mathcal{Q}},d^{\mathrm{ch}}_{\mathcal{A}_E})$). While the right-hand side can be viewed as a chiral version of the Hochschild complex, the left-hand side looks like the space differential forms on $(H^1(X,E))^{\vee}$ as $R^0\Gamma_{\mathrm{DR}}(X,E_{\mathcal{D}})=H^0(X,E)=\left(H^1(X,E)\right)^{\vee}=\left(R^1\Gamma_{\mathrm{DR}}(X,E_{\mathcal{D}})\right)^{\vee}$.  Beilinson and Drinfeld introduce a chiral PBW filtration on the chiral Weyl algebra $\mathcal{A}_E$ and show that $\mathrm{gr}.\mathcal{A}_E\simeq \mathrm{Sym}(E_{\mathcal{D}})$. The above quasi-isomorphism gets deformed to
$$
\mathrm{HKR}_{\mathrm{ch}}^{\mathbf{deformed}}:(\mathrm{Sym}\left(R\Gamma_{\mathrm{DR}}(X,E_{\mathcal{D}})[1]\right),-\Delta_{\mathrm{BV}})\xrightarrow{\sim} C^{\mathrm{ch}}\left(X,\mathcal{A}_E\right).
$$
The BV structure on the left-hand side comes from the (-1)-shifted symplectic pairing on $R\Gamma_{\mathrm{DR}}(X,E_{\mathcal{D}})$ defined by $\int_X\langle-,-\rangle:R\Gamma_{\mathrm{DR}}(X,E_{\mathcal{D}})\times R\Gamma_{\mathrm{DR}}(X,E_{\mathcal{D}})\rightarrow R\Gamma(X,\omega)\rightarrow \mathbb{C}$. Our trace map in Theorem \ref{MainThmIntro} can be seen as an explicit chain-level realization of this quasi-isomorphism. We leave the investigation of the general chiral HKR theorem for a future study.

Chiral Weyl algebra is also interesting on its own. In cases where it is purely even, it is known as symplectic bosons in physics literature ($\beta\gamma$-system is a special case of symplectic bosons). Since the symplectic chiral boson algebra is constructed from a holomorphic vector bundle. It is interesting to study the variation of the partition function (chiral homology) along the moduli space of bundles. It is suggested in the physics literature \cite{gaiotto2019twisted} that the chiral homology, in this context, will manifest as a quantization of Gaiotto's Lagrangian within the Hitchin moduli space (introduced in \cite{gaiotto2019twisted}, see \cite{ginzburg2018gaiotto,hitchin2017spinors} for more details). This quantization is effectively realized through the (twisted) $\mathcal{D}$-module structure inherent to the chiral homology. The subsequent result suggests that our trace map serves as a useful tool for a more explicit understanding of this quantization.

Using the explicit formula of the trace map in Theorem \ref{MainThmIntro}, we express the variation of the analytic torsion as a trace map with one operator insertion.

\begin{thm}[=Theorem \ref{InsertionCurrent}]
Let $\{E_s\}_{s\in \mathcal{M}}$ be a holomorphic family of vector bundles parameterized by a complex manifold $\mathcal{M}$ with $H^0(X,E_s)=H^1(X,E_s)=0,\forall s\in \mathcal{M}$ and $\dim \mathcal{M}=\dim H^1(X,\mathrm{End}(E))$. Choose a local coordinate $\{s_{\mathbf{e}_i}\}_{i=1,\dots,\dim H^1(X,\mathrm{End}(E))}$ of $\mathcal{M}$ label by a basis $\{\mathbf{e}_i\}_{i=1,\dots,\dim H^1(X,\mathrm{End}(E))}$ of $H^1(X,\mathrm{End}(E))$. Then the variation of the analytic torsion of $E_s$ is equal to the trace map of a current $J_{\nu}$, that is, we have
$$
\frac{d T(E_s)}{d s_{[\nu]}}=\mathrm{Tr}_{\mathcal{A}_{E_s}}(J_{\nu}).
$$
The notation $T(E_s)$ stands for analytic torsion of $E_s$.
\end{thm}
Here the current $J_{\nu}\in \Omega^{0,1}(X,\mathcal{A})$ is the coupling of the affine Kac-Moody current with a specific form $\nu\in \Omega^{0,1}(X,\mathrm{End}(E))$ that describing the deformation of the bundle along the direction $\frac{\partial}{\partial s_{[\nu]}}$. Furthermore, the trace map is now valued in complex numbers. This is because we assume that $H^1(X,E)=H^0(X,E)=0$ which implies that $O_{\mathrm{BV}}\simeq \mathbb{C}$.

We discuss some related works and future directions. This work is motivated and influenced by \cite{cattaneo2000path,costello2010geometric,costellogeometric,kontsevich2003deformation,li2023vertex}. Although we study the different quantum field theory model from \cite{cattaneo2000path,kontsevich2003deformation}, we follow closely their method and philosophy. Namely, one can study quantum field theory mathematically using the BV formalism, and various intricate algebraic structures are encoded in quantum consistency conditions in this formalism, that is, they are governed by BV \textit{quantum master equation} (QME). Our primary motivation is to understand Costello's construction \cite{costello2010geometric,costellogeometric} better and explore the meaning of the higher genus Witten genera. Another motivation is to have a more conceptual understanding of the work \cite{li2023vertex} which connects the BV quantum master equation and BRST reduction of vertex algebras (on the complex plane or an elliptic curve). This is partly achieved in \cite{gui2021elliptic}, where the formulation of the above connection as well as its proof are compactly reformulated in terms of chiral algebras. It would be interesting to extend this story to higher genus Riemann surfaces using the results in the present paper.     Both works \cite{costellogeometric,li2023vertex} have very detailed discussions on the renormalization of Feynman integrals. As discovered in \cite{gui2021elliptic}, the renormalization procedure of two-dimensional chiral quantum field theories can be greatly simplified using the trace map on the unit chiral chain complex. We use the same treatment in the present paper. It is interesting to use our results and techniques in \cite{li2023vertex} to study the holomorphic bosonic string theory proposed by \cite{gwilliam2018holomorphic}, as this theory can be described as BRST reduction of a chiral Weyl algebra on a higher genus curve.

\section{Conventions}\label{Conventions}

\begin{itemize}
\item In this paper, we denote by $X$  a smooth complex algebraic curve and $\omega_X$ the canonical sheaf. And $X^I$  stands for the Cartesian product
    $$
    X^I=\underbrace{X\times\cdots\times X}_{I}
    $$ of $X$ , where $I$ is a finite index set $I$. We also denote by $\Delta_{I}$ the big diagonal in $X^I$
    $$
    \Delta_I=\bigcup_{i,j\in I}\Delta_{ij}, \quad \text{where}\quad \Delta_{ij}=\{(\dots,z_{k},\dots)\in X^I| k\in I, z_i=z_j\}.
    $$

  \item   We will work in the analytic category in this paper. For $\mathcal{D}_X$-modules (resp. $\mathcal{O}_X$-modules) on $X$, we mean analytic $\mathcal{D}_{X^{\mathrm{an}}}$-modules (resp. $\mathcal{O}_{X^{\mathrm{an}}}$-modules) on the analytification $X^{\mathrm{an}}$ and we omit the superscript $"\mathrm{an}"$ in this paper.

  \item Following the convention in \cite[pp.280]{beilinson2004chiral}, for any quasi-coherent sheave $F$ on a smooth complex variety $X$, we denote $F_{\mathcal{Q}}$ to be $F\otimes_{\mathcal{O}_X}\mathcal{Q}_X$. Here $\mathcal{Q}_X$ is the Dolbeault complex
$$
\mathcal{Q}_{X^I}(U):\Omega^{0,0}(U)\xrightarrow{\bar{\partial}}\Omega^{0,1}(U)\xrightarrow{\bar{\partial}}\cdots.
$$
The Dolbeault complex $\mathcal{Q}_X$ can be also viewed as a complex of non-quasicoherent left $\mathcal{D}_X$-modules, since the holomorphic differential $\partial$ is an integrable connection and commutes with the Dolbeault differential $\bar{\partial}$.

\item Following the convention in \cite{beilinson2004chiral}, for $f: X\rightarrow Y$, we use $f_*$ (resp. $f^*$) for the push forward (resp. pull back) functor of $\mathcal{D}$-modules and $f_{\bullet}$ (resp. $f^{\bullet}$) for the sheaf theoretical push forward (resp. pull back). In this notation, we define the transfer bimodule by $\mathcal{D}_{X\rightarrow Y}:=\mathcal{O}_X\otimes_{f^{\bullet}\mathcal{O}_Y}f^{\bullet}\mathcal{D}_Y$

\item We denote the de Rham complex (usually called the Spencer complex in the context of $\mathcal{D}$-module, here we follow the convention in \cite{beilinson2004chiral}) of a $\mathcal{D}_X$-module by $\mathrm{DR}(M)$,  here
$$
\mathrm{DR}^i(M):=M\otimes_{\mathcal{O}_X}\wedge^{-i}\Theta_X,
$$
where $\Theta_X$ is the tangent sheaf. The differential $d_{\mathrm{DR}}$ is given by
\begin{align*}
  m\otimes\xi_1\wedge\cdots\wedge \xi_k &\xrightarrow{d_{\mathrm{DR}}}\sum_{i=1}^{k}(-1)^{i-1}m\xi_i\otimes \xi_1\wedge \cdots \wedge \hat{\xi}_i\wedge\cdots \wedge_k  \\
   & +\sum_{i<j}(-1)^{i+j}m\otimes [\xi_i,\xi_j]\wedge\xi_1\wedge\cdots \wedge\hat{\xi}_i\wedge\cdots \wedge\hat{\xi}_j\wedge \xi_k.
\end{align*}
\item For a right $\mathcal{D}_X$-module, set $h(M):=M\otimes_{\mathcal{D}_X}\mathcal{O}_X=M/M\Theta_X$.
\item  For abbreviation, we denote the tensor product $E\otimes_{\mathcal{O}_X} \omega_X^{\alpha}$ of a holomorphic vector bundle $X$ and some power of the canonical line bundle $\omega_X$ by $E_{\omega^{\alpha}}$. Here we allow $\alpha$ to be a rational number and we choose a $r$-spin structure if $\alpha\in \frac{1}{r}\mathbb{Z}$. We also write $E_{\mathcal{D}}$ for $E\otimes_{\mathcal{O}_X} \mathcal{D}_X$.

\item  We use $M^l=M_{\omega^{-1}}$ (resp. $N^r=N_{\omega}$) be the left-$\mathcal{D}_X$-module corresponding to the right-$\mathcal{D}_X$-module $M$ (resp. left $\mathcal{D}_X$-module $N$).

\item Then tensor product $M_1\otimes^{r} M_2$ of two right $\mathcal{D}_X$-modules $M_1$ and $M_2$ is given by $(M_1^l\otimes_{\mathcal{O}_X} M_2^l)^r$. For simplicity of notation, we write $\mathrm{Sym}\ M$ (resp.$\mathrm{Sym}\ N$) for the symmetric tensor power $\mathrm{Sym}^{\otimes^r} M$ (resp.$\mathrm{Sym}^{\otimes_{\mathcal{O}_X}} N$) for right $\mathcal{D}_X$-module $M$ (resp. left $\mathcal{D}_X$-module $N$).
\end{itemize}

\noindent\textbf{Acknowledgment.}
The author would like to thank Kevin Costello, Davide Gaiotto, Si Li and Kai Xu for helpful communications. Part of this work was done while the author was visiting the Perimeter Institute for Theoretical Physics in March 2023. The author thanks the institute for its hospitality and provision of the excellent working environment.
\section{Chiral Weyl algebras}
In this section, we review the notion of the chiral algebra introduced by Beilinson and Drinfeld \cite{beilinson2004chiral}. We first recall the definitions of chiral algebras and Lie* algebras. Given a Lie* algebra $\mathcal{L}$, one can construct a chiral algebra called the chiral enveloping algebra $U(\mathcal{L})$ satisfying a universal property. The chiral Weyl algebra can be defined as a twisted
version of the chiral enveloping algebra for a specific Lie* algebra. Given a vertex algebra equipped with extra structures, Frenkel and Ben-Zvi \cite{frenkel2004vertex} construct a chiral algebra on an arbitrary smooth Riemann surface. We explain the connection between these two constructions. Indeed, they are isomorphic in most cases that we are interested in.

\subsection{Chiral algebras}

The standard reference is \cite{beilinson2004chiral,gaitsgory1998notes}, here we follow the presentation of \cite{gaitsgory1998notes}. Let $X$ be a smooth complex curve. Let $I$ be a finite set and $\Delta_I=\cup_{i,j\in I}\Delta_{ij}$ be the big diagonal, where $\Delta_{ij}=\{(\dots,z_{k},\dots)\in X^I| k\in I, z_i=z_j\}.$ For an $\mathcal{O}_{X^I}$-module $\mathcal{F}$, we define the $\mathcal{O}_{X^I}$-module $\mathcal{F}(*\Delta_I)$ to be
$$
\mathcal{F}(*\Delta_I)=\cup_{n\geq 0}\mathcal{F}(n\cdot \Delta_I)
$$
which is the sheaf of sections of $\mathcal{F}$ but with arbitrary order of poles along $\Delta_I$.

For $\mathcal{D}$-modules, we always mean right $\mathcal{D}$-modules unless specified otherwise.

\begin{defn}\label{chiralDefn}
Let $\mathcal{A}$ be a $\mathcal{D}_X$-module. A chiral algebra structure on $\mathcal{A}$ is a $\mathcal{D}_{X^2}$-module map:
$$
\mu:(\mathcal{A}\boxtimes\mathcal{A})(*\Delta_{\{1,2\}})\rightarrow \Delta^{X\rightarrow X^2}_*(\mathcal{A}),\quad \ \Delta^{X\rightarrow X^2}:X\rightarrow X^2\ \text{the diagonal embedding}
$$
that satisfies the following two conditions:
\begin{itemize}
  \item Antisymmetry:

  If $f(z_1,z_2)\cdot a\boxtimes b$ is a local section of $\mathcal{A}\boxtimes\mathcal{A}(*\Delta_{\{1,2\}})$, then
  \begin{equation}\label{Antisymmetry}
  \mu(f(z_1,z_2)\cdot a\boxtimes b)=-\sigma_{1,2}\mu(f(z_2,z_1)\cdot b\boxtimes a),
  \end{equation}
  where $\sigma_{1,2}$ acts on $\Delta^{X\rightarrow X^2}_*\mathcal{A}=(\Delta^{X\rightarrow X^2})_{\bullet}\left(\mathcal{A}\otimes_{{\mathcal{D}_{X}}}\mathcal{D}_{X\rightarrow X^2}\right)$ by permuting two factors of $X^2$.

  \item Jacobi identity:

  If $a\boxtimes b\boxtimes c\cdot f(z_1,z_2,z_3)$ is a local section of $\mathcal{A}^{\boxtimes 3}(*\Delta_{\{1,2,3\}})$, then
  \begin{align*}
      \mu(\mu(f(z_1,z_2,z_3)\cdot &a\boxtimes b)\boxtimes c)+\sigma_{1,2,3}\mu(\mu(f(z_2,z_3,z_1)\cdot b\boxtimes c )\boxtimes a)+  \\
     & \sigma_{1,2,3}^{-1}\mu(\mu(f(z_3,z_1,z_2)\cdot c\boxtimes a )\boxtimes b)=0,
  \end{align*}
  here $\sigma_{1,2,3}$ denotes the cyclic permutation action on $\Delta^{X\rightarrow X^3}_*\mathcal{A}=(\Delta^{X\rightarrow X^2})_{\bullet}\left(\mathcal{A}\otimes_{\mathcal{D}_X}\mathcal{D}_{X\rightarrow X^3}\right)$, $\Delta^{X\rightarrow X^3}:X\rightarrow X^3\ \text{is the diagonal embedding}$.
\end{itemize}
\end{defn}

\begin{exa}
  The canonical sheaf $\omega_X$ is a chiral algebra and the chiral operation is given by the residual operation. It is also called unit chiral algebra. Throughout this paper, we denote the chiral operation on $\omega_X$ by $\mu_{\omega}$.

\end{exa}
\begin{exa}
    Let $E$ be a holomorphic vector bundle on $X$. Consider the $\mathcal{D}_X$-module $\mathrm{Sym}\ E_{\mathcal{D}} $ (see Section \ref{Conventions} for conventions). One can define a chiral product on $B=\mathrm{Sym}\ E_{\mathcal{D}}$ by the following sequence
    $$
B\boxtimes B(*\Delta)=\omega_{X^2}(*\Delta)\otimes (B^l\boxtimes B^l)\xrightarrow{\mu_{\omega}\otimes \mathrm{id}}\Delta_*\omega_X\otimes_{\mathcal{O}_{X^2}}(B^l\boxtimes B^l)\rightarrow\Delta_*B
    $$
the last arrow we use the commutative product $\mathrm{Sym}(E_{\mathcal{D}})^l\otimes \mathrm{Sym}(E_{\mathcal{D}})^l\rightarrow \mathrm{Sym}(E_{\mathcal{D}})^l$. We obtain a chiral algebra structure on $B=\mathrm{Sym}\ E_{\mathcal{D}}$. From now on, we denote the chiral operation of $B=\mathrm{Sym}\ E_{\mathcal{D}}$ by $\mu_{\mathrm{Sym}}$.
\end{exa}

\subsection{Lie* algebras and chiral enveloping algebras}
In this subsection, we review the notion of Lie* algebra and chiral enveloping algebra, see \cite{beilinson2004chiral,gaitsgory1998notes} for more details.

A Lie* algebra on $X$ is a $\mathcal{D}_{X^2}$-module with a map
$$
\mu_{\mathrm{Lie}}:\mathcal{L}\boxtimes \mathcal{L}\rightarrow\Delta_*\mathcal{L}
$$
which is antisymmetric and satisfies the Jacobi identity that is similar to the one in chiral algebras. More precisely, one deletes all "$(*\Delta)$" in the definition of chiral algebras.

The forgetful functor from chiral algebras to Lie* algebras admits a left adjoint \cite{beilinson2004chiral}
\begin{equation}\label{UniversalProp}
\mathrm{Hom}_{\mathrm{Lie}*}(\mathcal{L},\mathcal{A}^{\mathrm{Lie}})\simeq \mathrm{Hom}_{\mathrm{ch}}(\mathscr{U}(\mathcal{L}),\mathcal{A}).
\end{equation}
Here $\mathrm{Hom}_{\mathrm{ch}}(-,-)$ (resp.$\mathrm{Hom}_{\mathrm{Lie}*}(-,-)$)is the space of homomorphism between chiral algebras (resp. Lie* algebra)

Let us describe the construction of the chiral envelope.  Since $\mathscr{U}(\mathcal{L})$ is a local object, in the rest of this subsection we assume that $X$ is affine.  Denote by $p_I$ the projection $X\times X^I\rightarrow X^I$, and let $j_I:V_I\hookrightarrow X\times X^I$ be the open subset of those $(x,(x_i)_{i\in I})$ that $x\in X-\{x_i\}_{i\in I}$. Consider
$$
\mathcal{L}^{\natural}_{X^I}:=H^0(p_Ij_I)_{\bullet}j^*_I(\mathrm{DR}(\mathcal{L})\boxtimes\mathcal{O}_{X^I}).
$$
This is an $\mathcal{O}_{X^I}$-module and its fiber at $(x_i)\in X^I$ equals $\Gamma(X-\{x_i\}_{i\in I},h(\mathcal{L}))$. Here $h(\mathcal{L})$ is the middle de Rham cohomology sheaf (see Section \ref{Conventions} for conventions).

 We will mainly using $\mathcal{L}^{\natural}_{X}$ and $\mathcal{L}^{\natural}_{X^2}$. Denote
 $$
 \mathcal{L}^{\natural}_0:=\mathcal{L}^{\natural}_{X^{\empty}}=\Gamma(X,h(\mathcal{L})).
 $$

 Define
 $$
 \mathbf{U}_{X^I}(\mathcal{L}):=U(\mathcal{L}^{\natural}_{X^I})/U(\mathcal{L}^{\natural}_{X^I})\mathcal{L}^{\natural}_0.
 $$
Here $U(-)$ is the usual enveloping algebra. If we choose a local coordinate, we can present a local section of $\mathbf{U}_{X^I}(\mathcal{L})$ by the following expression
$$
g\cdot\left( f_1(t;\{x_i\})l_1 \cdots f_n(t;\{x_i\})l_n\cdot |0\rangle\right).
$$
Here $0|\rangle=\bar{1}\in \mathbf{U}_{X^I}(\mathcal{L})$ is the vacuum state, or equivalently the image of the identity. We use $l_i$ for local sections of $h(\mathcal{L})$, $g$ for a local holomorphic function on $X^I$ and $f_i(t;\{x_i\})$ for local functions that have poles only at $\{x_i\}$.

Considering the following diagram.
$$
\begin{tikzcd}
p_1^*\mathcal{L}^{\natural}_{X}\times p^*_2 \mathcal{L}^{\natural}_{X} \arrow[r,"c"] & \mathcal{L}^{\natural}_{X^2}\times \mathcal{L}^{\natural}_{X^2}                 & \mathcal{L}^{\natural}_{X^2} \arrow[l,"\iota"']                               \\
(\mathcal{L}^{\natural}_0\otimes\mathcal{O}_{X^2})^2 \arrow[r] \arrow[u]         & p_2^*\mathcal{L}^{\natural}_{X}\times p_1^*\mathcal{L}^{\natural}_{X} \arrow[u] & \mathcal{L}^{\natural}_0\otimes\mathcal{O}_{X^2} \arrow[u] \arrow[l]
\end{tikzcd}
$$
The map $c$ is the product of the natural maps $p_i^*\mathcal{L}^{\natural}_{X}\rightarrow \mathcal{L}^{\natural}_{X^2}.$ The map $\iota$ is the diagonal map.  The maps in the second line are defined in the same way.

Passing to the vacuum modules, we get the following map
 $$
 \mathbf{U}_{X}(\mathcal{L})\boxtimes  \mathbf{U}_{X}(\mathcal{L})\xrightarrow{c} U(\mathcal{L}^{\natural}_{X^2})/U(\mathcal{L}^{\natural}_{X^2})p^*_2\mathcal{L}^{\natural}_X\otimes U(\mathcal{L}^{\natural}_{X^2})/U(\mathcal{L}^{\natural}_{X^2})p^*_1\mathcal{L}^{\natural}_X\xleftarrow{\iota} \mathbf{U}_{X^2}(\mathcal{L})$$

Here we explain the vacuum modules $U(\mathcal{L}^{\natural}_{X^2})/U(\mathcal{L}^{\natural}_{X^2})p^*_i\mathcal{L}^{\natural}_X, i=1,2$. Let us denote the vacuum states by $|0\rangle_{z_j}\in U(\mathcal{L}^{\natural}_{X^2})/U(\mathcal{L}^{\natural}_{X^2})p^*_i\mathcal{L}^{\natural}_X, \{i,j\}=\{1,2\}$. The state $|0\rangle_{z_j}$ is annihilated by $p_i^*\mathcal{L}^{\natural}_X$ which consists sections from $\mathcal{L}^{\natural}_{X^2}$ but only have poles at $z_i$.

By \cite[pp 217, Section 3.7.7.]{beilinson2004chiral}, over $X\times X-\Delta$ both arrows are isomorphisms and $\mathscr{U}(\mathcal{L}):=\mathbf{U}_X^r(\mathcal{L})$ is a chiral algebra.  We can describe the chiral product of $\mathscr{U}(\mathcal{L})=\mathbf{U}_X^r(\mathcal{L})$ explicitly as follows. Given a local section
$$
a\cdot dz_1\boxtimes dz_2\in \mathbf{U}_X^r(\mathcal{L})\boxtimes \mathbf{U}_X^r(\mathcal{L})(*\Delta).
$$
We can find a sufficiently large integer $N>>0$ such that
$$
c\left((z_1-z_2)^N\cdot a\right)\in \mathrm{Im}(\iota).
$$
The chiral product is given by the following
\begin{equation}\label{ChiralProd}
\mu(a\cdot dz_1\boxtimes dz_2):=\mu_{\omega}(\frac{dz_1\boxtimes dz_2}{(z_1-z_2)^N})\cdot \left(\iota^{-1}\circ  c((z_1-z_2)^N\cdot a)\right).
\end{equation}
Here we use the isomorphism
$$
\Delta_*\omega_X\otimes_{\mathcal{O}_{X^2}} \mathbf{U}_{X^2}(\mathcal{L})=\Delta_*\mathbf{U}^r_{X}(\mathcal{L}).
$$

\begin{defn}
  Let $\mathcal{L}$ be a Lie* algebra. Then the $\mathcal{D}_X$-module $\mathscr{U}(\mathcal{L}):=\mathbf{U}^r_X(\mathcal{L})$  is a chiral algebra with the above chiral product (\ref{ChiralProd}).
\end{defn}
 The chiral algebra $\mathscr{U}(\mathcal{L})$ satisfies the universal property (\ref{UniversalProp}) and carries a canonical PBW filtration $\mathscr{U}(\mathcal{L})\simeq \cup_n\mathscr{U}(\mathcal{L})_n$. We have $\mathscr{U}(\mathcal{L})_0\simeq \omega_X$ and $\mathscr{U}(\mathcal{L})_1\simeq \mathscr{U}(\mathcal{L})_0\oplus \mathcal{L}$.

We have the following twisted version of the chiral envelope.
\begin{defn}

Let $\mathcal{L}^{\flat}$ be a central extension of $\mathcal{L}$ by $\omega_X$. Then the $\flat$-twisted  chiral enveloping algebra of $\mathcal{L}$ is the quotient  $\mathscr{U}(\mathcal{L})^{\flat}$ of $\mathscr{U}(\mathcal{L}^{\flat})$ modulo the ideal generated by $1-1^{\flat}$. Here $1=\omega_X\in  \mathscr{U}(\mathcal{L}^{\flat})$ is the unit the chiral envelope and $1^{\flat}=\omega_X\subset \mathcal{L}^{\flat}$.
\end{defn}

Now we are ready to define the main object in this paper.
\begin{defn}(Chiral Weyl algebras \cite[pp228,Section 3.8.1]{beilinson2004chiral})\label{DefnChiralWeyl}
    Let $E$ be a locally free sheaf on $X$ equipped with a symplectic pairing $E\otimes_{\mathcal{O}_X} E\rightarrow\omega_X$. Then the chiral Weyl algebra generated by $E$ is defined to be the twisted chiral enveloping algebra $\mathscr{U}(\mathcal{L})^{\flat}$. Here the abelian Lie* algebra $\mathcal{L}$ is $E_{\mathcal{D}}=E\otimes_{\mathcal{O}_X}\mathcal{D}_X$ and $\mathcal{L}^{\flat}=E_{\mathcal{D}}\oplus \omega_X$ is the central extension of $\mathcal{L}$ using the symplectic pairing on $E$.
\end{defn}

To illustrate the above construction, we assume $X=\mathbb{C}$ and $E=X\times \mathbb{E}$ is a trivial bundle for simplicity (here $\mathbb{E}$ is a symplectic vector space). We choose a basis $\{e_i\}_{i=1,\dots,m=\dim(\mathbb{E})}$ of $\mathbb{E}$ and write the symplectic pairing between $e_i, e_j$ by $\omega_{ij}$. Then a section of $\mathscr{U}(\mathcal{L})^{\flat}$ can be written as
$$
f(z)\cdot \frac{ e_{i_1}}{(t-z)^{n_1+1}}\cdots \frac{ e_{i_k}}{(t-z)^{n_k+1}}|0\rangle dz
$$
where $f(z)$ is a holomorphic function over $\mathbb{C}$. As an $\mathcal{O}_X$-module, $\mathscr{U}(\mathcal{L})^{\flat}=\mathbf{U}_X(\mathcal{L})\otimes_{\mathcal{O}_X}\omega_X$ is isomorphic to $\mathrm{Sym}\ \mathcal{L}=\mathrm{Sym}(E_{\mathcal{D}}\otimes_{\mathcal{O}_X}\omega^{-1}_X)\otimes_{\mathcal{O}_X}\omega_X$ under the following identification
$$
\frac{n_1!\cdot e_{i_1}}{(t-z)^{n_1+1}}\cdots \frac{n_k!\cdot e_{i_k}}{(t-z)^{n_k+1}}|0\rangle\cdot dz\leftrightarrow    ( e_{i_1}\otimes \partial_z^{n_1}\otimes dz^{-1}\cdots e_{i_k}\otimes \partial_z^{n_k}\otimes dz^{-1})\cdot dz.
$$

For example, if we want to compute
$$
\mu(\frac{e_i\boxtimes e_j}{z_1-z_2}).
$$
First, we note that
$$
\iota(\frac{e_i}{t-z_1}\cdot \frac{e_j}{t-z_2}|0\rangle)=\frac{e_i}{t-z_1}\cdot \frac{e_j}{t-z_2}|0\rangle_{z_1}\boxtimes 1+\frac{e_i}{t-z_1}|0\rangle_{z_1}\boxtimes \frac{e_j}{t-z_2}|0\rangle_{z_2}
$$

$$
+1\boxtimes\frac{e_i}{t-z_1}\cdot \frac{e_j}{t-z_2}|0\rangle_{z_2}
$$

$$
=0+\frac{e_i}{t-z_1}|0\rangle_{z_1}\boxtimes \frac{e_j}{t-z_2}|0\rangle_{z_2}-\frac{\omega_{ij}}{z_1-z_2}.
$$

Here $\omega_{ij}$ is the symplectic pairing between $e_i$ and $e_j$. Thus we have
$$
\iota \left((z_1-z_2) \cdot \frac{e_i}{t-z_1}\cdot \frac{e_j}{t-z_2}|0\rangle+\omega_{ij}\right)=c((z_1-z_2)\cdot e_i\boxtimes e_j).
$$

Finally, we have
\begin{align*}
\mu(\frac{e_idz_1\boxtimes e_jdz_2}{z_1-z_2})&=\mu_{\omega}(\frac{dz_1\boxtimes dz_2}{(z_1-z_2)^2})\cdot ((z_1-z_2)\cdot e_idz_1^{-1}\boxtimes e_jdz_2^{-1})\\
&=\mu_{\omega}(\frac{dz_1\boxtimes dz_2}{(z_1-z_2)^2})\cdot \left((z_1-z_2) \cdot \frac{e_i}{t-z_1}\cdot \frac{e_j}{t-z_2}|0\rangle+\omega_{ij}\right)\\
&=\frac{e_i}{t-z}\cdot \frac{e_j}{t-z}|0\rangle+\mu_{\omega}(\frac{\omega_{ij}dz_1\boxtimes dz_2}{(z_1-z_2)^2}).
\end{align*}

\begin{rem}\label{PBW}
    We have the chiral PBW theorem by \cite[pp224, Section 3.7.20]{beilinson2004chiral}
    $$
    \mathrm{gr}(\mathscr{U}(\mathcal{L})^{\flat})\simeq \mathrm{Sym}\ \mathcal {L}.
    $$
    More explicitly, if we choose a local coordinate, we have local isomorphism as $\mathcal{O}$-module
$$
\frac{n_1!\cdot e_{i_1}}{(t-z)^{n_1+1}}\cdots \frac{n_k!\cdot e_{i_k}}{(t-z)^{n_k+1}}|0\rangle\cdot dz\leftrightarrow    ( e_{i_1}\otimes \partial_z^{n_1}\otimes dz^{-1}\cdots e_{i_k}\otimes \partial_z^{n_k}\otimes dz^{-1})\cdot dz.
$$

The right-hand side operators are commuting because the commutator is
$$
\frac{\langle e_{i_p},e_{i_q}\rangle}{(t-z)^{n_p+n_q+2}}
$$
which vanishes in the de Rham cohomology $\Gamma(X,h(\omega_X))$. Lower-order terms will appear if we do the local coordinate change.
\end{rem}
 \begin{rem}
    In fact, $\{\mathbf{U}_{X^I}(\mathcal{L})\}_{I\in\mathcal{S}}$ forms a factorization algebra in the sense of Beilinson and Drinfeld \cite{beilinson2004chiral}. One obtains the chiral algebra structure on $\mathbf{U}_{X}(\mathcal{L})$ from the factorization structure. The formalism of factorization algebra is closely related to the usual point-splitting method in physics. For example, the normal order product $:e_ie_j:$ can be written as the pull back $\Delta^*(\frac{e_i}{t-z_1}\otimes \frac{e_j}{t-z_2}|0\rangle)$. While under the isomorphism $c^{-1}\circ \iota$ the element $\frac{e_i}{t-z_1}\otimes \frac{e_j}{t-z_2}|0\rangle\in \mathbf{U}_{X^2}(\mathcal{L})$  corresponds to $$
\frac{e_i}{t-z_1}|0\rangle\boxtimes\frac{e_j}{t-z_2}|0\rangle
     -\frac{\omega_{ij}|0\rangle\boxtimes |0\rangle}{z_1-z_2}\in j_*j^*\mathbf{U}_X(\mathcal{L})\boxtimes \mathbf{U}_X(\mathcal{L}).$$
     This is exactly how physicist define the normal ordered product using point-splitting
     $$
     :e_i(w)e_j(w):=\lim_{z\rightarrow w}\left(e_i(z)e_j(w)-\frac{\omega_{ij}}{z-w}\right).
     $$
  \end{rem}

\subsection{Comparison with the vertex algebra bundle construction}
In this subsection, we collect some connections between chiral algebra and vertex algebra. We show that some chiral envelopes are exactly corresponding to vertex algebra bundles associated to some familiar vertex algebras. A brief introduction to the general construction of vertex algebra bundle (this construction is due to Frenkel and Ben-Zvi \cite{frenkel2004vertex}) can be found in the Appendix \ref{VOAappendix}. Throughout this section, $X$ stands for a compact smooth Riemann surface.

We assume that our holomorphic vector bundle   $E$ has the following form
$$
E={E}_{\frac{1}{2}}\omega^{\frac{1}{2}}\bigoplus\oplus_{\alpha\in \mathbb{Q},\alpha<\frac{1}{2}}{E}_{\alpha}\omega_X^{1-\alpha}\oplus ({E}_{\alpha})^{\vee}\omega_X^{\alpha}
$$

We explain the notations in the above expression.  Holomorphic vector bundles ${E}_{\alpha},\alpha\in \mathbb{Q}, \alpha<\frac{1}{2}$ are given by the twisting construction
$$
E_{\alpha}\oplus ({E}_{\alpha})^{\vee}=\mathcal{P}_{\alpha}\times_{\mathrm{GL}_{n_{\alpha}}} \left(\mathbb{E}_{\alpha}\oplus (\mathbb{E}_{\alpha})^{\vee}\right).
$$

Here $\mathcal{P}_{\alpha}$ is holomorphic $\mathrm{GL}_{n_{\alpha}}$-principal bundle over $X$ and vectors spaces
$$
\mathbb{E}_{\alpha}=\mathrm{Span}\{e_{\alpha,j}\}_{j=1,\dots,n_{\alpha}},\quad
\mathbb{E}_{\alpha}^{\vee}=\mathrm{Span}\{e_{\alpha}^j\}_{j=1,\dots,n_{\alpha}},$$
are viewed as natural $\mathrm{GL}_{n_{\alpha}}$-representations. For the case $\alpha=\frac{1}{2}$, we have
$$
E_{\frac{1}{2}}=\mathcal{P}_{\frac{1}{2}}\times_{\mathrm{Sp}_{2n_{\frac{1}{2}}}} \mathbb{E}_{\frac{1}{2}}
$$
where $\mathcal{P}_{\frac{1}{2}}$ is holomorphic $\mathrm{Sp}_{2n_{\frac{1}{2}}}$-principal bundle over $X$ and $\mathbb{E}_{\frac{1}{2}}$ is a symplectic vector space $(\mathbb{E}_{\frac{1}{2}},\omega)=\left(\mathrm{Span}\{e_{\frac{1}{2}}^{i}\}_{i=1,\dots,2n_{\frac{1}{2}}},\omega(e_{\frac{1}{2}}^{i},e_{\frac{1}{2}}^{j})=\omega_{ij}\right).$

For simplicity, we assume that $\{\alpha\in\mathbb{Q}|\mathbb{E}_{\alpha}\neq 0\}$ is a finite set. For each $a \in \mathbb{E}_{\alpha}\oplus \mathbb{E}_{\alpha}^{\vee}$, we associate a field (see Appendix \ref{VOAappendix} for a review of vertex algebras)

\begin{equation}\label{ConformalWeight}
a(z)=\sum_{r \in \mathbb{Z}-\alpha} a_{r} z^{-r-\alpha}.
\end{equation}
We define their singular part of OPE by
$$
a(z) b(w) \sim\frac{\langle a, b\rangle}{z-w}, \quad \forall a, b \in \mathbb{E}_{\alpha}\oplus \mathbb{E}_{\alpha}^{\vee},
$$

\begin{rem}
  In (\ref{ConformalWeight}), we use the conformal weight notation. If we use the VOA notation, we have
  $$
  a(z)=\sum_{n \in \mathbb{Z}} a_{(n)} z^{-n-1},
  $$
  that is, $a_{(n)}=a_{n+1-\alpha}.$
\end{rem}

The vertex algebra $\mathbf{V}_{\alpha}$ that realizes the above OPE relations is given by the Fock representation space. The vacuum vector satisfies
$$
a_{r}|0\rangle=0, \quad \forall a \in \mathbb{E}_{\alpha}\oplus \mathbb{E}_{\alpha}^{\vee}, r+\alpha>0,
$$
and $\mathbf{V}_{\alpha}$ is freely generated from the vacuum by the operators $\left\{a_{r}\right\}_{r+\alpha \leq 0}, a \in \mathbb{E}_{\alpha}\oplus \mathbb{E}_{\alpha}^{\vee} .$ For any $a \in \mathbb{E}_{\alpha}\oplus \mathbb{E}_{\alpha}^{\vee}$, $a(z)$ becomes a field acting naturally on the Fock space $\mathbf{V}_{\alpha}$ . The vertex algebra $\mathbf{V}_{\frac{1}{2}}$ can be constructed in the same way using the symplectic form on $\mathbb{E}_{\frac{1}{2}}$.

Define
$$
\mathbb{E}_{\alpha}^{\mathrm{diff}}:=  \mathbb{E}_{\alpha}[t^{-1}]t^{-1},
$$
then
$$
\mathbf{V}_{\alpha}\simeq\mathrm{Sym}(\mathbb{E}_{\alpha}^{\mathrm{diff}}\oplus (\mathbb{E}_{\alpha})^{\vee,\mathrm{diff}} )
$$
as vector spaces. For $\alpha=\frac{1}{2}$, we have
$$
\mathbf{V}_{\frac{1}{2}}\simeq \mathrm{Sym}(\mathbb{E}_{\frac{1}{2}}^{\mathrm{diff}}),\quad \mathbb{E}_{\frac{1}{2}}^{\mathrm{diff}}:=  \mathbb{E}_{\frac{1}{2}}[t^{-1}]t^{-1}.
$$

 For $\mathbf{V}_{\alpha}$ (resp.$\mathbf{V}_{\frac{1}{2}}$), we have $\widehat{\mathfrak{gl}}_{n_{\alpha}}$-structure (resp.$\widehat{\mathfrak{sp}}_{2n_{\frac{1}{2}}}$-structure). Here
$$
\widehat{\mathfrak{gl}}_{n_{\alpha}}:=\widehat{\mathfrak{h}}\oplus \widehat{\mathfrak{sl}}_{n_{\alpha}} .
$$
Then $\mathrm{GL}_{n_{\alpha}}(\mathcal{O})$ acts on $\mathbf{V}_{\alpha}$ through
$$
E^i_{j}t^{n}\mapsto (J^i_j)_n
$$
where $J^i_j=:e_{\alpha}^ie_{\alpha,j}:$ and $E^i_j\in\mathfrak{gl}_{n_{\alpha}}$ is the elementary matrix.

Furthermore, each $\mathbf{V}_{\alpha}$ has the standard conformal structure
$$
T_{\alpha}=\sum_{i=1}^{n_{\alpha}}(1-\alpha):\partial e_{\alpha,i}\cdot e_{\alpha}^i:-\alpha :e_{\alpha,i}\cdot\partial e^i_{\alpha}:,\quad T_{\frac{1}{2}}=\frac{1}{2}\sum_{i,j=1,\dots,2n_{\frac{1}{2}}}\omega_{ij}:e^i_{\frac{1}{2}}\partial e^j_{\frac{1}{2}}:.
$$

Now we proceed to define the vertex algebra bundle. Define
$$
\mathcal{P}:=\times^{\alpha}_X\mathcal{P}_{\alpha}
$$
which is a $\mathrm{G}=\mathrm{Sp}_{2n_{\frac{1}{2}}}\times(\times_{\alpha>\frac{1}{2}} \mathrm{GL}_{n_{\alpha}})$-principal bundle. Let $\mathcal{O}:=\mathbb{C}[[t]]$  and $\widehat{\mathcal{P}}$ be the corresponding $\mathrm{Aut}(\mathcal{O})\ltimes\mathrm{G}(\mathcal{O})$-bundle. We now can form a vertex algebra bundle as
$$
\mathcal{V}:=\widehat{\mathcal{P}}\times_{\mathrm{Aut}(\mathcal{O})\ltimes \mathrm{G}(\mathcal{O})}\mathbf{V}
$$
where $\mathbf{V}=\oplus_{\alpha}\mathbf{V}_{\alpha}$.

\begin{rem}
    Since here our vertex algebra $\mathbf{V}$ is not quasi-conformal because of non-integral gradation by $L_0$, the twisting expression is only a formal one. However, for $\frac{1}{r}\mathbf{N},r\in \mathbf{N}$ gradation, one can construct the vertex algebra bundle by choosing r-spin structures. This is discussed in Appendix \ref{VOAappendix}.
 \end{rem}

From \cite{frenkel2004vertex}, $\mathcal{V}$ is left $\mathcal{D}_X$-module and the corresponding right $\mathcal{D}_X$-module $\mathcal{V}^r$ has a chiral algebra structure. Then as expected, we have
$$
\mathscr{U}(\mathcal{L})^{\flat}\simeq \mathcal{V}^r.
$$
Here $\mathcal{L}=E_{\mathcal{D}}$. This will be proved in Appendix \ref{VAbundleChiral}.
\section{Chiral homology
and trace map}
Our goal in this section is to introduce the notion of chiral homology and give a brief summary of BV quantization. The construction of solutions to so-called quantum master equation (QME) lies at the heart of BV formalism. We will introduce the notion of the trace map on chiral homology within BV formalism which can be viewed as a solution to generalized QME.

\subsection{Chiral homology}
 We refer the reader to \cite{beilinson2004chiral} for a detailed account of the chiral homology theory.  In \cite{beilinson2004chiral}, Beilinson and Drinfeld construct various chiral chain complexes to compute the chiral homology. We use the one in \cite[pp 309, Section 4.2.15.]{beilinson2004chiral}.

 We begin by fixing some notations. For a finite index set $I$, we will denote by ${\textsf{Q}}(I)$ the set of equivalence relations on $I$. Note that $\equiQ(I)$ is a disjoint union $\sqcup_{n\geq 1}\equiQ(I,n)$, where $\equiQ(I,n)$ is the set of equivalence relations on $I$ such that the number of the equivalence classes is $n$. An element of $\equiQ(I,n)$ can be described by a pair $(T,\pi:I\twoheadrightarrow T)$, where $T$ is a finite set such that $|T|=n$ and $\pi$ is a surjection. We denote $X^I$ to be the product of $X_i$'s, where $X_i$ is a copy of $X$ and $i\in I$. As before, $X$ stands for a compact smooth Riemann surface.

 Given a surjection $\pi:I\twoheadrightarrow T$,  we have the corresponding diagonal embedding
 $$
 \Delta^{(\pi)}=\Delta^{(I/T)}:X^T\hookrightarrow X^I,
 $$
 such that $\mathrm{pr}_i\Delta^{(I/T)}=\mathrm{pr}_j\Delta^{(I/T)}$ if and only if $\pi(i)=\pi(j)$, here $\mathrm{pr}_i: X^I\rightarrow X$ is the i-th projection.

\begin{rem}
  With this notation, $\Delta^{X\rightarrow X^2}$ in Definition \ref{chiralDefn} is the same as $\Delta^{(\{1,2\}/\{\bullet\})}:X\rightarrow X^2$.
\end{rem}
 Let $\mathcal{A}$ be a graded chiral algebra. Now we define a $\mathbb{Z}$-graded $\mathcal{D}_{X^I}-$module $C(\mathcal{A})_{X^I}$ on $X^I$
$$
\boxed{C(\mathcal{A})^{\bullet}_{X^I}=\bigoplus_{T\in \equiQ(I)} \Delta^{(I/T)}_*\big((\mathcal{A}\shift)^{\boxtimes T}(*\Delta_T)\big).}
$$
The differential looks as follows. Its component
$$
d_{\mu, (T,T')}: \Delta_*^{(I/T)}\big((\mathcal{A}\shift)^{\boxtimes T}(*\Delta_T)\big)\rightarrow\Delta_*^{(I/T')}\big((\mathcal{A}\shift)^{\boxtimes T'}(*\Delta_{T'})\big)
$$
can be non-zero only for $T'\in \equiQ(T,|T|-1)$. Then $T=T''\bigsqcup\{\alpha',\alpha''\}, T'=T''\bigsqcup\{\alpha\}$ and $d_{\mathrm{ch},(T,T')}$ is the exterior tensor product of the chiral operation
$$
d_{\mu,(T,T')}=\mu_{\mathcal{A}}\shift: (\mathcal{A}_{\alpha'}\shift\boxtimes \mathcal{A}_{\alpha''}\shift)(*\Delta_{\{\alpha',\alpha''\}})\rightarrow\Delta_*\mathcal{A}_{\alpha}\shift
$$
and the identity map for $\mathcal{A}^{\boxtimes T''}$.

 We can extend the above construction of the Chevalley-Cousin complex to the case with coefficient  $\mathcal{Q}_{X^\bullet}$ (here $\mathcal{Q}$ stands for Dolbeault complex, see Section \ref{Conventions} for conventions)
$$
\boxed{C(\mathcal{A})_{\mathcal{Q},X^I}:=\bigoplus_{T\in \equiQ(I)} \Delta^{(I/T)}_*((\mathcal{A}\shift)^{\boxtimes T}(*\Delta_T)_\mathcal{Q}).}
$$
where the differential is $$d^{\mathcal{Q}}_{\mu,(T,T')}:=\mu^{\mathcal{Q}}_{\mathcal{A}}\shift:(\mathcal{A}_{\alpha'}\shift\boxtimes \mathcal{A}_{\alpha''}\shift)(*\Delta_{\{\alpha',\alpha''\}})_\mathcal{Q}\rightarrow \Delta_*((\mathcal{A}_{\alpha}\shift)_\mathcal{Q}).$$

 Since we mainly deal with $(C(\mathcal{A})_{\mathcal{Q},X^I},d^{\mathcal{Q}}_{\mu})$, from now on we omit the superscript ${\mathcal{Q}}$ of the differential $d^{\mathcal{Q}}_{\mu,(T,T')}$.
\begin{rem}
  Here we view a smooth form in $\Omega^{0,\bullet}$ as a degree $\bullet$ element.
\end{rem}

Now we are ready to introduce the chiral chain complex.

\begin{defn} {\cite{beilinson2004chiral}}\label{ChiralChain}
  Let $\mathcal{A}$ be a chiral algebra on $X$.
Let $\mathcal{S}$ be the category of finite non-empty sets whose morphisms are surjections.

The chiral chain complex is defined as follows
  $$
\boxed{\tilde{C}^{\mathrm{ch}}(X,\mathcal{A})^{-\bullet}_{\mathcal{Q}}:=\mathrm{hocolim}(I\in\mathcal{S}, \Gamma(X^I, F^{\bullet}_{X^I}))}  $$
  where $F_{X^I}=\mathrm{DR}(C(\mathcal{A})_{\mathcal{Q},X^I})=\mathrm{DR}\big(\mathop{\bigoplus}\limits_{T\in \equiQ(I)}\Delta_*^{(I/T)}((\mathcal{A}\shift)^{\boxtimes T}(*\Delta_T)_\mathcal{Q})\big)$. The notation $\mathrm{hocolim}$ means that we take the homotopy colimit of the diagram in the bracket. The differential is
  $$
  d_{\mathcal{A}}^{\mathrm{ch}}=d_{\mathrm{DR}}+\bar{\partial}+d_{\mu}+d_{\mathrm{ho}}.
   $$
   The differential $d_{\mathrm{ho}}$ is coming from the homotopy colimit and will not play any role in our paper.
\end{defn}

 From the definition of the chiral chain complex, the homological grading is
$$
\mathrm{hodeg}(\alpha)=-\mathrm{deg}(\alpha)=-q+|T|-p,\quad\alpha\in \Omega^{0,q}(X^I,\mathrm{DR}^p\Delta_*^{(I/T)}(\mathcal{A}\shift)^{\boxtimes T}(*\Delta_T)),
$$
here we write
$$
\alpha=\eta\cdot a,\quad\eta\in \Omega^{0,q}(X^I,\mathrm{DR}^p\Delta_*^{(I/T)}(\omega_X\shift)^{\boxtimes T}(*\Delta_T)), a\in \Omega^{0,0}(X^I,\mathcal{V}^{\boxtimes I}).
$$
\begin{defn}
  {\cite{beilinson2004chiral}} The chiral homology is defined as the homology of the chiral chain complex
$$
H^{\mathrm{ch}}_{i}(X,\mathcal{A})=H_i\left(\tilde{C}^{\mathrm{ch}}(X,\mathcal{A})^{\bullet}_{\mathcal{Q}}\right).
$$
\end{defn}

We do not need the full details in the construction of the chiral chain complex. We will focus on the following subsector of the chiral chain complex since the most relevant differentials for our construction are $d_{\mu}$ and $\bar{\partial}$.

$$
\begin{tikzcd}
\cdots \arrow[r,"\bar{\partial}"] & {\Omega^{0,\bullet}\left(X^I, \Delta^{(I/T)}_*((\mathcal{A}\shift)^{\boxtimes T}(*\Delta_T))\right)} \arrow[d,"d_{\mu}"] \arrow[r,"\bar{\partial}"] & {\Omega^{0,\bullet+1}\left(X^I, \Delta^{(I/T)}_*((\mathcal{A}\shift)^{\boxtimes T}(*\Delta_T))\right)} \arrow[d,"d_{\mu}"] \arrow[r,"\bar{\partial}"] & \cdots \\
\cdots \arrow[r,"\bar{\partial}"] & {\Omega^{0,\bullet}\left(X^I, \Delta^{(I/T')}_*((\mathcal{A}\shift)^{\boxtimes T'}(*\Delta_{T'}))\right)} \arrow[r,"\bar{\partial}"]                                                                      & {\Omega^{0,\bullet+1}\left(X^I, \Delta^{(I/T')}_*((\mathcal{A}\shift)^{\boxtimes T'}(*\Delta_{T'}))\right)}  \arrow[r,"\bar{\partial}"]                                                  & \cdots
\end{tikzcd}
$$

Here $\Omega^{0,\bullet}\left(X^I, \Delta^{(I/T)}_*((\mathcal{A}\shift)^{\boxtimes T}(*\Delta_T))\right):=\Gamma\left(X^I,\Delta^{(I/T)}_*((\mathcal{A}\shift)^{\boxtimes T}(*\Delta_T)\otimes \Omega^{0,\bullet}_{X^T})\right)$.

We have a similar complex for Lie* algebra. One defines
\begin{equation}\label{ChainLie}
\boxed{\tilde{C}^{\mathrm{Lie}}(X,\mathcal{L})^{-\bullet}_{\mathcal{Q}}:=\mathrm{hocolim}(I\in\mathcal{S}, \Gamma(X^I, F^{\mathrm{Lie},\bullet}_{X^I}))}
\end{equation}

  where $F^{\mathrm{Lie}}_{X^I}=\mathrm{DR}(C^{\mathrm{Lie}}(\mathcal{L})_{\mathcal{Q},X^I})=\mathrm{DR}\big(\mathop{\bigoplus}\limits_{T\in \equiQ(I)}\Delta_*^{(I/T)}((\mathcal{L}\shift)^{\boxtimes T}_\mathcal{Q})\big)$. The differential $d_{\mathrm{Lie}}$ is defined in a similar way to the chiral differential $d_{\mathrm{ch}}$.

\subsection{Batalin-Vilkovisky (BV) formalism}
We recall the notion of BV quantization \cite{batalin1981gauge} and list the mathematical structures that are relevant to us. For thorough treatments, we refer the reader to \cite{costello2022renormalization,costello2021factorization}.

\begin{defn}[BV algebra]
A Batalin-Vilkovisky (BV) algebra is a pair $(O_{\mathrm{BV}},\Delta_{\mathrm{BV}})$ where
\begin{itemize}
  \item $O_{\mathrm{BV}}$ is a $\mathbb{Z}$-graded commutative associative unital algebra over $\mathbb{C}$.
  \item $\Delta:O_{\mathrm{BV}}\rightarrow O_{\mathrm{BV}}$ is a linear operator of degree 1 such that $\Delta^2=0$.
  \item Define $\{-,-\}:O_{\mathrm{BV}}\otimes O_{\mathrm{BV}}\rightarrow O_{\mathrm{BV}}$ by
  $$
  \{a,b\}:=\Delta(ab)-(\Delta a)b-(-1)^{|a|}a\Delta b, \ a,b\in O_{\mathrm{BV}}.
  $$
  Then $\{-,-\}$ satisfies the following graded Leibnitz rule
  $$
  \{a,bc\}:=\{a,b\}c+(-1)^{(|a|+1)|b|}b\{a,c\},\ \ a,b,c\in O_{\mathrm{BV}}.
  $$
\end{itemize}

\end{defn}

The main example of BV algebras that we are going to study is a finitely generated algebra construct as follows.

Let $E$ be a holomorphic Hermitian vector bundle over $X$ with a symplectic pairing $E\otimes_{\mathcal{O}_X} E\rightarrow \omega_X$. Denote the space of harmonic forms by $\mathbb{H}(X,E)$, then the pairing
$$
\int_X\langle-,-\rangle: \mathbb{H}(X,E)\otimes \mathbb{H}(X,E)\rightarrow \mathbb{C}
$$
is a (-1)-shifted symplectic pairing. This induces the BV algebra structure on the algebra $O(\mathbb{H}(X,E))$.

\begin{defn}
Let $(O_{\mathrm{BV}},\Delta)$ be a BV algebra. Assume that $I\in O_{\mathrm{BV}}$ is a nilpotent element. Then the element $I$ is said to satisfy the quantum master equation (QME) if
$$
{{\Delta e^{I}=0.}}
$$
\end{defn}

This is equivalent to
\begin{equation}\label{QME2}
\Delta I+\frac{1}{2}\{I,I\}=0.
\end{equation}

In general, if $(C_{\bullet},d_{C})$ is a $\mathbb{C}$-chain complex, we introduce the following generalized notion of the quantum master equation.
\begin{defn}\label{GenQME}
 W say a $\mathbb{C}$-linear map
$$
\langle-\rangle: C_{\bullet}\rightarrow O_{\mathrm{BV}}
$$
satisfies QME if
$$
{(d_{C}+\Delta)\langle-\rangle=0}.
$$
\end{defn}

\begin{rem}\label{GenQMERem}
If we take $(C_{\bullet},d)=(\mathbb{C},0)$, the map $I(-)$
$$
I(-):\mathbb{C}\rightarrow O_{\mathrm{BV}}, \ I(c)=ce^{I}
$$
satisfies QME if and only if $I\in O_{\mathrm{BV}}$ itself satisfies QME.
\end{rem}
Roughly speaking, the effective BV quantization theory $\mathcal{T}$ on a smooth manifold $X$ contains
\begin{itemize}
 \item  Quantum  observable algebra (factorization algebra)  : ${\mathrm{Obs}_\mathcal{T}}$, a vector space equipped with a certain algebraic structure. Here we mean factorization algebras in the sense of Costello and Gwilliam \cite{costello2021factorization}.
  \item Factorization homology (complex): ${(C_{\bullet}(\mathrm{Obs}_\mathcal{T}),d_\mathcal{T})}$, a $\mathbb{C}$-chain complex.
  \item A BV algebra ${(O_{\mathrm{BV},\mathcal{T}},\Delta)}$. This BV algebra corresponds to the algebra of the zero modes in effective BV quantization.
  \item A linear map
  $$ \mathrm{Tr}_T: C_{\bullet}(\mathrm{Obs}_\mathcal{T})\rightarrow (O_{\mathrm{BV}}, -\Delta_{\mathrm{BV}}),
  $$
which satisfies QME, that is, we have
  $$
{(d_\mathcal{T}+\Delta_{\mathrm{BV}})\mathrm{Tr}_\mathcal{T}(-)=0.}
  $$
\end{itemize}

\begin{rem}
    In quantum field theory, people are interested in calculating correlation functions of quantum observables. In BV formalism, one interprets the physical path integral of the theory with BV action $S_{\mathrm{BV}}(\phi)$
    $$
    \int \mathcal{D}\phi \ e^{S_{\mathrm{BV}}(\phi)}\mathcal{O}_1\cdots \mathcal{O}_n,
    $$
    as a solution of the quantum master equation, that is, a chain map
    $$
    \mathrm{Tr}: (C_{\bullet}(\mathrm{Obs}_{\mathcal{T}}),d_{\mathcal{T}})\rightarrow (O_{\mathrm{BV}},- \Delta_{\mathrm{BV}}).
    $$
    Here we intuitively treat the physical operator insertion $\mathcal{O}_1\cdots \mathcal{O}_n$ as an element in the complex $(C_{\bullet}(\mathrm{Obs}_{\mathcal{T}}),d_{\mathcal{T}})$. The algebraic structure of the quantum observables is encoded in the differential $d_{\mathcal{T}}$.
\end{rem}

\subsection{Trace map on chiral homology of chiral Weyl algebras}
The physical model behind the chiral Weyl algebra is the quantization of a free two-dimensional chiral quantum field theory. The action functional of this free theory is given by
$$
S_{\mathrm{BV}}(\phi)=\int_{X}\langle\phi,\bar{\partial}\phi\rangle,\quad \phi \in \mathcal{E}=\Omega^{0,\bullet}(X,E).
$$

The space $\Omega^{0,\bullet}(X,\mathrm{Sym}\ \mathcal{L})$ can be viewed as classical local functionals. The quantization of the chiral Lagrangian theory $\mathcal{T}_E$ will give rise to the chiral Weyl algebra $\mathcal{A}=\mathscr{U}(\mathcal{L})^{\flat}$. This chiral algebra captures the algebraic structure of quantum observables in the theory $\mathcal{T}_E$.

In this theory, the corresponding BV algebra is the finitely generated algebra $O_{\mathrm{BV}}:=O(\mathbb{H}(X,E))$. Here we choose a Hermitian metric on $E$ and denote the harmonic elements in $\Omega^{0,\bullet}(X,E)$ by $\mathbb{H}(X,E)$. The (-1)-shifted symplectic pairing on $\mathbb{H}(X,E)$ is given by $\int_X\langle-,-\rangle$. Applying the BV quantization machinery, we expect to have a trace map
$$
\mathrm{Tr}:\tilde{C}^{\mathrm{ch}}(X,\mathcal{A})_{\mathcal{Q}}\rightarrow O_{\mathrm{BV}}.
$$

To summarize, we have (the trace map is yet to be constructed)

\begin{itemize}
 \item Local observable algebra (Factorization algebra) : $\mathrm{Obs}_{\mathcal{T}_E}=\mathcal{A}:=\mathscr{U}(\mathcal{L})^{\flat}$ the chiral Weyl algebra,
  \item Factorization homology (complex): ${(C_{\bullet}(\mathrm{Obs}_{\mathcal{T}_E}),d_{\mathcal{T}_E})}=\tilde{C}^{\mathrm{ch}}(X,\mathcal{A})_{\mathcal{Q}}$ is the chiral chain complex.
  \item The BV algebra ${(O_{\mathrm{BV},\mathcal{T}_E}=\mathcal{O}(\mathbb{H}(X,E)),\Delta_{\mathrm{BV}})}$ which is polynomial functions on the space of the harmonic elements $\mathbb{H}(X,E)$.
  \item A linear map
  $$ \mathrm{Tr}_{\mathcal{T}_E}: \tilde{C}^{\mathrm{ch}}(X,\mathcal{A})_{\mathcal{Q}}\rightarrow (O_{\mathrm{BV}}, -\Delta_{\mathrm{BV}}),
  $$
which satisfies QME, that is, we have
  $$
{(d_{\mathcal{T}_E}+\Delta_{\mathrm{BV}})\mathrm{Tr}_{\mathcal{T}_E}(-)=0.}
  $$
\end{itemize}

 The main theorem of this paper is the following result.

\begin{thm}
     There is an explicit trace map $\mathrm{Tr}_{\mathcal{T}_E}$ constructed using Feynman diagrams and is furthermore a quasi-isomorphism.
     \end{thm}
 This is proved in the next section.

\section{Construction of the trace map}

We prove the main theorem (Theorem \ref{MainThmIntro}) in this section. Since the chiral algebra $\mathcal{A}=\mathscr{U}(\mathcal{L})^{\flat}$ has a complicated coordinate change formula, the Feynman diagram construction in \cite{gui2021elliptic} which works for the genus 1 case and trivial vector bundles does not apply directly. One can first write the Feynman diagram formula using local coordinates and then check that the expression is well defined under the coordinate change. Here we take another strategy. First notice that the Feynman diagram is easy to define for linear field insertions. Then by presenting the element in $\mathcal{A}$ as an iterated chiral product of linear fields, we can define the Feynman diagram for general operator insertions. Finally, we use some results in \cite{gui2021elliptic} to prove that this construction is independent of the presentation we choose. For simplicity of notation, we work with purely even holomorphic bundles but the proof works for super case as well.

Throughout this section, $X$ denotes a compact smooth Riemann surface and $E$ is a holomorphic Hermitian vector bundle over $X$ equipped with a symplectic pairing $E\otimes_{\mathcal{O}_X}E\rightarrow \omega_X$. The chiral Weyl algebra $\mathcal{A}=\mathscr{U}(\mathcal{L})^{\flat}$ is defined in Definition \ref{DefnChiralWeyl}, where $\mathcal{L}=E_{\mathcal{D}}=E\otimes_{\mathcal{O}_X}\mathcal{D}_X$. The chiral operation of $\mathcal{A}$ (resp. $\mathrm{Sym} \ \mathcal{L}$) is denoted by $\mu_{\mathcal{A}}$ (resp. $\mu_{\mathrm{Sym}}$). For any vector bundle $F$ over $X^n=X\times\cdots\times X$, the notation $F(*\Delta)$ always means that the sheaf of sections of $F$ with poles along the big diagonal.
\subsection{Trace map on Lie* algebras}

In this section, we construct a trace map on the complex $\tilde{C}^{\mathrm{Lie}}(X,\mathcal{L})_{\mathcal{Q}}$ (See (\ref{ChainLie})). Since we only need to deal with linear fields (that is, $\mathcal{L}\subset \mathcal{A}=\mathscr{U}(\mathcal{L})^{\flat})$, the construction is much easier than the one for the chiral envelope. We will use this construction in the later sections.

Given a section $e\in \Omega^{0,\bullet}(X,E)$. We can define the operation
$$
\partial_e:\Omega^{0,\bullet}(X,E_{\mathcal{D}})\rightarrow \Omega^{0,\bullet}(X,\omega_X)
$$
induced by the following sequence of maps
$$
E\otimes_{\mathcal{O}_X} E_{\mathcal{D}}=E\otimes_{\mathcal{O}_X} E\otimes_{\mathcal{O}_X} \mathcal{D}_X\rightarrow \omega_X\otimes_{\mathcal{O}_X} \mathcal{D}_X\rightarrow \omega_X\otimes_{\mathcal{D}_X} \mathcal{D}_X=\omega_X,
$$
here we use the pairing $E\otimes E\rightarrow\omega_X$. Furthermore, we see that $\partial_e$ is compatible with the right $\mathcal{D}_X$-module structure as the pairing is on the left. One can extend the operator $\partial_e$ to an operator on $\Omega^{0,\bullet}(X^n,\mathcal{L}^{\flat\boxtimes n})=\Omega^{0,\bullet}(X^n,(E_{\mathcal{D}}\oplus \omega_X)^{\boxtimes n})$
$$
(\partial_e)_n:\Omega^{0,\bullet}(X^n,\mathcal{L}^{\flat\boxtimes n})\rightarrow \Omega^{0,\bullet}(X^n,\mathcal{L}^{\flat\boxtimes n}),
$$
by setting
$$
(\partial_e)_n:=\sum_{i=1}^n1\boxtimes \cdots\boxtimes  \underbrace{\partial_e}_{\text{i-th}}\boxtimes\cdots\boxtimes 1
$$
and $\partial_e$ acts trivially on $\omega_X\subset \mathcal{L}^{\flat}$. By abuse of notation, we write $\partial_e$ instead of $(\partial_e)_n$.

\begin{rem}
    This operation can be easily extended to $\Omega^{0,\bullet}(X,\mathrm{Sym}\ \mathcal{L})$ by Leibniz rule.
    \end{rem}

Similarly, we can define an operation
$$
\partial_K:\Omega^{0,\bullet}(X^2,\mathcal{L}^{\boxtimes 2})\rightarrow \Omega^{0,\bullet}(X^2,\omega_{X^2}(*\Delta)).
$$
for a kernel function $K\in \Omega^{0,\bullet}(X^2,E\boxtimes E(*\Delta))$. As before, we can extend it to an operator on $\Omega^{0,\bullet}(X^n,\mathcal{L}^{\flat\boxtimes n}(*\Delta))$
$$
\partial_K:\Omega^{0,\bullet}(X^n,\mathcal{L}^{\flat\boxtimes n}(*\Delta))\rightarrow \Omega^{0,\bullet}(X^n,\mathcal{L}^{\flat\boxtimes n}(*\Delta)).
$$

\begin{rem}
    Later we will use the bidifferential operator $\partial_K$ acting on the space $\Omega^{0,\bullet}(X^n,(\mathrm{Sym}\ \mathcal{L}^{\flat})^{\boxtimes n}(*\Delta))$.
\end{rem}

In this paper, we will mainly use the Szeg\"{o} kernel $P\in \Omega^{0,0}(X^2,E\boxtimes E(*\Delta))$ which depends on the Hermitian metric on $E$. From now on, we fix such a metric on $E$. The Szeg\"{o} kernel is the kernel function that inverts $\bar{\partial}$ on the orthogonal complement of $\mathbb{H}^0(X,E)$. Here $\mathbb{H}(X,E)=\mathbb{H}^0(X,E)\oplus \mathbb{H}^1(X,E)$ is the space of harmonic forms in $\Omega^{0,\bullet}(X,E)$ with respect to the Hermitian metric on $E$.

For more details about the Szeg\"{o} kernel, see \cite[Chapter 2]{fay1992kernel}.

Write
\begin{equation}\label{ZeroModes}
\bar{\partial}P(z_1,z_2)=\sum_{i,j} I^{ij}_1\mathbf{e}^1_i\boxtimes \mathbf{e}^0_j+\sum_{i,j} I^{ij}_2\mathbf{e}^0_i\boxtimes \mathbf{e}^1_j,
\end{equation}
here $\{\mathbf{e}^1_i\}$ (resp.$\{\mathbf{e}^0_i\}$) are harmonic basis of $\mathbb{H}^1(X,E)$ (resp. $\mathbb{H}^0(X,E)$). Notice that there is a perfect paring between $\mathbb{H}^1(X,E)$ and $\mathbb{H}^0(X,E)$ by
$$
\int_X\langle-,-\rangle: \mathbb{H}^0(X,E)\otimes \mathbb{H}^1(X,E)\rightarrow \mathbb{C}.
$$

Using the formula
$$
\int_X\langle P(z_1,z_2),\bar{\partial}_{z_2}e(z_2)\rangle=e(z_1)-\int_X\langle \bar{\partial}_{z_2}P(z_1,z_2),e(z_2)\rangle,\quad e\in \Omega^0(X,E).
$$
We find out that the matrix $I^t_1=I_2=[I_2^{ij}]$ is the inverse of $[\int_X\langle \mathbf{e}^1_j,\mathbf{e}^0_i\rangle]$.  The BV operator $\Delta_{\mathrm{BV}}$ on $O_{\mathrm{BV}}=\mathcal{O}(\mathbb{H}(X,E))$ is given by the following formula
$$
\Delta_{\mathrm{BV}}=\sum_{i,j}I^{ij}_1\partial_{\mathbf{e}^0_i}\cdot\partial_{\mathbf{e}^1_j}.
$$

Now we are ready to construct the trace map on the Lie* algebra $\mathcal{L}^{\flat}=E_{\mathcal{D}}\oplus\omega_X$. First, we define

$$
\mathcal{W}_{\mathrm{Lie}}:\Omega^{0,\bullet}(X^n,\mathcal{L}^{\flat\boxtimes n})\xrightarrow{e^{\partial_P}} \Omega^{0,\bullet}(X^n,\mathcal{L}^{\flat\boxtimes n}(*\Delta)).
$$

Denote the projection map
$$
\Omega^{0,\bullet}(X^n,\mathcal{L}^{\flat\boxtimes n}(*\Delta))\rightarrow  \Omega^{0,\bullet}(X^n,\omega_{X^n}(*\Delta))
$$
by $\mathbf{p}$ which is induced by projection $\mathcal{L}^{\flat}=E_{\mathcal{D}}\oplus\omega_X\rightarrow \omega_X$. By \cite[pp315,Section 4.3.3.]{beilinson2004chiral}, there is a trace map
$$
\mathrm{tr}_{\omega}:\tilde{C}^{\mathrm{ch}}(X,\omega_X)_{\mathcal{Q}} \rightarrow \mathbb{C}.
$$
Furthermore, the above map is a quasi-isomorphism. In particular, we have trace map on $\Omega^{0,\bullet}(X^n,\omega_{X^n}(*\Delta))$ by the following  sequence of maps
$$
\Omega^{0,\bullet}(X^n,\omega_{X^n}(*\Delta))\rightarrow \tilde{C}^{\mathrm{ch}}(X,\omega_X)_{\mathcal{Q}} \xrightarrow{\sim} \mathbb{C}.
$$

By abuse of notation, we continue to write $\mathrm{tr}_{\omega}$ for the above map.

\begin{rem}
    The above trace map can be viewed as a renormalized integral for singular differential forms on  $X^n$ with poles along the diagonal. It is shown in \cite[Section 2.3.]{gui2021elliptic}, this trace map coincides with the regularized integral introduced in \cite{li2021regularized}. For a deeper discussion of the regularized integral, we refer the reader to \cite{li2021regularized}.
\end{rem}

Then we construct the trace map as follows, define
$$
\mathrm{Tr}_{\mathrm{Lie}}: \Omega^{0,\bullet}(X^n,\mathcal{L}^{\flat\boxtimes n})\rightarrow O_{\mathrm{BV}}.
$$
Here
$$
\mathrm{Tr}_{\mathrm{Lie}}(\eta)[\mathbf{e}]:=\sum_{k\geq 0}\frac{1}{k!}\mathrm{tr}_{\omega}\circ \mathbf{p}\left(\partial^k_{\mathbf{e}}\mathcal{W}_{\mathrm{Lie}}(\eta)\right),\quad \eta\in \Omega^{0,\bullet}(X^n,\mathcal{L}^{\flat\boxtimes n}), \mathbf{e}\in \mathbb{H}(X,E) .
$$
Since all the construction is compatible with the $\mathcal{D}$-module structure, the above map extends naturally to a map $\mathrm{Tr}_{\mathrm{Lie}}:\tilde{C}^{\mathrm{Lie}}(X,\mathcal{L})_{\mathcal{Q}}\rightarrow O_{\mathrm{BV}}$ using the same formula. See Fig. \ref{TraceLie} for an intuitive illustration of the construction.

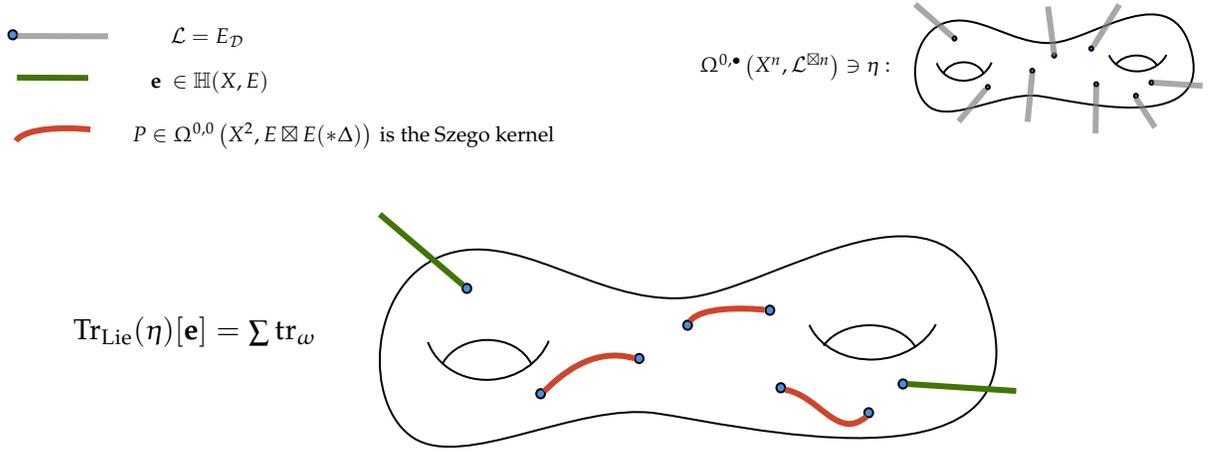
\begin{figure}
    \centering

\tikzset{every picture/.style={line width=0.75pt}} 

\begin{tikzpicture}[x=0.75pt,y=0.75pt,yscale=-1,xscale=1]

\draw  [fill={rgb, 255:red, 74; green, 144; blue, 226 }  ,fill opacity=1 ] (362.14,167.1) .. controls (362.14,165.82) and (363.13,164.79) .. (364.35,164.79) .. controls (365.58,164.79) and (366.57,165.82) .. (366.57,167.1) .. controls (366.57,168.38) and (365.58,169.42) .. (364.35,169.42) .. controls (363.13,169.42) and (362.14,168.38) .. (362.14,167.1) -- cycle ;
\draw  [fill={rgb, 255:red, 74; green, 144; blue, 226 }  ,fill opacity=1 ] (403.64,159.54) .. controls (403.64,158.26) and (404.63,157.22) .. (405.85,157.22) .. controls (407.07,157.22) and (408.06,158.26) .. (408.06,159.54) .. controls (408.06,160.82) and (407.07,161.86) .. (405.85,161.86) .. controls (404.63,161.86) and (403.64,160.82) .. (403.64,159.54) -- cycle ;
\draw   (239.97,132.33) .. controls (277.32,117.36) and (316.9,156.14) .. (362.3,153.37) .. controls (407.7,150.6) and (498.51,77.2) .. (518.3,165.83) .. controls (538.09,254.47) and (387.91,217.08) .. (348.33,211.54) .. controls (308.75,206) and (251.61,240.36) .. (226.1,224) .. controls (200.59,207.65) and (202.62,147.3) .. (239.97,132.33) -- cycle ;
\draw  [draw opacity=0] (489.57,166.61) .. controls (484.54,176.46) and (472.82,183.72) .. (458.94,184.4) .. controls (443.77,185.14) and (430.62,177.8) .. (425.76,166.97) -- (457.99,156.78) -- cycle ; \draw   (489.57,166.61) .. controls (484.54,176.46) and (472.82,183.72) .. (458.94,184.4) .. controls (443.77,185.14) and (430.62,177.8) .. (425.76,166.97) ;
\draw  [draw opacity=0] (432.91,177.44) .. controls (437.04,171.54) and (445.52,167.37) .. (455.38,167.14) .. controls (466.3,166.89) and (475.68,171.55) .. (479.4,178.32) -- (455.67,185.16) -- cycle ; \draw   (432.91,177.44) .. controls (437.04,171.54) and (445.52,167.37) .. (455.38,167.14) .. controls (466.3,166.89) and (475.68,171.55) .. (479.4,178.32) ;
\draw  [draw opacity=0] (294.35,174.98) .. controls (289.87,185.81) and (278.52,193.89) .. (264.99,194.63) .. controls (250.34,195.43) and (237.69,187.36) .. (233.37,175.59) -- (264.09,165.71) -- cycle ; \draw   (294.35,174.98) .. controls (289.87,185.81) and (278.52,193.89) .. (264.99,194.63) .. controls (250.34,195.43) and (237.69,187.36) .. (233.37,175.59) ;
\draw  [draw opacity=0] (241.01,186.06) .. controls (244.72,179.5) and (252.98,174.81) .. (262.63,174.56) .. controls (273.21,174.3) and (282.27,179.45) .. (285.57,186.85) -- (262.9,193.44) -- cycle ; \draw   (241.01,186.06) .. controls (244.72,179.5) and (252.98,174.81) .. (262.63,174.56) .. controls (273.21,174.3) and (282.27,179.45) .. (285.57,186.85) ;
\draw  [fill={rgb, 255:red, 74; green, 144; blue, 226 }  ,fill opacity=1 ] (288.02,201.69) .. controls (288.02,200.41) and (289.01,199.37) .. (290.24,199.37) .. controls (291.46,199.37) and (292.45,200.41) .. (292.45,201.69) .. controls (292.45,202.97) and (291.46,204.01) .. (290.24,204.01) .. controls (289.01,204.01) and (288.02,202.97) .. (288.02,201.69) -- cycle ;
\draw  [fill={rgb, 255:red, 74; green, 144; blue, 226 }  ,fill opacity=1 ] (250.81,148.51) .. controls (250.81,147.23) and (251.8,146.2) .. (253.02,146.2) .. controls (254.24,146.2) and (255.24,147.23) .. (255.24,148.51) .. controls (255.24,149.79) and (254.24,150.83) .. (253.02,150.83) .. controls (251.8,150.83) and (250.81,149.79) .. (250.81,148.51) -- cycle ;
\draw  [fill={rgb, 255:red, 74; green, 144; blue, 226 }  ,fill opacity=1 ] (453.48,211.29) .. controls (453.48,210.01) and (454.47,208.97) .. (455.69,208.97) .. controls (456.91,208.97) and (457.9,210.01) .. (457.9,211.29) .. controls (457.9,212.57) and (456.91,213.61) .. (455.69,213.61) .. controls (454.47,213.61) and (453.48,212.57) .. (453.48,211.29) -- cycle ;
\draw  [fill={rgb, 255:red, 74; green, 144; blue, 226 }  ,fill opacity=1 ] (470.76,196.76) .. controls (470.76,195.48) and (471.75,194.44) .. (472.97,194.44) .. controls (474.2,194.44) and (475.19,195.48) .. (475.19,196.76) .. controls (475.19,198.04) and (474.2,199.07) .. (472.97,199.07) .. controls (471.75,199.07) and (470.76,198.04) .. (470.76,196.76) -- cycle ;
\draw  [fill={rgb, 255:red, 74; green, 144; blue, 226 }  ,fill opacity=1 ] (337.71,183.91) .. controls (337.71,182.63) and (338.7,181.59) .. (339.93,181.59) .. controls (341.15,181.59) and (342.14,182.63) .. (342.14,183.91) .. controls (342.14,185.19) and (341.15,186.23) .. (339.93,186.23) .. controls (338.7,186.23) and (337.71,185.19) .. (337.71,183.91) -- cycle ;
\draw  [fill={rgb, 255:red, 74; green, 144; blue, 226 }  ,fill opacity=1 ] (409.16,198.68) .. controls (409.16,197.4) and (410.15,196.36) .. (411.37,196.36) .. controls (412.59,196.36) and (413.58,197.4) .. (413.58,198.68) .. controls (413.58,199.96) and (412.59,201) .. (411.37,201) .. controls (410.15,201) and (409.16,199.96) .. (409.16,198.68) -- cycle ;
\draw [color={rgb, 255:red, 206; green, 69; blue, 44 }  ,draw opacity=1 ][fill={rgb, 255:red, 255; green, 255; blue, 255 }  ,fill opacity=1 ][line width=2.25]    (292.08,200.46) .. controls (299.08,193.46) and (316.08,177.46) .. (337.71,183.91) ;
\draw [color={rgb, 255:red, 65; green, 117; blue, 5 }  ,draw opacity=1 ][line width=2.25]    (209.11,111.12) -- (251.02,146.83) ;
\draw [color={rgb, 255:red, 206; green, 69; blue, 44 }  ,draw opacity=1 ][fill={rgb, 255:red, 255; green, 255; blue, 255 }  ,fill opacity=1 ][line width=2.25]    (365.57,165.1) .. controls (370.57,158.1) and (391.57,158.1) .. (403.12,159.35) ;
\draw [color={rgb, 255:red, 206; green, 69; blue, 44 }  ,draw opacity=1 ][fill={rgb, 255:red, 255; green, 255; blue, 255 }  ,fill opacity=1 ][line width=2.25]    (413.58,199.68) .. controls (430.58,202.51) and (439.85,225.67) .. (454.48,213.29) ;
\draw [color={rgb, 255:red, 65; green, 117; blue, 5 }  ,draw opacity=1 ][line width=2.25]    (475.19,196.76) -- (530.08,200.29) ;
\draw  [fill={rgb, 255:red, 74; green, 144; blue, 226 }  ,fill opacity=1 ] (21.81,20.51) .. controls (21.81,19.23) and (22.8,18.2) .. (24.02,18.2) .. controls (25.24,18.2) and (26.24,19.23) .. (26.24,20.51) .. controls (26.24,21.79) and (25.24,22.83) .. (24.02,22.83) .. controls (22.8,22.83) and (21.81,21.79) .. (21.81,20.51) -- cycle ;
\draw [color={rgb, 255:red, 128; green, 128; blue, 128 }  ,draw opacity=0.68 ][line width=2.25]    (26.11,21.12) -- (72.08,21.29) ;
\draw  [fill={rgb, 255:red, 74; green, 144; blue, 226 }  ,fill opacity=1 ] (548.22,30.96) .. controls (548.22,30.37) and (548.67,29.89) .. (549.22,29.89) .. controls (549.77,29.89) and (550.22,30.37) .. (550.22,30.96) .. controls (550.22,31.55) and (549.77,32.03) .. (549.22,32.03) .. controls (548.67,32.03) and (548.22,31.55) .. (548.22,30.96) -- cycle ;
\draw  [fill={rgb, 255:red, 74; green, 144; blue, 226 }  ,fill opacity=1 ] (566.97,27.47) .. controls (566.97,26.88) and (567.41,26.4) .. (567.97,26.4) .. controls (568.52,26.4) and (568.96,26.88) .. (568.96,27.47) .. controls (568.96,28.06) and (568.52,28.54) .. (567.97,28.54) .. controls (567.41,28.54) and (566.97,28.06) .. (566.97,27.47) -- cycle ;
\draw   (493.04,14.9) .. controls (509.91,7.99) and (527.79,25.9) .. (548.3,24.62) .. controls (568.8,23.34) and (609.82,-10.55) .. (618.76,30.37) .. controls (627.7,71.3) and (559.86,54.03) .. (541.99,51.48) .. controls (524.11,48.92) and (498.3,64.78) .. (486.77,57.23) .. controls (475.25,49.68) and (476.17,21.82) .. (493.04,14.9) -- cycle ;
\draw  [draw opacity=0] (605.83,30.64) .. controls (603.59,35.24) and (598.26,38.63) .. (591.95,38.95) .. controls (585.06,39.29) and (579.1,35.86) .. (576.93,30.82) -- (591.52,26.19) -- cycle ; \draw   (605.83,30.64) .. controls (603.59,35.24) and (598.26,38.63) .. (591.95,38.95) .. controls (585.06,39.29) and (579.1,35.86) .. (576.93,30.82) ;
\draw  [draw opacity=0] (580.15,35.8) .. controls (581.99,33.04) and (585.85,31.08) .. (590.34,30.98) .. controls (595.3,30.86) and (599.56,33.04) .. (601.22,36.19) -- (590.47,39.3) -- cycle ; \draw   (580.15,35.8) .. controls (581.99,33.04) and (585.85,31.08) .. (590.34,30.98) .. controls (595.3,30.86) and (599.56,33.04) .. (601.22,36.19) ;
\draw  [draw opacity=0] (517.64,34.5) .. controls (515.64,39.55) and (510.49,43.33) .. (504.34,43.67) .. controls (497.69,44.04) and (491.95,40.28) .. (490.03,34.81) -- (503.93,30.32) -- cycle ; \draw   (517.64,34.5) .. controls (515.64,39.55) and (510.49,43.33) .. (504.34,43.67) .. controls (497.69,44.04) and (491.95,40.28) .. (490.03,34.81) ;
\draw  [draw opacity=0] (493.47,39.78) .. controls (495.13,36.72) and (498.88,34.52) .. (503.27,34.4) .. controls (508.08,34.28) and (512.18,36.69) .. (513.66,40.13) -- (503.4,43.12) -- cycle ; \draw   (493.47,39.78) .. controls (495.13,36.72) and (498.88,34.52) .. (503.27,34.4) .. controls (508.08,34.28) and (512.18,36.69) .. (513.66,40.13) ;
\draw  [fill={rgb, 255:red, 74; green, 144; blue, 226 }  ,fill opacity=1 ] (514.74,46.93) .. controls (514.74,46.34) and (515.19,45.86) .. (515.74,45.86) .. controls (516.29,45.86) and (516.74,46.34) .. (516.74,46.93) .. controls (516.74,47.52) and (516.29,48) .. (515.74,48) .. controls (515.19,48) and (514.74,47.52) .. (514.74,46.93) -- cycle ;
\draw  [fill={rgb, 255:red, 74; green, 144; blue, 226 }  ,fill opacity=1 ] (497.93,22.38) .. controls (497.93,21.79) and (498.38,21.31) .. (498.93,21.31) .. controls (499.48,21.31) and (499.93,21.79) .. (499.93,22.38) .. controls (499.93,22.97) and (499.48,23.45) .. (498.93,23.45) .. controls (498.38,23.45) and (497.93,22.97) .. (497.93,22.38) -- cycle ;
\draw  [fill={rgb, 255:red, 74; green, 144; blue, 226 }  ,fill opacity=1 ] (589.48,51.36) .. controls (589.48,50.77) and (589.93,50.29) .. (590.48,50.29) .. controls (591.03,50.29) and (591.48,50.77) .. (591.48,51.36) .. controls (591.48,51.95) and (591.03,52.43) .. (590.48,52.43) .. controls (589.93,52.43) and (589.48,51.95) .. (589.48,51.36) -- cycle ;
\draw  [fill={rgb, 255:red, 74; green, 144; blue, 226 }  ,fill opacity=1 ] (597.29,44.65) .. controls (597.29,44.06) and (597.73,43.58) .. (598.29,43.58) .. controls (598.84,43.58) and (599.29,44.06) .. (599.29,44.65) .. controls (599.29,45.24) and (598.84,45.72) .. (598.29,45.72) .. controls (597.73,45.72) and (597.29,45.24) .. (597.29,44.65) -- cycle ;
\draw  [fill={rgb, 255:red, 74; green, 144; blue, 226 }  ,fill opacity=1 ] (537.19,38.72) .. controls (537.19,38.13) and (537.64,37.65) .. (538.19,37.65) .. controls (538.74,37.65) and (539.19,38.13) .. (539.19,38.72) .. controls (539.19,39.31) and (538.74,39.79) .. (538.19,39.79) .. controls (537.64,39.79) and (537.19,39.31) .. (537.19,38.72) -- cycle ;
\draw  [fill={rgb, 255:red, 74; green, 144; blue, 226 }  ,fill opacity=1 ] (569.46,45.54) .. controls (569.46,44.95) and (569.91,44.47) .. (570.46,44.47) .. controls (571.01,44.47) and (571.46,44.95) .. (571.46,45.54) .. controls (571.46,46.13) and (571.01,46.61) .. (570.46,46.61) .. controls (569.91,46.61) and (569.46,46.13) .. (569.46,45.54) -- cycle ;
\draw [color={rgb, 255:red, 128; green, 128; blue, 128 }  ,draw opacity=0.69 ][line width=2.25]    (479.09,5.11) -- (498.03,21.6) ;
\draw [color={rgb, 255:red, 128; green, 128; blue, 128 }  ,draw opacity=0.68 ][line width=2.25]    (599.29,44.65) -- (624.08,46.28) ;
\draw [color={rgb, 255:red, 128; green, 128; blue, 128 }  ,draw opacity=0.68 ][line width=2.25]    (514.74,47.93) -- (501.08,64.29) ;
\draw [color={rgb, 255:red, 128; green, 128; blue, 128 }  ,draw opacity=0.68 ][line width=2.25]    (590.48,52.43) -- (600.08,67.29) ;
\draw [color={rgb, 255:red, 128; green, 128; blue, 128 }  ,draw opacity=0.68 ][line width=2.25]    (570.46,46.61) -- (570.08,70.29) ;
\draw [color={rgb, 255:red, 128; green, 128; blue, 128 }  ,draw opacity=0.68 ][line width=2.25]    (538.19,39.79) -- (535.81,64.47) ;
\draw [color={rgb, 255:red, 128; green, 128; blue, 128 }  ,draw opacity=0.68 ][line width=2.25]    (546.08,6.29) -- (549.3,30.62) ;
\draw [color={rgb, 255:red, 128; green, 128; blue, 128 }  ,draw opacity=0.68 ][line width=2.25]    (568.96,26.47) -- (582.08,5.29) ;
\draw [color={rgb, 255:red, 65; green, 117; blue, 5 }  ,draw opacity=1 ][line width=2.25]    (26.11,42.12) -- (62.08,42.29) ;
\draw [color={rgb, 255:red, 206; green, 69; blue, 44 }  ,draw opacity=1 ][fill={rgb, 255:red, 255; green, 255; blue, 255 }  ,fill opacity=1 ][line width=2.25]    (25.57,74.1) .. controls (30.57,67.1) and (51.57,67.1) .. (63.12,68.35) ;

\draw (53,161.4) node [anchor=north west][inner sep=0.75pt]    {$\mathrm{Tr}_{\mathrm{Lie}}( \eta )[\mathbf{e}] =\sum \mathrm{tr}_{\omega }$};
\draw (102,15.4) node [anchor=north west][inner sep=0.75pt]  [font=\scriptsize]  {$\mathcal{L} =E_{\mathcal{D}}$};
\draw (369,28.4) node [anchor=north west][inner sep=0.75pt]  [font=\scriptsize]  {$\Omega ^{0,\bullet }\left( X^{n} ,\mathcal{L}^{\boxtimes n}\right) \ni \eta :$};
\draw (92,37.4) node [anchor=north west][inner sep=0.75pt]  [font=\scriptsize]  {$\mathbf{e} \ \in \mathbb{H}( X,E)$};
\draw (83,63.4) node [anchor=north west][inner sep=0.75pt]  [font=\scriptsize]  {$P\in \Omega ^{0,0}\left( X^{2} ,E\boxtimes E( *\Delta )\right)\text{ is the Szego kernel}$};

\end{tikzpicture}
    \caption{$\mathrm{Tr}_{\mathrm{Lie}}(\eta)[\mathbf{e}]$}
    \label{TraceLie}
\end{figure}

We have the following main theorem of this section.

\begin{thm}
    The map
    $$
    \mathrm{Tr}_{Lie}:(\tilde{C}^{\mathrm{Lie}}(X,\mathcal{L})_{\mathcal{Q}},d^{\mathrm{Lie}*}_{\mathcal{L}})\rightarrow (O_{\mathrm{BV}},-\Delta_{\mathrm{BV}})
    $$
    is a chain map.
\end{thm}
\begin{proof}

Using (\ref{ZeroModes}), we obtain
\begin{align*}
    \mathrm{Tr}_{\mathrm{Lie}}(\bar{\partial}\eta)[\mathbf{e}]&=\sum_{k\geq 0}\frac{1}{k!}\mathrm{tr}_{\omega}\circ \mathbf{p}\left(\partial^k_{\mathbf{e}}\mathcal{W}_{\mathrm{Lie}}(\bar{\partial}\eta)\right)\\
    &=\sum_{k\geq 0}\frac{1}{k!}\mathrm{tr}_{\omega}\circ \mathbf{p}\left(\partial^k_{\mathbf{e}}(\bar{\partial}-\partial_{\bar{\partial}P(z_1,z_2)})\mathcal{W}_{\mathrm{Lie}}(\eta)\right)\\
    &=\sum_{k\geq 0}\frac{1}{k!}\mathrm{tr}_{\omega}\circ \mathbf{p}\left(\partial^k_{\mathbf{e}}(\bar{\partial}-\Delta_{\mathrm{BV}})\mathcal{W}_{\mathrm{Lie}}(\eta)\right)
\end{align*}

    Notice that by definition, the Lie* pairing
    $$
    \mathcal{L}\boxtimes \mathcal{L}\rightarrow \Delta_*\omega_X
    $$
coincides with the composition
$$
    \mathcal{L}\boxtimes \mathcal{L}\xrightarrow{\partial_P}    \omega_X \boxtimes \omega_X(*\Delta)\xrightarrow{\mu_{\omega}}\Delta_*\omega_X.
    $$
Then
\begin{align*}
    \mathrm{Tr}_{\mathrm{Lie}}(d^{\mathrm{Lie}*}_{\mathcal{L}}\eta)[\mathbf{e}]&=\sum_{k\geq 0}\frac{1}{k!}\mathrm{tr}_{\omega}\circ \mathbf{p}\left(\partial^k_{\mathbf{e}}\mathcal{W}_{\mathrm{Lie}}(d^{\mathrm{Lie}*}_{\mathcal{L}}\eta)\right)\\
    &=\sum_{k\geq 0}\frac{1}{k!}\mathrm{tr}_{\omega}\circ \mathbf{p}\left((d_{\mathrm{Sym}}^{\mathrm{ch}}-\Delta_{\mathrm{BV}})\partial^k_{\mathbf{e}}\mathcal{W}_{\mathrm{Lie}}(\eta)\right)\\
    &=\underbrace{\sum_{k\geq 0}\frac{1}{k!}\mathrm{tr}_{\omega}\circ d_{\omega}^{\mathrm{ch}}\mathbf{p}\left(\partial^k_{\mathbf{e}}\mathcal{W}_{\mathrm{Lie}}(\eta)\right)}_{=0}-\sum_{k\geq 0}\frac{1}{k!}\mathrm{tr}_{\omega}\circ \mathbf{p}\left(\Delta_{\mathrm{BV}}\partial^k_{\mathbf{e}}\mathcal{W}_{\mathrm{Lie}}(\eta)\right)\\
    &=-\left(\Delta_{\mathrm{BV}}\mathrm{Tr}_{\mathrm{Lie}}(\eta)\right)[\mathbf{e}].
\end{align*}

    .
\end{proof}

\subsection{Normal ordering map}
By Remark \ref{PBW}, we have locally defined identification once we choose  an open chart $U\subset X$ with a local coordinate $z$
$$
\tau^z_U: \mathcal{A}|_U\rightarrow \mathrm{Sym}\ \mathcal{L}|_U.
$$
By abuse of notation, we use the same notation for $\mathcal{A}^l|_U\rightarrow \mathrm{Sym}\ \mathcal{L}^l|_U$. However, this map is not well defined globally as the coordinate change formula of $\mathcal{A}$ is different from that of $\mathrm{Sym}\ \mathcal{L}$. The goal of this section is to construct a smooth (non-holomorphic)  but globally defined map
$$
\mathcal{W}^{\mathbf{v}}_P:\mathcal{A}\rightarrow \mathrm{Sym}\ \mathcal{L}
$$
using regularization of Feynman diagrams built out by the propagator $P$ (the Szeg\"{o} kernel). Roughly speaking, the combination of coordinate change formulas of the regularization and $\mathrm{Sym}\ \mathcal{L}$ gives the correct one of the chiral algebra $\mathcal{A}$.

To define the normal ordering map $\mathcal{W}^{\mathbf{v}}=\mathcal{W}^{\mathbf{v}}_P$, we first need a lemma.

\begin{lem}
    For any section $v\in\Omega^{0,\bullet}(X,\mathcal{A})$, we can find $\tilde{v}\in \Omega^{0,\bullet}(X^n,\mathcal{L}^{\flat\boxtimes n}(*\Delta))$  such that
    $$
    \Vec{\mu}_{\mathcal{A}}(\tilde{v})=v.
    $$
    Here $\Vec{\mu}_{\mathcal{A}}$ is the iterated chiral product
    \begin{equation}\label{IteratedChiral}
    \Vec{\mu}_{\mathcal{A}}(f\cdot v_1\boxtimes \cdots \boxtimes v_n)=\mu_{\mathcal{A}}(\cdots \mu_{\mathcal{A}}(\mu_{\mathcal{A}}(f\cdot v_1\boxtimes v_2)\boxtimes v_3)\boxtimes \cdots \boxtimes v_n)
    \end{equation}
    for $f\cdot v_1\boxtimes \cdots \boxtimes v_n\in \Omega^{0,\bullet}(X^n,\mathcal{L}^{\flat\boxtimes n}(*\Delta))$.
\end{lem}
\begin{proof}
We use the partition of the unity. Let $\{\rho_i\}_{i\in I}$ be a partition of the unity which subordinate to an open cover $\{U_i\}_{i\in I}$ of $X$.

We can find $\tilde{v}_{i_1\cdots i_n}\in \Omega^{0,\bullet}(U_{i_1}\times \cdots U_{i_n},\mathcal{L}^{\flat\boxtimes n}(*\Delta))$ such that
$$
\Vec{\mu}(\tilde{v}_{i_1\cdots i_n})=v|_{U_{i_1}\cap\cdots\cap U_{i_n}}.
$$
Then we can take
$$
v=\sum \tilde{v}_{i_1\cdots i_n} \cdot \rho_{i_1}\boxtimes \cdots\boxtimes \rho_{i_n}$$
which satisfies $\Vec{\mu}(\tilde{v})=v$.
\end{proof}

By the above lemma, we define
\begin{equation}\label{NormalOrdering}
\mathcal{W}^{\mathbf{v}}(v;\tilde{v})=\Vec{\mu}_{\mathrm{Sym}}(e^{\partial_P}\tilde{v})\in \Delta_* \mathrm{Sym}\ \mathcal{L},\quad \tilde{v}\in \Omega^{0,\bullet}(X^n,\mathcal{L}^{\flat\boxtimes n}(*\Delta))
\end{equation}

where $\Vec{\mu}_{\mathcal{A}}(\tilde{v})=v$ and $\Vec{\mu}_{\mathrm{Sym}}$ is defined in the same way as in (\ref{IteratedChiral}). Note we view $e^{\partial_P}\tilde{v}\subset \Omega^{0,\bullet}(X^n,\ \mathcal{L}^{\flat\ \boxtimes n}(*\Delta) )$ as an element in $\Omega^{0,\bullet}(X^n,(\mathrm{Sym}\ \mathcal{L})^{\boxtimes n}(*\Delta) )$ and then apply the iterated chiral product $\vec{\mu}_{\mathrm{Sym}}$. See Fig. \ref{NormalOrder}.

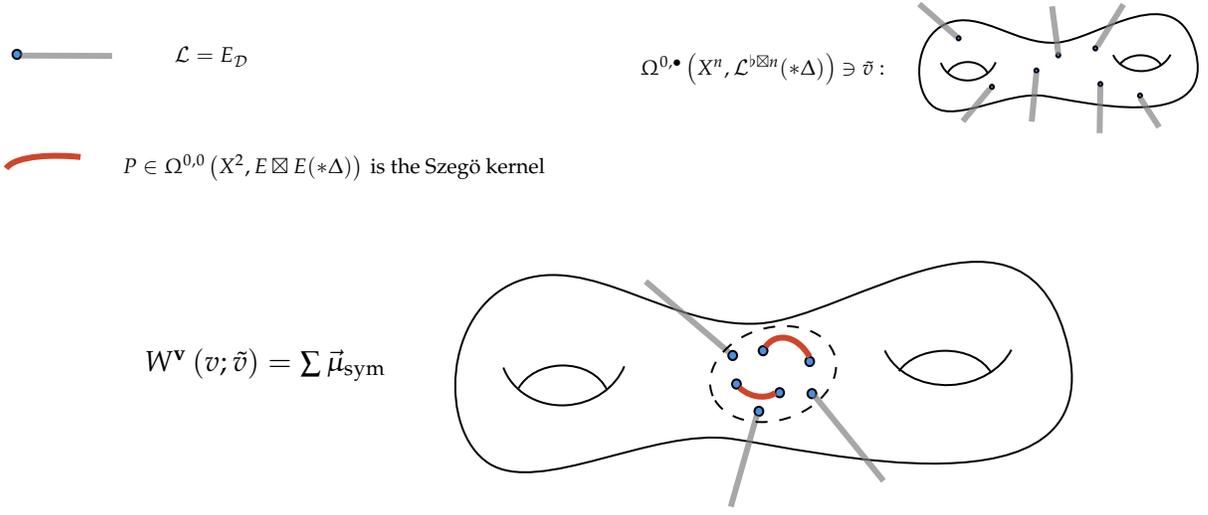
\begin{figure}
    \centering

\tikzset{every picture/.style={line width=0.75pt}} 

\begin{tikzpicture}[x=0.75pt,y=0.75pt,yscale=-1,xscale=1]

\draw  [fill={rgb, 255:red, 74; green, 144; blue, 226 }  ,fill opacity=1 ] (400.14,190.1) .. controls (400.14,188.82) and (401.13,187.79) .. (402.35,187.79) .. controls (403.58,187.79) and (404.57,188.82) .. (404.57,190.1) .. controls (404.57,191.38) and (403.58,192.42) .. (402.35,192.42) .. controls (401.13,192.42) and (400.14,191.38) .. (400.14,190.1) -- cycle ;
\draw  [fill={rgb, 255:red, 74; green, 144; blue, 226 }  ,fill opacity=1 ] (423.64,195.54) .. controls (423.64,194.26) and (424.63,193.22) .. (425.85,193.22) .. controls (427.07,193.22) and (428.06,194.26) .. (428.06,195.54) .. controls (428.06,196.82) and (427.07,197.86) .. (425.85,197.86) .. controls (424.63,197.86) and (423.64,196.82) .. (423.64,195.54) -- cycle ;
\draw   (277.97,155.33) .. controls (315.32,140.36) and (354.9,179.14) .. (400.3,176.37) .. controls (445.7,173.6) and (536.51,100.2) .. (556.3,188.83) .. controls (576.09,277.47) and (425.91,240.08) .. (386.33,234.54) .. controls (346.75,229) and (289.61,263.36) .. (264.1,247) .. controls (238.59,230.65) and (240.62,170.3) .. (277.97,155.33) -- cycle ;
\draw  [draw opacity=0] (527.57,189.61) .. controls (522.54,199.46) and (510.82,206.72) .. (496.94,207.4) .. controls (481.77,208.14) and (468.62,200.8) .. (463.76,189.97) -- (495.99,179.78) -- cycle ; \draw   (527.57,189.61) .. controls (522.54,199.46) and (510.82,206.72) .. (496.94,207.4) .. controls (481.77,208.14) and (468.62,200.8) .. (463.76,189.97) ;
\draw  [draw opacity=0] (470.91,200.44) .. controls (475.04,194.54) and (483.52,190.37) .. (493.38,190.14) .. controls (504.3,189.89) and (513.68,194.55) .. (517.4,201.32) -- (493.67,208.16) -- cycle ; \draw   (470.91,200.44) .. controls (475.04,194.54) and (483.52,190.37) .. (493.38,190.14) .. controls (504.3,189.89) and (513.68,194.55) .. (517.4,201.32) ;
\draw  [draw opacity=0] (332.35,197.98) .. controls (327.87,208.81) and (316.52,216.89) .. (302.99,217.63) .. controls (288.34,218.43) and (275.69,210.36) .. (271.37,198.59) -- (302.09,188.71) -- cycle ; \draw   (332.35,197.98) .. controls (327.87,208.81) and (316.52,216.89) .. (302.99,217.63) .. controls (288.34,218.43) and (275.69,210.36) .. (271.37,198.59) ;
\draw  [draw opacity=0] (279.01,209.06) .. controls (282.72,202.5) and (290.98,197.81) .. (300.63,197.56) .. controls (311.21,197.3) and (320.27,202.45) .. (323.57,209.85) -- (300.9,216.44) -- cycle ; \draw   (279.01,209.06) .. controls (282.72,202.5) and (290.98,197.81) .. (300.63,197.56) .. controls (311.21,197.3) and (320.27,202.45) .. (323.57,209.85) ;
\draw  [fill={rgb, 255:red, 74; green, 144; blue, 226 }  ,fill opacity=1 ] (398.02,220.69) .. controls (398.02,219.41) and (399.01,218.37) .. (400.24,218.37) .. controls (401.46,218.37) and (402.45,219.41) .. (402.45,220.69) .. controls (402.45,221.97) and (401.46,223.01) .. (400.24,223.01) .. controls (399.01,223.01) and (398.02,221.97) .. (398.02,220.69) -- cycle ;
\draw  [fill={rgb, 255:red, 74; green, 144; blue, 226 }  ,fill opacity=1 ] (384.81,192.51) .. controls (384.81,191.23) and (385.8,190.2) .. (387.02,190.2) .. controls (388.24,190.2) and (389.24,191.23) .. (389.24,192.51) .. controls (389.24,193.79) and (388.24,194.83) .. (387.02,194.83) .. controls (385.8,194.83) and (384.81,193.79) .. (384.81,192.51) -- cycle ;
\draw  [fill={rgb, 255:red, 74; green, 144; blue, 226 }  ,fill opacity=1 ] (408.48,211.29) .. controls (408.48,210.01) and (409.47,208.97) .. (410.69,208.97) .. controls (411.91,208.97) and (412.9,210.01) .. (412.9,211.29) .. controls (412.9,212.57) and (411.91,213.61) .. (410.69,213.61) .. controls (409.47,213.61) and (408.48,212.57) .. (408.48,211.29) -- cycle ;
\draw  [fill={rgb, 255:red, 74; green, 144; blue, 226 }  ,fill opacity=1 ] (424.76,211.76) .. controls (424.76,210.48) and (425.75,209.44) .. (426.97,209.44) .. controls (428.2,209.44) and (429.19,210.48) .. (429.19,211.76) .. controls (429.19,213.04) and (428.2,214.07) .. (426.97,214.07) .. controls (425.75,214.07) and (424.76,213.04) .. (424.76,211.76) -- cycle ;
\draw  [fill={rgb, 255:red, 74; green, 144; blue, 226 }  ,fill opacity=1 ] (386.71,206.91) .. controls (386.71,205.63) and (387.7,204.59) .. (388.93,204.59) .. controls (390.15,204.59) and (391.14,205.63) .. (391.14,206.91) .. controls (391.14,208.19) and (390.15,209.23) .. (388.93,209.23) .. controls (387.7,209.23) and (386.71,208.19) .. (386.71,206.91) -- cycle ;
\draw [color={rgb, 255:red, 128; green, 128; blue, 128 }  ,draw opacity=0.69 ][line width=2.25]    (343.11,155.12) -- (385.02,190.83) ;
\draw [color={rgb, 255:red, 206; green, 69; blue, 44 }  ,draw opacity=1 ][fill={rgb, 255:red, 255; green, 255; blue, 255 }  ,fill opacity=1 ][line width=2.25]    (403.57,188.1) .. controls (408.57,181.1) and (419.08,181.57) .. (425.85,193.22) ;
\draw [color={rgb, 255:red, 206; green, 69; blue, 44 }  ,draw opacity=1 ][fill={rgb, 255:red, 255; green, 255; blue, 255 }  ,fill opacity=1 ][line width=2.25]    (390.37,208.67) .. controls (396.87,214.88) and (405.87,213.88) .. (408.48,211.29) ;
\draw [color={rgb, 255:red, 128; green, 128; blue, 128 }  ,draw opacity=0.69 ][line width=2.25]    (428.19,212.76) -- (463.08,255.81) ;
\draw [color={rgb, 255:red, 128; green, 128; blue, 128 }  ,draw opacity=0.69 ][line width=2.25]    (399.19,222.76) -- (386.08,268.81) ;
\draw  [dash pattern={on 4.5pt off 4.5pt}] (377.02,213) .. controls (371.88,201.1) and (381.31,186.48) .. (398.09,180.35) .. controls (414.87,174.22) and (432.64,178.89) .. (437.78,190.79) .. controls (442.92,202.68) and (433.49,217.3) .. (416.71,223.43) .. controls (399.94,229.56) and (382.16,224.89) .. (377.02,213) -- cycle ;
\draw  [fill={rgb, 255:red, 74; green, 144; blue, 226 }  ,fill opacity=1 ] (550.22,40.96) .. controls (550.22,40.37) and (550.67,39.89) .. (551.22,39.89) .. controls (551.77,39.89) and (552.22,40.37) .. (552.22,40.96) .. controls (552.22,41.55) and (551.77,42.03) .. (551.22,42.03) .. controls (550.67,42.03) and (550.22,41.55) .. (550.22,40.96) -- cycle ;
\draw  [fill={rgb, 255:red, 74; green, 144; blue, 226 }  ,fill opacity=1 ] (568.97,37.47) .. controls (568.97,36.88) and (569.41,36.4) .. (569.97,36.4) .. controls (570.52,36.4) and (570.96,36.88) .. (570.96,37.47) .. controls (570.96,38.06) and (570.52,38.54) .. (569.97,38.54) .. controls (569.41,38.54) and (568.97,38.06) .. (568.97,37.47) -- cycle ;
\draw   (495.04,24.9) .. controls (511.91,17.99) and (529.79,35.9) .. (550.3,34.62) .. controls (570.8,33.34) and (611.82,-0.55) .. (620.76,40.37) .. controls (629.7,81.3) and (561.86,64.03) .. (543.99,61.48) .. controls (526.11,58.92) and (500.3,74.78) .. (488.77,67.23) .. controls (477.25,59.68) and (478.17,31.82) .. (495.04,24.9) -- cycle ;
\draw  [draw opacity=0] (607.83,40.64) .. controls (605.59,45.24) and (600.26,48.63) .. (593.95,48.95) .. controls (587.06,49.29) and (581.1,45.86) .. (578.93,40.82) -- (593.52,36.19) -- cycle ; \draw   (607.83,40.64) .. controls (605.59,45.24) and (600.26,48.63) .. (593.95,48.95) .. controls (587.06,49.29) and (581.1,45.86) .. (578.93,40.82) ;
\draw  [draw opacity=0] (582.15,45.8) .. controls (583.99,43.04) and (587.85,41.08) .. (592.34,40.98) .. controls (597.3,40.86) and (601.56,43.04) .. (603.22,46.19) -- (592.47,49.3) -- cycle ; \draw   (582.15,45.8) .. controls (583.99,43.04) and (587.85,41.08) .. (592.34,40.98) .. controls (597.3,40.86) and (601.56,43.04) .. (603.22,46.19) ;
\draw  [draw opacity=0] (519.64,44.5) .. controls (517.64,49.55) and (512.49,53.33) .. (506.34,53.67) .. controls (499.69,54.04) and (493.95,50.28) .. (492.03,44.81) -- (505.93,40.32) -- cycle ; \draw   (519.64,44.5) .. controls (517.64,49.55) and (512.49,53.33) .. (506.34,53.67) .. controls (499.69,54.04) and (493.95,50.28) .. (492.03,44.81) ;
\draw  [draw opacity=0] (495.47,49.78) .. controls (497.13,46.72) and (500.88,44.52) .. (505.27,44.4) .. controls (510.08,44.28) and (514.18,46.69) .. (515.66,50.13) -- (505.4,53.12) -- cycle ; \draw   (495.47,49.78) .. controls (497.13,46.72) and (500.88,44.52) .. (505.27,44.4) .. controls (510.08,44.28) and (514.18,46.69) .. (515.66,50.13) ;
\draw  [fill={rgb, 255:red, 74; green, 144; blue, 226 }  ,fill opacity=1 ] (516.74,56.93) .. controls (516.74,56.34) and (517.19,55.86) .. (517.74,55.86) .. controls (518.29,55.86) and (518.74,56.34) .. (518.74,56.93) .. controls (518.74,57.52) and (518.29,58) .. (517.74,58) .. controls (517.19,58) and (516.74,57.52) .. (516.74,56.93) -- cycle ;
\draw  [fill={rgb, 255:red, 74; green, 144; blue, 226 }  ,fill opacity=1 ] (499.93,32.38) .. controls (499.93,31.79) and (500.38,31.31) .. (500.93,31.31) .. controls (501.48,31.31) and (501.93,31.79) .. (501.93,32.38) .. controls (501.93,32.97) and (501.48,33.45) .. (500.93,33.45) .. controls (500.38,33.45) and (499.93,32.97) .. (499.93,32.38) -- cycle ;
\draw  [fill={rgb, 255:red, 74; green, 144; blue, 226 }  ,fill opacity=1 ] (591.48,61.36) .. controls (591.48,60.77) and (591.93,60.29) .. (592.48,60.29) .. controls (593.03,60.29) and (593.48,60.77) .. (593.48,61.36) .. controls (593.48,61.95) and (593.03,62.43) .. (592.48,62.43) .. controls (591.93,62.43) and (591.48,61.95) .. (591.48,61.36) -- cycle ;
\draw  [fill={rgb, 255:red, 74; green, 144; blue, 226 }  ,fill opacity=1 ] (539.19,48.72) .. controls (539.19,48.13) and (539.64,47.65) .. (540.19,47.65) .. controls (540.74,47.65) and (541.19,48.13) .. (541.19,48.72) .. controls (541.19,49.31) and (540.74,49.79) .. (540.19,49.79) .. controls (539.64,49.79) and (539.19,49.31) .. (539.19,48.72) -- cycle ;
\draw  [fill={rgb, 255:red, 74; green, 144; blue, 226 }  ,fill opacity=1 ] (571.46,55.54) .. controls (571.46,54.95) and (571.91,54.47) .. (572.46,54.47) .. controls (573.01,54.47) and (573.46,54.95) .. (573.46,55.54) .. controls (573.46,56.13) and (573.01,56.61) .. (572.46,56.61) .. controls (571.91,56.61) and (571.46,56.13) .. (571.46,55.54) -- cycle ;
\draw [color={rgb, 255:red, 128; green, 128; blue, 128 }  ,draw opacity=0.69 ][line width=2.25]    (481.09,15.11) -- (500.03,31.6) ;
\draw [color={rgb, 255:red, 128; green, 128; blue, 128 }  ,draw opacity=0.68 ][line width=2.25]    (516.74,57.93) -- (503.08,74.29) ;
\draw [color={rgb, 255:red, 128; green, 128; blue, 128 }  ,draw opacity=0.68 ][line width=2.25]    (592.48,62.43) -- (602.08,77.29) ;
\draw [color={rgb, 255:red, 128; green, 128; blue, 128 }  ,draw opacity=0.68 ][line width=2.25]    (572.46,56.61) -- (572.08,80.29) ;
\draw [color={rgb, 255:red, 128; green, 128; blue, 128 }  ,draw opacity=0.68 ][line width=2.25]    (540.19,49.79) -- (537.81,74.47) ;
\draw [color={rgb, 255:red, 128; green, 128; blue, 128 }  ,draw opacity=0.68 ][line width=2.25]    (548.08,16.29) -- (551.3,40.62) ;
\draw [color={rgb, 255:red, 128; green, 128; blue, 128 }  ,draw opacity=0.68 ][line width=2.25]    (570.96,36.47) -- (584.08,15.29) ;
\draw  [fill={rgb, 255:red, 74; green, 144; blue, 226 }  ,fill opacity=1 ] (23.81,40.51) .. controls (23.81,39.23) and (24.8,38.2) .. (26.02,38.2) .. controls (27.24,38.2) and (28.24,39.23) .. (28.24,40.51) .. controls (28.24,41.79) and (27.24,42.83) .. (26.02,42.83) .. controls (24.8,42.83) and (23.81,41.79) .. (23.81,40.51) -- cycle ;
\draw [color={rgb, 255:red, 128; green, 128; blue, 128 }  ,draw opacity=0.68 ][line width=2.25]    (28.11,41.12) -- (74.08,41.29) ;
\draw [color={rgb, 255:red, 206; green, 69; blue, 44 }  ,draw opacity=1 ][fill={rgb, 255:red, 255; green, 255; blue, 255 }  ,fill opacity=1 ][line width=2.25]    (20.57,98.1) .. controls (25.57,91.1) and (46.57,91.1) .. (58.12,92.35) ;

\draw (89,187.4) node [anchor=north west][inner sep=0.75pt]    {$W^{\mathbf{v}}\left( v;\tilde{v}\right) =\sum \vec{\mu }_{\mathrm{sym}}$};
\draw (339,38.4) node [anchor=north west][inner sep=0.75pt]  [font=\scriptsize]  {$\Omega ^{0,\bullet }\left( X^{n} ,\mathcal{L}^{\flat\boxtimes n}( *\Delta )\right) \ni \tilde{v} :$};
\draw (104,35.4) node [anchor=north west][inner sep=0.75pt]  [font=\scriptsize]  {$\mathcal{L} =E_{\mathcal{D}}$};
\draw (78,89.4) node [anchor=north west][inner sep=0.75pt]  [font=\scriptsize]  {$P\in \Omega ^{0,0}\left( X^{2} ,E\boxtimes E( *\Delta )\right)\text{ is the Szeg\"{o} kernel}$};

\end{tikzpicture}
    \caption{$\mathcal{W}^{\mathbf{v}}(v;\tilde{v})$}
    \label{NormalOrder}
\end{figure}

In the rest of this section, we prove that $\mathcal{W}^{\mathbf{v}}(v;\tilde{v})\in  \mathrm{Sym}\ \mathcal{L}\subset \Delta_*\mathrm{Sym}\ \mathcal{L}$ and it is independent of the choice of $\tilde{v}$. The proof uses the Wick theorem in the form of chiral algebras and results in \cite{gui2021elliptic}. We first introduce some notations.

Given a possibly singular symmetric bisection $K\in \Omega^{0,\bullet}(X^2,E\boxtimes E(*\Delta))^{\sigma_2}$, we define an operator $e^{K_{\mathbf{sing}}}$ acting on $\Omega^{0,\bullet}(X^n,(\mathrm{Sym}\ \mathcal{L})^{\boxtimes n}(*\Delta) )$. We will construct this operator on the sections of left $\mathcal{D}_{X^n}$-modules $\Omega^{0,\bullet}(X^n,(\mathrm{Sym}\ \mathcal{L}^l)^{\boxtimes n}(*\Delta) )$ and then transfer the construction back to the world of right $\mathcal{D}_{X^n}$-modules.

We describe this operator on the homogeneous components, define
$$
e^{K_{\mathbf{Sing}}}:\Omega^{0,\bullet}(X^n,\mathrm{Sym}^{k_1}\mathcal{L}^{l}\boxtimes \cdots \boxtimes \mathrm{Sym}^{k_n}\mathcal{L}^{l} (*\Delta))\rightarrow \Omega^{0,\bullet}(X^n,\mathrm{Sym}^{\leq k_1}\mathcal{L}^l\boxtimes \cdots \boxtimes \mathrm{Sym}^{\leq k_n}\mathcal{L}^l ((*\Delta)))
$$
as follows. We first define the operator $(\partial_K)_{k_1,\dots,k_n}$ acting on
$$
(\mathcal{L}^l\oplus \mathcal{O}_X)^{\boxtimes k_1}\boxtimes \cdots\boxtimes (\mathcal{L}^l\oplus \mathcal{O}_X)^{\boxtimes k_n}.
$$
Formally write $K=e_1\boxtimes e_2$, then define
$$
(\partial_K)_{k_1,\dots,k_n}=\sum_{1\leq  i<j\leq n} 1^{\boxtimes k_1}\boxtimes \cdots\boxtimes \Delta_{i*}\partial_{e_1} \boxtimes\cdots \boxtimes \Delta_{j*}\partial_{e_2}\boxtimes \cdots\boxtimes 1^{\boxtimes k_n},
$$
where $\Delta_{i*}\partial_e$ acts on $(\mathcal{L}^l\oplus \mathcal{O}_X)^{\boxtimes k_i}$
$$
\Delta_{i*}\partial_e=\sum_{r=1}^{k_i} 1\boxtimes \cdots \boxtimes\underbrace{\partial_e}_{\text{s-th}}\boxtimes \cdots \boxtimes 1
$$
Given $v\in \Omega^{0,\bullet}(X^n,\mathrm{Sym}^{k_1}\mathcal{L}^{l}\boxtimes \cdots \boxtimes \mathrm{Sym}^{k_n}\mathcal{L}^{l} (*\Delta)) $. We can find
$$
v'\in \Omega^{0,\bullet}\left(X^I,(\mathcal{L}^l)^{\boxtimes k_1}\boxtimes \cdots\boxtimes (\mathcal{L}^l)^{\boxtimes k_n}(*\Delta_{\mathbf{trans}})\right).
$$
 such that
 \begin{equation}\label{Pullback}
(\Delta^*_1\boxtimes\cdots\Delta^*_n)(v')=v.
 \end{equation}

Here $I$ is an index set such that $|I|=k_1+\dots+k_n$ and let $\pi:I\twoheadrightarrow \{1,\dots,n\}$ be a surjective map. Then the diagonal embedding $\Delta^{(\pi)}:X^n\hookrightarrow X^I$ can be written as $\Delta^{(\pi)}=\Delta_1\times\cdots\times \Delta_n$ where $\Delta_i:X\rightarrow X^{k_i}=X^{\pi^{-1}(i)}$.  The transversal diagonal $\Delta_{\mathbf{trans}}$ is defined as follows
$$
\Delta_{\mathbf{trans}}=\{(\dots,x_i,\dots)\in X^I|x_i=x_j \ \text{if} \ \pi(i)\neq \pi(j)\}.
$$

In (\ref{Pullback}), we use the fact that $\Delta_i^*(\mathcal{L}^{l})^{\boxtimes k_i}$ is isomorphic to $\mathrm{Sym}^{k_i}(\mathcal{L}^l)$ after symmetrization (we will omit the symmetrization notation).  We define
$$
\boxed{e^{K_{\mathbf{Sing}}}(v):=(\Delta^*_1\boxtimes\cdots\Delta^*_n)(e^{(\partial_k)_{k_1,\dots,k_n}}v').}
$$

The above definition does not depend on the choice $v'$. Suppose that
$$
(\Delta^*_1\boxtimes\cdots\Delta^*_n)(v'-v'')=0.
$$
The one can find $i\in \{1,\dots,n\}$ and $i_1,i_2\in \pi^{-1}(i)$  such that locally we have
$$
v'-v''=w\cdot (z_{i_1}-z_{i_2}),\quad w \ \text{is a local section of } \Omega^{0,\bullet}\otimes \left((\mathcal{L}^l)^{\boxtimes k_1}\boxtimes \cdots\boxtimes (\mathcal{L}^l)^{\boxtimes k_n}(*\Delta_{\mathbf{trans}})\right).
$$
Here $z_{i_s},s=1,\dots,k_i=|\pi^{-1}(i)|$ are local coordinates of $X^{\pi^{-1}(i)}=X^{k_i}$. The reason that we write the multiplication of $(z_{i_1}-z_{i_2})$ on the right is that the $\mathcal{O}_X$-module structure of $\mathcal{L}^l=E\otimes \mathcal{D}_X\otimes \omega_X^{-1}$ is the one on the right (coming from $\mathcal{O}_X\subset \mathcal{D}_X^{\mathrm{op}}\simeq \omega_X\otimes \mathcal{D}_X\otimes\omega^{-1}_X$). Then by the fact that $\partial_K$ is compatible with $\mathcal{D}$-module structure, we have
$$
e^{(\partial_K)_{k_1,\dots,k_n}}(v'-v'')=(e^{(\partial_K)_{k_1,\dots,k_n}}w)\cdot (z_{i_1}-z_{i_2}),
$$
which implies that

$$
(\Delta^*_1\boxtimes\cdots\Delta^*_n)(e^{(\partial_K)_{k_1,\dots,k_n}}v')=(\Delta^*_1\boxtimes\cdots\Delta^*_n)(e^{(\partial_K)_{k_1,\dots,k_n}}v'').
$$

Finally we see that $e^{K_{\mathbf{sing}}}$ is well defined and compatible with $\mathcal{D}_{X^n}$-module structure.

Now for a regular symmetric bisection $Q\in \Omega^{0,\bullet}(X^2,E\boxtimes E)^{\sigma_2}$. we define another operation $e^{Q_{\mathbf{reg}}}$ acting on $\Omega^{0,\bullet}(X^n,(\mathrm{Sym}\ \mathcal{L})^{\boxtimes n}(*\Delta) )$.

As before, we describe its action on the components
$$
e^{Q_{\mathbf{reg}}}:\Omega^{0,\bullet}(X^n,\mathrm{Sym}^{k_1}\mathcal{L}^{l}\boxtimes \cdots \boxtimes \mathrm{Sym}^{k_n}\mathcal{L}^{l} (*\Delta))\rightarrow \Omega^{0,\bullet}(X^n,\mathrm{Sym}^{\leq k_1}\mathcal{L}^l\boxtimes \cdots \boxtimes \mathrm{Sym}^{\leq k_n}\mathcal{L}^l ((*\Delta)))
$$

We first define the operator $(\partial_Q)_{k_1+\cdots+k_n}$ acting on
$$
(\mathcal{L}^l\oplus \mathcal{O}_X)^{\boxtimes k_1}\boxtimes \cdots\boxtimes (\mathcal{L}^l\oplus \mathcal{O}_X)^{\boxtimes k_n}=(\mathcal{L}^l\oplus \mathcal{O}_X)^{\boxtimes k_1+\cdots+k_n}
$$
by

$$
(\partial_Q)_{k_1+\cdots+k_n}=\sum_{1\leq i<j\leq k_1+\cdots +k_n} 1\boxtimes \cdots\boxtimes \underbrace{\partial_{e_1}}_{\text{i-th}} \boxtimes\cdots \boxtimes \underbrace{\partial_{e_2}}_{\text{j-th}}\boxtimes \cdots\boxtimes 1,
$$
where as before we formally write $Q=e_1\boxtimes e_2$. Define
$$
\boxed{e^{Q_{\mathbf{reg}}}(v):=(\Delta^*_1\boxtimes\cdots\boxtimes\Delta^*_n)(e^{(\partial_Q)_{k_1+\cdots+k_n}}v'),}
$$
here we use the same notations $v$ and $v'$ as in (\ref{Pullback}).

Then by definition
\begin{equation}\label{Qpullback}
e^{Q_{\mathbf{reg}}}\left(\Delta^*(v_1\boxtimes
 v_2)\right)=\Delta^*\left(e^{Q_{\mathbf{reg}}} (v_1\boxtimes
 v_2)\right).
\end{equation}

Here we summarize the above construction.

\begin{itemize}
  \item The operator $e^{K_{\mathbf{sing}}}$ is constructed from a singular kernel function $K\in \Omega^{0,\bullet}(X^2,E\boxtimes E(*\Delta))^{\sigma_2}$. It does not have self-loops (see Fig. \ref{Singular}).
  \item The operator $e^{Q_{\mathbf{reg}}}$ is constructed from a regular kernel function $Q\in \Omega^{0,\bullet}(X^2,E\boxtimes E)^{\sigma_2}$. It contains self-loops (see Fig. \ref{Regular}).
\end{itemize}
\begin{figure}
    \centering

\tikzset{every picture/.style={line width=0.75pt}} 

\begin{tikzpicture}[x=0.75pt,y=0.75pt,yscale=-1,xscale=1]

\draw  [fill={rgb, 255:red, 74; green, 144; blue, 226 }  ,fill opacity=1 ] (449.64,192.54) .. controls (449.64,191.26) and (450.63,190.22) .. (451.85,190.22) .. controls (453.07,190.22) and (454.06,191.26) .. (454.06,192.54) .. controls (454.06,193.82) and (453.07,194.86) .. (451.85,194.86) .. controls (450.63,194.86) and (449.64,193.82) .. (449.64,192.54) -- cycle ;
\draw   (292.97,169.33) .. controls (330.32,154.36) and (369.9,193.14) .. (415.3,190.37) .. controls (460.7,187.6) and (551.51,114.2) .. (571.3,202.83) .. controls (591.09,291.47) and (440.91,254.08) .. (401.33,248.54) .. controls (361.75,243) and (304.61,277.36) .. (279.1,261) .. controls (253.59,244.65) and (255.62,184.3) .. (292.97,169.33) -- cycle ;
\draw  [draw opacity=0] (542.57,203.61) .. controls (537.54,213.46) and (525.82,220.72) .. (511.94,221.4) .. controls (496.77,222.14) and (483.62,214.8) .. (478.76,203.97) -- (510.99,193.78) -- cycle ; \draw   (542.57,203.61) .. controls (537.54,213.46) and (525.82,220.72) .. (511.94,221.4) .. controls (496.77,222.14) and (483.62,214.8) .. (478.76,203.97) ;
\draw  [draw opacity=0] (485.91,214.44) .. controls (490.04,208.54) and (498.52,204.37) .. (508.38,204.14) .. controls (519.3,203.89) and (528.68,208.55) .. (532.4,215.32) -- (508.67,222.16) -- cycle ; \draw   (485.91,214.44) .. controls (490.04,208.54) and (498.52,204.37) .. (508.38,204.14) .. controls (519.3,203.89) and (528.68,208.55) .. (532.4,215.32) ;
\draw  [draw opacity=0] (347.35,211.98) .. controls (342.87,222.81) and (331.52,230.89) .. (317.99,231.63) .. controls (303.34,232.43) and (290.69,224.36) .. (286.37,212.59) -- (317.09,202.71) -- cycle ; \draw   (347.35,211.98) .. controls (342.87,222.81) and (331.52,230.89) .. (317.99,231.63) .. controls (303.34,232.43) and (290.69,224.36) .. (286.37,212.59) ;
\draw  [draw opacity=0] (294.01,223.06) .. controls (297.72,216.5) and (305.98,211.81) .. (315.63,211.56) .. controls (326.21,211.3) and (335.27,216.45) .. (338.57,223.85) -- (315.9,230.44) -- cycle ; \draw   (294.01,223.06) .. controls (297.72,216.5) and (305.98,211.81) .. (315.63,211.56) .. controls (326.21,211.3) and (335.27,216.45) .. (338.57,223.85) ;
\draw  [fill={rgb, 255:red, 74; green, 144; blue, 226 }  ,fill opacity=1 ] (335.02,241.69) .. controls (335.02,240.41) and (336.01,239.37) .. (337.24,239.37) .. controls (338.46,239.37) and (339.45,240.41) .. (339.45,241.69) .. controls (339.45,242.97) and (338.46,244.01) .. (337.24,244.01) .. controls (336.01,244.01) and (335.02,242.97) .. (335.02,241.69) -- cycle ;
\draw  [fill={rgb, 255:red, 74; green, 144; blue, 226 }  ,fill opacity=1 ] (301.81,183.51) .. controls (301.81,182.23) and (302.8,181.2) .. (304.02,181.2) .. controls (305.24,181.2) and (306.24,182.23) .. (306.24,183.51) .. controls (306.24,184.79) and (305.24,185.83) .. (304.02,185.83) .. controls (302.8,185.83) and (301.81,184.79) .. (301.81,183.51) -- cycle ;
\draw  [fill={rgb, 255:red, 74; green, 144; blue, 226 }  ,fill opacity=1 ] (459.76,244.76) .. controls (459.76,243.48) and (460.75,242.44) .. (461.97,242.44) .. controls (463.2,242.44) and (464.19,243.48) .. (464.19,244.76) .. controls (464.19,246.04) and (463.2,247.07) .. (461.97,247.07) .. controls (460.75,247.07) and (459.76,246.04) .. (459.76,244.76) -- cycle ;
\draw [color={rgb, 255:red, 128; green, 128; blue, 128 }  ,draw opacity=0.68 ][line width=2.25]    (267.08,154.67) -- (302.02,181.83) ;
\draw [color={rgb, 255:red, 128; green, 128; blue, 128 }  ,draw opacity=0.67 ][line width=2.25]    (461.97,247.07) -- (463.08,290.67) ;
\draw [color={rgb, 255:red, 128; green, 128; blue, 128 }  ,draw opacity=0.67 ][line width=2.25]    (336.19,243.76) -- (315.08,278.67) ;
\draw  [fill={rgb, 255:red, 74; green, 144; blue, 226 }  ,fill opacity=1 ] (583.97,51.47) .. controls (583.97,50.88) and (584.41,50.4) .. (584.97,50.4) .. controls (585.52,50.4) and (585.96,50.88) .. (585.96,51.47) .. controls (585.96,52.06) and (585.52,52.54) .. (584.97,52.54) .. controls (584.41,52.54) and (583.97,52.06) .. (583.97,51.47) -- cycle ;
\draw   (510.04,38.9) .. controls (526.91,31.99) and (544.79,49.9) .. (565.3,48.62) .. controls (585.8,47.34) and (626.82,13.45) .. (635.76,54.37) .. controls (644.7,95.3) and (576.86,78.03) .. (558.99,75.48) .. controls (541.11,72.92) and (515.3,88.78) .. (503.77,81.23) .. controls (492.25,73.68) and (493.17,45.82) .. (510.04,38.9) -- cycle ;
\draw  [draw opacity=0] (622.83,54.64) .. controls (620.59,59.24) and (615.26,62.63) .. (608.95,62.95) .. controls (602.06,63.29) and (596.1,59.86) .. (593.93,54.82) -- (608.52,50.19) -- cycle ; \draw   (622.83,54.64) .. controls (620.59,59.24) and (615.26,62.63) .. (608.95,62.95) .. controls (602.06,63.29) and (596.1,59.86) .. (593.93,54.82) ;
\draw  [draw opacity=0] (597.15,59.8) .. controls (598.99,57.04) and (602.85,55.08) .. (607.34,54.98) .. controls (612.3,54.86) and (616.56,57.04) .. (618.22,60.19) -- (607.47,63.3) -- cycle ; \draw   (597.15,59.8) .. controls (598.99,57.04) and (602.85,55.08) .. (607.34,54.98) .. controls (612.3,54.86) and (616.56,57.04) .. (618.22,60.19) ;
\draw  [draw opacity=0] (534.64,58.5) .. controls (532.64,63.55) and (527.49,67.33) .. (521.34,67.67) .. controls (514.69,68.04) and (508.95,64.28) .. (507.03,58.81) -- (520.93,54.32) -- cycle ; \draw   (534.64,58.5) .. controls (532.64,63.55) and (527.49,67.33) .. (521.34,67.67) .. controls (514.69,68.04) and (508.95,64.28) .. (507.03,58.81) ;
\draw  [draw opacity=0] (510.47,63.78) .. controls (512.13,60.72) and (515.88,58.52) .. (520.27,58.4) .. controls (525.08,58.28) and (529.18,60.69) .. (530.66,64.13) -- (520.4,67.12) -- cycle ; \draw   (510.47,63.78) .. controls (512.13,60.72) and (515.88,58.52) .. (520.27,58.4) .. controls (525.08,58.28) and (529.18,60.69) .. (530.66,64.13) ;
\draw  [fill={rgb, 255:red, 74; green, 144; blue, 226 }  ,fill opacity=1 ] (531.74,70.93) .. controls (531.74,70.34) and (532.19,69.86) .. (532.74,69.86) .. controls (533.29,69.86) and (533.74,70.34) .. (533.74,70.93) .. controls (533.74,71.52) and (533.29,72) .. (532.74,72) .. controls (532.19,72) and (531.74,71.52) .. (531.74,70.93) -- cycle ;
\draw  [fill={rgb, 255:red, 74; green, 144; blue, 226 }  ,fill opacity=1 ] (514.93,46.38) .. controls (514.93,45.79) and (515.38,45.31) .. (515.93,45.31) .. controls (516.48,45.31) and (516.93,45.79) .. (516.93,46.38) .. controls (516.93,46.97) and (516.48,47.45) .. (515.93,47.45) .. controls (515.38,47.45) and (514.93,46.97) .. (514.93,46.38) -- cycle ;
\draw  [fill={rgb, 255:red, 74; green, 144; blue, 226 }  ,fill opacity=1 ] (586.46,69.54) .. controls (586.46,68.95) and (586.91,68.47) .. (587.46,68.47) .. controls (588.01,68.47) and (588.46,68.95) .. (588.46,69.54) .. controls (588.46,70.13) and (588.01,70.61) .. (587.46,70.61) .. controls (586.91,70.61) and (586.46,70.13) .. (586.46,69.54) -- cycle ;
\draw [color={rgb, 255:red, 128; green, 128; blue, 128 }  ,draw opacity=0.69 ][line width=2.25]    (504.08,35.41) -- (515.03,45.6) ;
\draw [color={rgb, 255:red, 128; green, 128; blue, 128 }  ,draw opacity=0.68 ][line width=2.25]    (531.74,71.93) -- (518.08,88.29) ;
\draw [color={rgb, 255:red, 128; green, 128; blue, 128 }  ,draw opacity=0.68 ][line width=2.25]    (587.46,70.61) -- (585.08,79.41) ;
\draw [color={rgb, 255:red, 128; green, 128; blue, 128 }  ,draw opacity=0.68 ][line width=2.25]    (584.96,50.47) -- (597.08,40.41) ;
\draw  [fill={rgb, 255:red, 74; green, 144; blue, 226 }  ,fill opacity=1 ] (49.81,57.51) .. controls (49.81,56.23) and (50.8,55.2) .. (52.02,55.2) .. controls (53.24,55.2) and (54.24,56.23) .. (54.24,57.51) .. controls (54.24,58.79) and (53.24,59.83) .. (52.02,59.83) .. controls (50.8,59.83) and (49.81,58.79) .. (49.81,57.51) -- cycle ;
\draw [color={rgb, 255:red, 128; green, 128; blue, 128 }  ,draw opacity=0.68 ][line width=2.25]    (54.11,58.12) -- (90.08,56.29) ;
\draw [color={rgb, 255:red, 128; green, 128; blue, 128 }  ,draw opacity=0.68 ][line width=2.25]    (53.11,56.12) -- (70.08,23.67) ;
\draw [color={rgb, 255:red, 128; green, 128; blue, 128 }  ,draw opacity=0.68 ][line width=2.25]    (82.08,80.67) -- (53.02,58.83) ;
\draw [color={rgb, 255:red, 128; green, 128; blue, 128 }  ,draw opacity=0.68 ][line width=2.25]    (35.04,90.64) -- (52.02,58.2) ;
\draw [color={rgb, 255:red, 128; green, 128; blue, 128 }  ,draw opacity=0.68 ][line width=2.25]    (42.08,20.67) -- (51.02,55.2) ;
\draw [color={rgb, 255:red, 128; green, 128; blue, 128 }  ,draw opacity=0.66 ][line width=2.25]    (482.08,155.67) -- (453.02,191.83) ;
\draw [color={rgb, 255:red, 144; green, 19; blue, 254 }  ,draw opacity=1 ][fill={rgb, 255:red, 255; green, 255; blue, 255 }  ,fill opacity=1 ][line width=2.25]    (40.57,118.1) .. controls (45.57,111.1) and (66.57,111.1) .. (78.12,112.35) ;
\draw [color={rgb, 255:red, 144; green, 19; blue, 254 }  ,draw opacity=1 ][fill={rgb, 255:red, 255; green, 255; blue, 255 }  ,fill opacity=1 ][line width=2.25]    (307.24,184.51) .. controls (312.24,177.51) and (370.08,216.41) .. (449.64,193.54) ;
\draw [color={rgb, 255:red, 144; green, 19; blue, 254 }  ,draw opacity=1 ][fill={rgb, 255:red, 255; green, 255; blue, 255 }  ,fill opacity=1 ][line width=2.25]    (450.64,194.54) .. controls (435.08,214.41) and (438.08,234.41) .. (459.76,243.76) ;
\draw [color={rgb, 255:red, 144; green, 19; blue, 254 }  ,draw opacity=1 ][fill={rgb, 255:red, 255; green, 255; blue, 255 }  ,fill opacity=1 ][line width=2.25]    (454.06,193.54) .. controls (471.51,195.41) and (478.51,232.41) .. (464.19,242.76) ;
\draw [color={rgb, 255:red, 128; green, 128; blue, 128 }  ,draw opacity=0.69 ][line width=2.25]    (517.09,47.11) -- (531.08,49.41) ;
\draw [color={rgb, 255:red, 128; green, 128; blue, 128 }  ,draw opacity=0.69 ][line width=2.25]    (505.75,69.63) -- (531.74,70.93) ;
\draw [color={rgb, 255:red, 128; green, 128; blue, 128 }  ,draw opacity=0.69 ][line width=2.25]    (506.08,75.41) -- (531.74,71.93) ;
\draw [color={rgb, 255:red, 128; green, 128; blue, 128 }  ,draw opacity=0.69 ][line width=2.25]    (572.08,52.41) -- (584.29,51.27) ;
\draw [color={rgb, 255:red, 128; green, 128; blue, 128 }  ,draw opacity=0.69 ][line width=2.25]    (588.87,68.87) -- (603.08,66.41) ;
\draw [color={rgb, 255:red, 128; green, 128; blue, 128 }  ,draw opacity=0.69 ][line width=2.25]    (588.46,71.54) -- (604.08,75.41) ;
\draw [color={rgb, 255:red, 128; green, 128; blue, 128 }  ,draw opacity=0.69 ][line width=2.25]    (576.08,67.41) -- (586.46,69.54) ;
\draw [color={rgb, 255:red, 128; green, 128; blue, 128 }  ,draw opacity=0.69 ][line width=2.25]    (587.46,68.54) -- (591.08,60.41) ;
\draw [color={rgb, 255:red, 128; green, 128; blue, 128 }  ,draw opacity=0.69 ][line width=2.25]    (583.08,61.41) -- (584.97,52.54) ;
\draw [color={rgb, 255:red, 128; green, 128; blue, 128 }  ,draw opacity=0.69 ][line width=2.25]    (585.96,51.47) -- (598.08,50.41) ;
\draw [color={rgb, 255:red, 128; green, 128; blue, 128 }  ,draw opacity=0.67 ][line width=2.25]    (335.19,240.08) -- (292.08,239.12) ;
\draw [color={rgb, 255:red, 128; green, 128; blue, 128 }  ,draw opacity=0.67 ][line width=2.25]    (336.19,242.76) -- (292.08,251.12) ;
\draw [color={rgb, 255:red, 128; green, 128; blue, 128 }  ,draw opacity=0.67 ][line width=2.25]    (464.19,245.76) -- (510.3,253.48) ;
\draw [color={rgb, 255:red, 128; green, 128; blue, 128 }  ,draw opacity=0.67 ][line width=2.25]    (465.19,243.76) -- (515.08,228.12) ;

\draw (155,201.4) node [anchor=north west][inner sep=0.75pt]    {$e^{K_{\mathbf{sing}}} =\sum $};
\draw (334,52.4) node [anchor=north west][inner sep=0.75pt]  [font=\scriptsize]  {$\Omega ^{0,\bullet }\left( X^{n} ,( \mathrm{Sym}\ \mathcal{L})^{\boxtimes n}( *\Delta )\right) \ni v:$};
\draw (119,49.4) node [anchor=north west][inner sep=0.75pt]  [font=\scriptsize]  {$\mathrm{Sym}\ \mathcal{L}$};
\draw (93,103.4) node [anchor=north west][inner sep=0.75pt]  [font=\scriptsize]  {$K\in \Omega ^{0,\bullet }\left( X^{2} ,E\boxtimes E( *\Delta )\right)^{\sigma _{2}}\text{ }$};

\end{tikzpicture}
    \caption{$e^{K_{\mathbf{sing}}}$}
    \label{Singular}
\end{figure}

\begin{figure}
    \centering

\tikzset{every picture/.style={line width=0.75pt}} 

\begin{tikzpicture}[x=0.75pt,y=0.75pt,yscale=-1,xscale=1]

\draw  [fill={rgb, 255:red, 74; green, 144; blue, 226 }  ,fill opacity=1 ] (449.64,192.54) .. controls (449.64,191.26) and (450.63,190.22) .. (451.85,190.22) .. controls (453.07,190.22) and (454.06,191.26) .. (454.06,192.54) .. controls (454.06,193.82) and (453.07,194.86) .. (451.85,194.86) .. controls (450.63,194.86) and (449.64,193.82) .. (449.64,192.54) -- cycle ;
\draw   (292.97,169.33) .. controls (330.32,154.36) and (369.9,193.14) .. (415.3,190.37) .. controls (460.7,187.6) and (551.51,114.2) .. (571.3,202.83) .. controls (591.09,291.47) and (440.91,254.08) .. (401.33,248.54) .. controls (361.75,243) and (304.61,277.36) .. (279.1,261) .. controls (253.59,244.65) and (255.62,184.3) .. (292.97,169.33) -- cycle ;
\draw  [draw opacity=0] (542.57,203.61) .. controls (537.54,213.46) and (525.82,220.72) .. (511.94,221.4) .. controls (496.77,222.14) and (483.62,214.8) .. (478.76,203.97) -- (510.99,193.78) -- cycle ; \draw   (542.57,203.61) .. controls (537.54,213.46) and (525.82,220.72) .. (511.94,221.4) .. controls (496.77,222.14) and (483.62,214.8) .. (478.76,203.97) ;
\draw  [draw opacity=0] (485.91,214.44) .. controls (490.04,208.54) and (498.52,204.37) .. (508.38,204.14) .. controls (519.3,203.89) and (528.68,208.55) .. (532.4,215.32) -- (508.67,222.16) -- cycle ; \draw   (485.91,214.44) .. controls (490.04,208.54) and (498.52,204.37) .. (508.38,204.14) .. controls (519.3,203.89) and (528.68,208.55) .. (532.4,215.32) ;
\draw  [draw opacity=0] (347.35,211.98) .. controls (342.87,222.81) and (331.52,230.89) .. (317.99,231.63) .. controls (303.34,232.43) and (290.69,224.36) .. (286.37,212.59) -- (317.09,202.71) -- cycle ; \draw   (347.35,211.98) .. controls (342.87,222.81) and (331.52,230.89) .. (317.99,231.63) .. controls (303.34,232.43) and (290.69,224.36) .. (286.37,212.59) ;
\draw  [draw opacity=0] (294.01,223.06) .. controls (297.72,216.5) and (305.98,211.81) .. (315.63,211.56) .. controls (326.21,211.3) and (335.27,216.45) .. (338.57,223.85) -- (315.9,230.44) -- cycle ; \draw   (294.01,223.06) .. controls (297.72,216.5) and (305.98,211.81) .. (315.63,211.56) .. controls (326.21,211.3) and (335.27,216.45) .. (338.57,223.85) ;
\draw  [fill={rgb, 255:red, 74; green, 144; blue, 226 }  ,fill opacity=1 ] (335.02,241.69) .. controls (335.02,240.41) and (336.01,239.37) .. (337.24,239.37) .. controls (338.46,239.37) and (339.45,240.41) .. (339.45,241.69) .. controls (339.45,242.97) and (338.46,244.01) .. (337.24,244.01) .. controls (336.01,244.01) and (335.02,242.97) .. (335.02,241.69) -- cycle ;
\draw  [fill={rgb, 255:red, 74; green, 144; blue, 226 }  ,fill opacity=1 ] (301.81,183.51) .. controls (301.81,182.23) and (302.8,181.2) .. (304.02,181.2) .. controls (305.24,181.2) and (306.24,182.23) .. (306.24,183.51) .. controls (306.24,184.79) and (305.24,185.83) .. (304.02,185.83) .. controls (302.8,185.83) and (301.81,184.79) .. (301.81,183.51) -- cycle ;
\draw  [fill={rgb, 255:red, 74; green, 144; blue, 226 }  ,fill opacity=1 ] (459.76,244.76) .. controls (459.76,243.48) and (460.75,242.44) .. (461.97,242.44) .. controls (463.2,242.44) and (464.19,243.48) .. (464.19,244.76) .. controls (464.19,246.04) and (463.2,247.07) .. (461.97,247.07) .. controls (460.75,247.07) and (459.76,246.04) .. (459.76,244.76) -- cycle ;
\draw [color={rgb, 255:red, 128; green, 128; blue, 128 }  ,draw opacity=0.68 ][line width=2.25]    (267.08,154.67) -- (302.02,181.83) ;
\draw [color={rgb, 255:red, 128; green, 128; blue, 128 }  ,draw opacity=0.67 ][line width=2.25]    (461.97,247.07) -- (463.08,290.67) ;
\draw [color={rgb, 255:red, 128; green, 128; blue, 128 }  ,draw opacity=0.67 ][line width=2.25]    (336.19,243.76) -- (315.08,278.67) ;
\draw  [fill={rgb, 255:red, 74; green, 144; blue, 226 }  ,fill opacity=1 ] (583.97,51.47) .. controls (583.97,50.88) and (584.41,50.4) .. (584.97,50.4) .. controls (585.52,50.4) and (585.96,50.88) .. (585.96,51.47) .. controls (585.96,52.06) and (585.52,52.54) .. (584.97,52.54) .. controls (584.41,52.54) and (583.97,52.06) .. (583.97,51.47) -- cycle ;
\draw   (510.04,38.9) .. controls (526.91,31.99) and (544.79,49.9) .. (565.3,48.62) .. controls (585.8,47.34) and (626.82,13.45) .. (635.76,54.37) .. controls (644.7,95.3) and (576.86,78.03) .. (558.99,75.48) .. controls (541.11,72.92) and (515.3,88.78) .. (503.77,81.23) .. controls (492.25,73.68) and (493.17,45.82) .. (510.04,38.9) -- cycle ;
\draw  [draw opacity=0] (622.83,54.64) .. controls (620.59,59.24) and (615.26,62.63) .. (608.95,62.95) .. controls (602.06,63.29) and (596.1,59.86) .. (593.93,54.82) -- (608.52,50.19) -- cycle ; \draw   (622.83,54.64) .. controls (620.59,59.24) and (615.26,62.63) .. (608.95,62.95) .. controls (602.06,63.29) and (596.1,59.86) .. (593.93,54.82) ;
\draw  [draw opacity=0] (597.15,59.8) .. controls (598.99,57.04) and (602.85,55.08) .. (607.34,54.98) .. controls (612.3,54.86) and (616.56,57.04) .. (618.22,60.19) -- (607.47,63.3) -- cycle ; \draw   (597.15,59.8) .. controls (598.99,57.04) and (602.85,55.08) .. (607.34,54.98) .. controls (612.3,54.86) and (616.56,57.04) .. (618.22,60.19) ;
\draw  [draw opacity=0] (534.64,58.5) .. controls (532.64,63.55) and (527.49,67.33) .. (521.34,67.67) .. controls (514.69,68.04) and (508.95,64.28) .. (507.03,58.81) -- (520.93,54.32) -- cycle ; \draw   (534.64,58.5) .. controls (532.64,63.55) and (527.49,67.33) .. (521.34,67.67) .. controls (514.69,68.04) and (508.95,64.28) .. (507.03,58.81) ;
\draw  [draw opacity=0] (510.47,63.78) .. controls (512.13,60.72) and (515.88,58.52) .. (520.27,58.4) .. controls (525.08,58.28) and (529.18,60.69) .. (530.66,64.13) -- (520.4,67.12) -- cycle ; \draw   (510.47,63.78) .. controls (512.13,60.72) and (515.88,58.52) .. (520.27,58.4) .. controls (525.08,58.28) and (529.18,60.69) .. (530.66,64.13) ;
\draw  [fill={rgb, 255:red, 74; green, 144; blue, 226 }  ,fill opacity=1 ] (531.74,70.93) .. controls (531.74,70.34) and (532.19,69.86) .. (532.74,69.86) .. controls (533.29,69.86) and (533.74,70.34) .. (533.74,70.93) .. controls (533.74,71.52) and (533.29,72) .. (532.74,72) .. controls (532.19,72) and (531.74,71.52) .. (531.74,70.93) -- cycle ;
\draw  [fill={rgb, 255:red, 74; green, 144; blue, 226 }  ,fill opacity=1 ] (514.93,46.38) .. controls (514.93,45.79) and (515.38,45.31) .. (515.93,45.31) .. controls (516.48,45.31) and (516.93,45.79) .. (516.93,46.38) .. controls (516.93,46.97) and (516.48,47.45) .. (515.93,47.45) .. controls (515.38,47.45) and (514.93,46.97) .. (514.93,46.38) -- cycle ;
\draw  [fill={rgb, 255:red, 74; green, 144; blue, 226 }  ,fill opacity=1 ] (586.46,69.54) .. controls (586.46,68.95) and (586.91,68.47) .. (587.46,68.47) .. controls (588.01,68.47) and (588.46,68.95) .. (588.46,69.54) .. controls (588.46,70.13) and (588.01,70.61) .. (587.46,70.61) .. controls (586.91,70.61) and (586.46,70.13) .. (586.46,69.54) -- cycle ;
\draw [color={rgb, 255:red, 128; green, 128; blue, 128 }  ,draw opacity=0.69 ][line width=2.25]    (504.08,35.41) -- (515.03,45.6) ;
\draw [color={rgb, 255:red, 128; green, 128; blue, 128 }  ,draw opacity=0.68 ][line width=2.25]    (531.74,71.93) -- (518.08,88.29) ;
\draw [color={rgb, 255:red, 128; green, 128; blue, 128 }  ,draw opacity=0.68 ][line width=2.25]    (587.46,70.61) -- (585.08,79.41) ;
\draw [color={rgb, 255:red, 128; green, 128; blue, 128 }  ,draw opacity=0.68 ][line width=2.25]    (584.96,50.47) -- (597.08,40.41) ;
\draw [color={rgb, 255:red, 184; green, 233; blue, 134 }  ,draw opacity=1 ][fill={rgb, 255:red, 255; green, 255; blue, 255 }  ,fill opacity=1 ][line width=2.25]    (335.08,242.41) .. controls (295.08,269.41) and (265.08,249.41) .. (335.02,241.69) ;
\draw  [fill={rgb, 255:red, 74; green, 144; blue, 226 }  ,fill opacity=1 ] (49.81,57.51) .. controls (49.81,56.23) and (50.8,55.2) .. (52.02,55.2) .. controls (53.24,55.2) and (54.24,56.23) .. (54.24,57.51) .. controls (54.24,58.79) and (53.24,59.83) .. (52.02,59.83) .. controls (50.8,59.83) and (49.81,58.79) .. (49.81,57.51) -- cycle ;
\draw [color={rgb, 255:red, 128; green, 128; blue, 128 }  ,draw opacity=0.68 ][line width=2.25]    (54.11,58.12) -- (90.08,56.29) ;
\draw [color={rgb, 255:red, 128; green, 128; blue, 128 }  ,draw opacity=0.68 ][line width=2.25]    (53.11,56.12) -- (70.08,23.67) ;
\draw [color={rgb, 255:red, 128; green, 128; blue, 128 }  ,draw opacity=0.68 ][line width=2.25]    (82.08,80.67) -- (53.02,58.83) ;
\draw [color={rgb, 255:red, 128; green, 128; blue, 128 }  ,draw opacity=0.68 ][line width=2.25]    (35.04,90.64) -- (52.02,58.2) ;
\draw [color={rgb, 255:red, 128; green, 128; blue, 128 }  ,draw opacity=0.68 ][line width=2.25]    (42.08,20.67) -- (51.02,55.2) ;
\draw [color={rgb, 255:red, 128; green, 128; blue, 128 }  ,draw opacity=0.66 ][line width=2.25]    (482.08,155.67) -- (453.02,191.83) ;
\draw [color={rgb, 255:red, 184; green, 233; blue, 134 }  ,draw opacity=1 ][fill={rgb, 255:red, 255; green, 255; blue, 255 }  ,fill opacity=1 ][line width=2.25]    (40.57,118.1) .. controls (45.57,111.1) and (66.57,111.1) .. (78.12,112.35) ;
\draw [color={rgb, 255:red, 184; green, 233; blue, 134 }  ,draw opacity=1 ][fill={rgb, 255:red, 255; green, 255; blue, 255 }  ,fill opacity=1 ][line width=2.25]    (464.08,244.41) .. controls (545.08,277.41) and (587.08,205.41) .. (464.02,243.69) ;
\draw [color={rgb, 255:red, 184; green, 233; blue, 134 }  ,draw opacity=1 ][fill={rgb, 255:red, 255; green, 255; blue, 255 }  ,fill opacity=1 ][line width=2.25]    (307.24,184.51) .. controls (312.24,177.51) and (370.08,216.41) .. (449.64,193.54) ;
\draw [color={rgb, 255:red, 184; green, 233; blue, 134 }  ,draw opacity=1 ][fill={rgb, 255:red, 255; green, 255; blue, 255 }  ,fill opacity=1 ][line width=2.25]    (450.64,194.54) .. controls (435.08,214.41) and (438.08,234.41) .. (459.76,243.76) ;
\draw [color={rgb, 255:red, 184; green, 233; blue, 134 }  ,draw opacity=1 ][fill={rgb, 255:red, 255; green, 255; blue, 255 }  ,fill opacity=1 ][line width=2.25]    (454.06,193.54) .. controls (471.51,195.41) and (478.51,232.41) .. (464.19,242.76) ;
\draw [color={rgb, 255:red, 128; green, 128; blue, 128 }  ,draw opacity=0.69 ][line width=2.25]    (517.09,47.11) -- (531.08,49.41) ;
\draw [color={rgb, 255:red, 128; green, 128; blue, 128 }  ,draw opacity=0.69 ][line width=2.25]    (505.75,69.63) -- (531.74,70.93) ;
\draw [color={rgb, 255:red, 128; green, 128; blue, 128 }  ,draw opacity=0.69 ][line width=2.25]    (506.08,75.41) -- (531.74,71.93) ;
\draw [color={rgb, 255:red, 128; green, 128; blue, 128 }  ,draw opacity=0.69 ][line width=2.25]    (572.08,52.41) -- (584.29,51.27) ;
\draw [color={rgb, 255:red, 128; green, 128; blue, 128 }  ,draw opacity=0.69 ][line width=2.25]    (588.87,68.87) -- (603.08,66.41) ;
\draw [color={rgb, 255:red, 128; green, 128; blue, 128 }  ,draw opacity=0.69 ][line width=2.25]    (588.46,71.54) -- (604.08,75.41) ;
\draw [color={rgb, 255:red, 128; green, 128; blue, 128 }  ,draw opacity=0.69 ][line width=2.25]    (576.08,67.41) -- (586.46,69.54) ;
\draw [color={rgb, 255:red, 128; green, 128; blue, 128 }  ,draw opacity=0.69 ][line width=2.25]    (587.46,68.54) -- (591.08,60.41) ;
\draw [color={rgb, 255:red, 128; green, 128; blue, 128 }  ,draw opacity=0.69 ][line width=2.25]    (583.08,61.41) -- (584.97,52.54) ;
\draw [color={rgb, 255:red, 128; green, 128; blue, 128 }  ,draw opacity=0.69 ][line width=2.25]    (585.96,51.47) -- (598.08,50.41) ;

\draw (155,201.4) node [anchor=north west][inner sep=0.75pt]    {$e^{Q_{\mathbf{reg}}} =\sum $};
\draw (334,52.4) node [anchor=north west][inner sep=0.75pt]  [font=\scriptsize]  {$\Omega ^{0,\bullet }\left( X^{n} ,( \mathrm{Sym}\ \mathcal{L})^{\boxtimes n}( *\Delta )\right) \ni v:$};
\draw (119,49.4) node [anchor=north west][inner sep=0.75pt]  [font=\scriptsize]  {$\mathrm{Sym}\ \mathcal{L}$};
\draw (93,103.4) node [anchor=north west][inner sep=0.75pt]  [font=\scriptsize]  {$Q\in \Omega ^{0,\bullet }\left( X^{2} ,E\boxtimes E\right)^{\sigma _{2}}\text{ }$};

\end{tikzpicture}
    \caption{$e^{Q_{\mathbf{reg}}}$}
    \label{Regular}
\end{figure}

For any $x\in X$, we choose an open neighbourhood $U$ with local coordinate $z.$ Then we can split the Szeg\"{o} kernel into the singular part and the regular part
\begin{equation}\label{SplitPropagator}
P|_{U\times U}=\mathfrak{P}+\mathfrak{Q},\quad \mathfrak{P}=\frac{\mathbf{Id}\cdot dz}{z_1-z_2},\mathfrak{Q}\in \Omega^{0,0}(U\times U,E\boxtimes E),
\end{equation}

here we use the canonical isomorphism $E\simeq E^{\vee}\otimes \omega_X$ and denote $\mathbf{Id}\cdot dz\in E\boxtimes E^{\vee}\otimes \omega_X\simeq E\boxtimes E$.

We define
$$
\mathcal{W}^{\tau^z_{U}}:\Omega^{0,\bullet}\left(U^n,(\mathrm{Sym}\ \mathcal{L})^{\boxtimes n}(*\Delta) \right)\xrightarrow{e^{\mathfrak{P}_{\mathbf{sing}}+\mathfrak{Q}_{\mathbf{reg}}}}\Omega^{0,\bullet}\left(U^n,(\mathrm{Sym}\ \mathcal{L})^{\boxtimes n}(*\Delta) \right).
$$
Within these notations, we can rewrite $\mathcal{W}^{\mathbf{v}}(v;\tilde{v})$ locally as
$$
\mathcal{W}^{\mathbf{v}}(v;\tilde{v})|_U=\Vec{\mu}_{\mathrm{Sym}}\left( \mathcal{W}^{\tau^z_U}(\tau^{z\ \boxtimes n}_U\tilde{v})\right).
$$
Recall that $\mathcal{W}^{\mathbf{v}}$ is defined in (\ref{NormalOrdering}) and $\tilde{v}\in \Omega^{0,\bullet}\left(U^n,\mathcal{L}^{\flat\boxtimes n}(*\Delta)\right)$.

Here we reformulate Wick the theorem using chiral algebras as we will use later.

\begin{thm}\label{WickTheoremChiral}
    We have
    $$
    \tau^{z}_U\mu_{\mathcal{A}}(\eta \cdot v_1dz_1\boxtimes v_2dz_2)=\mu_{\mathrm{Sym}}\left(\eta \cdot e^{\mathfrak{P}_{\mathbf{sing}}} \tau^{z}_U(v_1)dz_1\boxtimes  \tau^{z}_U(v_2)dz_2\right).
    $$
    Here $z_1$ and $z_2$ are two copies of the local coordinate $z$ on $U\times U$.
\end{thm}
\begin{proof}
The proof is in Appendix \ref{WickThm}.
\end{proof}

The following result is established in \cite{gui2021elliptic}.
\begin{thm}\label{ChiralIntertwines}
    The map $\mathcal{W}^{\tau^z_U}$ is compatible with $\mathcal{D}$-module structure and intertwines the chiral operation. In other words, we have
    $$
    \mathcal{W}^{\tau^z_U}\left(\tau^z_U\mu_{\mathcal{A}}(v)\right)=\mu_{\mathrm{Sym}}\left(\mathcal{W}^{\tau^z_U}(\tau^{z\ \boxtimes 2}_U(v))\right)\quad v\in \Omega^{0,\bullet}\left(X^2,\mathcal{A}^{\boxtimes 2}(*\Delta)\right).
    $$
\end{thm}

\begin{proof}
The compatibility with $\mathcal{D}$-module structure is obvious. To prove the intertwining property, we need some preparation. We write
$$
e^{\mathfrak{P}_{\mathbf{sing}}}\tau^{z}_U(v_1)dz_1\boxtimes\tau^{z}_U(v_2)dz_2=\sum_{k=1}^{N_1} w_k a^k dz_1\boxtimes b^k dz_2,
$$
where
$$
w_k a^k dz_1\boxtimes b^k dz_2\in \Omega^{0,\bullet}(U\times U,(\mathrm{Sym}\ \mathcal{L})^{\boxtimes 2}(*\Delta)),\quad k=1,\dots,N_1,
$$
and
$$
\mu_{\omega}(\eta\cdot w_kdz_1\boxtimes dz_2)=\sum_{l=1}^{N_2} f_{kl}dz\otimes_{\mathcal{D}_{X\rightarrow X^2}} P_{kl}(\partial_{z_1},\partial_{z_2}).
$$
We write
$$
P_{kl}(\partial_{z_1},\partial_{z_2})(f\cdot g)=\sum_{r=1}^{N_3} P_{kl}^{r\blacktriangle}(\partial_{z_1},\partial_{z_2})f\cdot  P_{kl}^{r\blacktriangledown}(\partial_{z_1},\partial_{z_2})g.
$$
for local functions $f,g$ on $U\times U$.

With the above notations, we have
\begin{align*}
  \mu_{\mathrm{Sym}}\left(\mathcal{W}^{\tau^{z}_U}(\eta\cdot\tau^{z}_U(v_1)dz_1\boxtimes\tau^{z}_U(v_2)dz_2))\right)&=\mu_{\mathrm{Sym}}\left(e^{\mathfrak{P}_{\mathbf{sing}}+\mathfrak{Q}_{\mathbf{reg}}}(\eta\cdot\tau^{z}_U(v_1)dz_1\boxtimes\tau^{z}_U(v_2)dz_2))\right)\\
  &=\mu_{\mathrm{Sym}}\left(e^{\mathfrak{Q}_{\mathbf{reg}}}(\eta\cdot\sum_{k=1}^{N_1} w_k a^k dz_1\boxtimes b^k dz_2\right)\\
  &=\sum_{k=1}^{N_1}\mu_{\omega}\left(\eta\cdot w_k dz_1\boxtimes dz_2\right)\cdot(e^{\mathfrak{Q}_{\mathbf{reg}}}a^k \boxtimes b^k\otimes 1)\\
  &=\sum_{k=1}^{N_1} \sum_{l=1}^{N_2} f_{kl}dz\otimes_{\mathcal{D}_{X\rightarrow X^2}} P_{kl}(\partial_{z_1},\partial_{z_2})\cdot(e^{\mathfrak{Q}_{\mathbf{reg}}}a^k \boxtimes b^k\otimes 1)\\
\end{align*}

Now notice that $a^k\boxtimes b^k\in (\mathrm{Sym}\ \mathcal{L})^{l\boxtimes 2}=(\mathrm{Sym}\ \mathcal{L})^{\boxtimes 2}\otimes \omega_{X^2}^{-1}$ where the left $\mathcal{D}_{X^2}$-module structure is given by the right $\mathcal{D}^{\mathrm{op}}_{X^2}\simeq\omega_{X^2}\otimes\mathcal{D}_{X^2}\otimes \omega_{X^2}^{-1} $-module structure. This means that we can write
$$
P_{kl}(\partial_{z_1},\partial_{z_2})\cdot(e^{\mathfrak{Q}_{\mathbf{reg}}}a^k \boxtimes b^k\otimes 1)=\sum_{r=1}^{N_3} e^{\mathfrak{Q}_{\mathbf{reg}}}a^k \boxtimes b^kP_{kl}^{r\blacktriangle}(\partial_{z_1},\partial_{z_2})^{\mathrm{op}}\otimes  P_{kl}^{r\blacktriangledown}(\partial_{z_1},\partial_{z_2}).
$$

Using the fact that $\Delta^*e^{\mathfrak{Q}_{\mathbf{reg}}}=e^{\mathfrak{Q}_{\mathbf{reg}}}\Delta^*$ we calculate
$$
\sum_{k=1}^{N_1}\mu_{\omega}\left(\eta\cdot w_k dz_1\boxtimes dz_2\right)\cdot(e^{\mathfrak{Q}_{\mathbf{reg}}}a^k \boxtimes b^k\otimes 1)
$$
$$
=\sum_{k=1}^{N_1}\sum_{l=1}^{N_2} \sum_{r=1}^{N_3} f_{kl}dz\otimes_{\mathcal{D}_{X\rightarrow X^2}} \cdot(e^{\mathfrak{Q}_{\mathbf{reg}}}a^k \boxtimes b^k P^{r\blacktriangle}_{kl}(\partial_{z_1},\partial_{z_2})^{\mathrm{op}}\otimes P^{r\blacktriangledown}_{kl}(\partial_{z_1},\partial_{z_2}))
$$
$$
=\sum_{k=1}^{N_1}\sum^{N_2}_{l=1} \sum^{N_3}_{r=1} f_{kl}dz \cdot \Delta^*(e^{\mathfrak{Q}_{\mathbf{reg}}}a^k \boxtimes b^k P^{r\blacktriangle}_{kl}(\partial_{z_1},\partial_{z_2})^{\mathrm{op}})\otimes_{\mathcal{D}_{X\rightarrow X^2}}  P^{r\blacktriangledown}_{kl}(\partial_{z_1},\partial_{z_2})
$$
$$
=e^{\mathfrak{Q}_{\mathbf{reg}}}\sum_{k=1}^{N_1}\sum^{N_2}_{l=1} \sum^{N_3}_{r=1} f_{kl}dz \cdot \Delta^*(a^k \boxtimes b^k P^{r\blacktriangle}_{kl}(\partial_{z_1},\partial_{z_2})^{\mathrm{op}})\otimes_{\mathcal{D}_{X\rightarrow X^2}}  P^{r\blacktriangledown}_{kl}(\partial_{z_1},\partial_{z_2})
$$
$$
=e^{\mathfrak{Q}_{\mathbf{reg}}}\tau^z_U(\mu_{\mathcal{A}}(\eta\cdot v_1dz_1\boxtimes v_2dz_2))=    \mathcal{W}^{\tau^z_U}(\tau^z_U\mu_{\mathcal{A}}(v)).
$$
Here in the last step, we use the Wick theorem. The proof is complete.
\end{proof}

\begin{rem}
    Strictly speaking, the theorem in \cite[Section 3.3.]{gui2021elliptic} concerns the unit chiral operation. However, the proof works for the commutative chiral algebra $\mathrm{Sym}\ \mathcal{L}$ as well. Here we reformulate the proof for the reader's convenience.
\end{rem}

\begin{prop}\label{IndependentChoice}
    Let $v\in\Omega^{0,\bullet}(X,\mathcal{A})$ and $\tilde{v}\in \Omega^{0,\bullet}(X^n,\mathcal{L}^{\flat\boxtimes n}(*\Delta))$  such that
    $    \Vec{\mu}_{\mathcal{A}}(\tilde{v})=v.$    The element $\mathcal{W}^{\mathbf{v}}(v;\tilde{v})$ is independent of the choice of the presentation $\tilde{v}$. Furthermore, $\mathcal{W}^{\mathbf{v}}(v;\tilde{v})$ has the following explicit formula if we choose a local chart $U$ with coordinate $z$
    $$
    \mathcal{W}^{\mathbf{v}}(v;\tilde{v})|_U=e^{\mathfrak{Q}_{\mathbf{reg}}}\tau^z_U(v).
    $$
\end{prop}
\begin{proof}
     We will prove this proposition by showing that
     $$
     \mathcal{W}^{\mathbf{v}}|_U(v|_U)=\mathcal{W}^{\tau^z_U}(\tau^z_U(v|_U))
     $$
     and thus only depends on $v$ itself.

     Suppose that we have a presentation $\tilde{v}_U\in \Omega^{0,\bullet}\left(U^n,\mathcal{L}^{\flat\boxtimes n}(*\Delta)\right)$ of $v$. We have
     $$
     \Vec{\mu}_{\mathcal{A}}(\tilde{v}_U)=v|_U.
     $$
Then by definition
$$
\mathcal{W}^{\mathbf{v}}(v;\tilde{v})=\Vec{\mu}_{\mathrm{Sym}}(\mathcal{W}^{\tau^z_{U}}\tau_U^{z\ \boxtimes n}(\tilde{v}))
$$
         Using Theorem \ref{ChiralIntertwines}, we get

    \begin{align*}
\mathcal{W}^{\mathbf{v}}(v;\tilde{v})&=\Vec{\mu}_{\mathrm{Sym}}(\mathcal{W}^{\tau^z_U}(\tau_U^{z\ \boxtimes n}(\tilde{v})))\\
&=\mathcal{W}^{\tau^z_U}(\tau^z_U\Vec{\mu}_{\mathcal{A}}(\tilde{v}))\\
&=\mathcal{W}^{\tau^z_U}(\tau^z_U(v)).
    \end{align*}

Now suppose that we have two global presentations $\tilde{v},\vardbtilde{v}$. Then
$$
\mathcal{W}^{\mathbf{v}}(v;\tilde{v})|_U=\mathcal{W}^{\tau^{z}_U}(\tau^{z}_U(v))=\mathcal{W}^{\mathbf{v}}(v;\vardbtilde{v})|_U.
$$
We conclude that $\mathcal{W}^{\mathbf{v}}(v;\tilde{v})=\mathcal{W}^{\mathbf{v}}(v;\vardbtilde{v}).$

\end{proof}

We have the following corollary.
\begin{cor}
Let $\{U_i\}_{i\in I}$ be a finite open cover of $X$ with local coordinate $z_i$ for each $U_i$.    The collection of maps $\{\mathcal{W}^{\tau_{U_i}}(\tau_{U_i}(v))\}_{i\in I}$ can be glued to a globally well defined map $$
    \mathcal{W}^{\mathbf{v}}:\mathcal{A}\rightarrow \mathrm{Sym}\ \mathcal{L}.$$
\end{cor}
\begin{proof}
    We only need to note that $\mathcal{W}^{\mathbf{v}}(v;\tilde{v})$ is globally well defined since we can find a global section $\tilde{v}$ such that $\Vec{\mu}_{\mathcal{A}}(\tilde{v})=v$. Moreover, it is equal to $\mathcal{W}^{\tau^{z_i}_{U_i}}(\tau^{z_i}_{U_i}(v))$ on each $U_i$ by Proposition \ref{IndependentChoice}.
\end{proof}
\subsection{ Trace map on chiral envelopes}

In this section, we will give the construction of the trace map on chiral Weyl algebra for any Riemann surface.

We define
$$
\mathcal{W}_{\mathrm{ch}}:\Omega^{0,\bullet}(X^n,\mathcal{A}^{\boxtimes n}(*\Delta))\rightarrow \Omega^{0,\bullet}(X^n,(\mathrm{Sym}\ \mathcal{L})^{\boxtimes n}(*\Delta)).
$$
to be the composition of the following sequence of maps
$$
\Omega^{0,\bullet}(X^n,\mathcal{A}^{\boxtimes n}(*\Delta))\xrightarrow{(\mathcal{W}^{\mathbf{v}})^{\boxtimes n}} \Omega^{0,\bullet}(X^n,(\mathrm{Sym}\ \mathcal{L})^{\boxtimes n}(*\Delta))\xrightarrow {e^{P_{\mathbf{sing}}}}\Omega^{0,\bullet}(X^n,(\mathrm{Sym}\ \mathcal{L})^{\boxtimes n}(*\Delta)).
$$

Notice that locally on $U^n$, the map $\mathcal{W}_{\mathrm{ch}}$ is exactly $\mathcal{W}^{\tau^{z}_U}\circ \tau^{z\ \boxtimes n}_U$. This can be seen as follows, we split the propagator locally on $U$ as in (\ref{SplitPropagator}) and observe that
$$
e^{\mathfrak{Q}_{\mathbf{Sing}}}\mathcal{W}^{\mathbf{v}\boxtimes n}=e^{\mathfrak{Q}_{\mathbf{reg}}}.
$$
Then on $U$, we have
$$
e^{P_{\mathbf{sing}}}\mathcal{W}^{\mathbf{v}\boxtimes n}=e^{\mathfrak{P}_{\mathbf{sing}}+\mathfrak{Q}_{\mathbf{Sing}}}\mathcal{W}^{\mathbf{v}\boxtimes n}=e^{\mathfrak{P}_{\mathbf{sing}}+\mathfrak{Q}_{\mathbf{reg}}}.
$$

Denote the projection map
$$
\Omega^{0,\bullet}(X^n,(\mathrm{Sym}\ \mathcal{L})^{\boxtimes n}(*\Delta))\rightarrow  \Omega^{0,\bullet}(X^n,\omega_{X^n}(*\Delta))
$$
by $\mathbf{p}$ which is induced by projection $\mathrm{Sym}\ \mathcal{L}\rightarrow \mathrm{Sym}^0\mathcal{L}=\omega_X$.

Then we construct the trace map as follows, define
$$
\mathrm{Tr}_{\mathrm{ch}}: \Omega^{0,\bullet}(X^n,\mathcal{A}^{\boxtimes n}(*\Delta))\rightarrow O_{\mathrm{BV}}.
$$
Here
$$
\mathrm{Tr}_{\mathrm{ch}}(\eta)[\mathbf{e}]:=\sum_{k\geq 0}\frac{1}{k!}\mathrm{tr}_{\omega}\circ \mathbf{p}\left(\partial^k_{\mathbf{e}}\mathcal{W}_{\mathrm{ch}}(\eta)\right).
$$
As before, this map extends naturally to a map $\mathrm{Tr}_{\mathrm{ch}}:\tilde{C}^{\mathrm{ch}}(X,\mathcal{A})_{\mathcal{Q}}\rightarrow O_{\mathrm{BV}}$

\begin{thm}\label{MainTheorem510}
    The map
    $$
    \mathrm{Tr}_{\mathrm{ch}}:(\tilde{C}^{\mathrm{ch}}(X,\mathcal{A})_{\mathcal{Q}},d^{\mathrm{ch}}_{\mathcal{A}})\rightarrow (O_{\mathrm{BV}},-\Delta_{\mathrm{BV}})
    $$
    is a chain map which is furthermore a quasi-isomorphism.
\end{thm}

\begin{proof}
We need to prove that
    $$
    [\bar{\partial},\mathcal{W}_{\mathrm{ch}}]=\Delta_{\mathrm{BV}}\circ \mathcal{W}_{\mathrm{ch}}.
    $$

    However, since we can present a section of the chiral algebra as the chiral product of linear fields in $\mathcal{L}$, the above identity reduces to the one for Lie* algebra which we already proved in our construction of Trace map for Lie* algebra $\mathcal{L}^{\flat}$. More precisely, we can present a chiral chain $\eta$ as a iterated chiral product of a Lie* chain $\tilde{\eta}$
    $$
    \eta=\Vec{\mu}_{\mathcal{A}}(\tilde{\eta}).
    $$
    Then
    \begin{align*}
      \mathcal{W}_{\mathrm{ch}}(\bar{\partial}\eta) &=   \mathcal{W}_{\mathrm{ch}}\left(\bar{\partial}\Vec{\mu}_{\mathcal{A}}(\tilde{\eta})\right)  \\
       & =\Vec{\mu}_{\mathrm{Sym}}(  \mathcal{W}_{\mathrm{Lie}}\left(\bar{\partial}\tilde{\eta})\right)\\
       &=\bar{\partial}\Vec{\mu}_{\mathrm{Sym}}(  \mathcal{W}_{\mathrm{Lie}}\left(\tilde{\eta})\right)-\Delta_{\mathrm{BV}}\Vec{\mu}_{\mathrm{Sym}}(  \mathcal{W}_{\mathrm{Lie}}\left(\tilde{\eta})\right)\\
       &=\bar{\partial}\mathcal{W}_{\mathrm{ch}}(\eta)-\Delta_{\mathrm{BV}}\mathcal{W}_{\mathrm{ch}}(\eta).
    \end{align*}

    Now using Theorem \ref{ChiralIntertwines}, we have
    \begin{align*}
    \mathrm{Tr}_{\mathrm{ch}}(d^{\mathrm{ch}}_{\mathcal{A}}\eta)[\mathbf{e}]&=\sum_{k\geq 0}\frac{1}{k!}\mathrm{tr}_{\omega}\circ \mathbf{p}\left(\partial^k_{\mathbf{e}}\mathcal{W}_{\mathrm{ch}}(d^{\mathrm{ch}}_{\mathcal{A}}\eta)\right)\\
    &=\sum_{k\geq 0}\frac{1}{k!}\mathrm{tr}_{\omega}\circ \mathbf{p}\left((d_{\mathrm{Sym}}^{\mathrm{ch}}-\Delta_{\mathrm{BV}})\partial^k_{\mathbf{e}}\mathcal{W}_{\mathrm{ch}}(\eta)\right)\\
    &=\underbrace{\sum_{k\geq 0}\frac{1}{k!}\mathrm{tr}_{\omega}\circ d_{\omega}^{\mathrm{ch}}\mathbf{p}\left(\partial^k_{\mathbf{e}}\mathcal{W}_{\mathrm{ch}}(\eta)\right)}_{=0}-\sum_{k\geq 0}\frac{1}{k!}\mathrm{tr}_{\omega}\circ \mathbf{p}\left(\Delta_{\mathrm{BV}}\partial^k_{\mathbf{e}}\mathcal{W}_{\mathrm{ch}}(\eta)\right)\\
    &=-\left(\Delta_{\mathrm{BV}}\mathrm{Tr}_{\mathrm{ch}}(\eta)\right)[\mathbf{e}].
\end{align*}

Finally, by applying the same argument as in \cite[Theorem 3.18]{gui2021elliptic}, we see that $\mathrm{Tr}_{\mathrm{ch}}$ is a quasi-isomorphism.

\end{proof}

\section{Examples}
In this section, we compute some examples of our trace map. Suppose that the chiral homology is one-dimensional and concentrated in the degree 0.  Suppose that we have an operator $\mathcal{O}\in \Omega^{0,\bullet}(X,\mathcal{A})$.
Since $\mathcal{O}$ is closed as a chiral chain, it must be homologous to the scalar multiplication of the unit. In other words, we have
$$
\mathcal{O}=\langle\mathcal{O} \rangle\cdot 1+\text{exact-term}\in C^{ch}(X,\mathcal{A}).
$$
The complex number $\langle \mathcal{O}\rangle$ is called the expectation value of $\mathcal{O}$. In general, it is hard to find the closed formula of this number as the chiral chain complex is very complicated. However, the trace map helps to compute this number quickly. Applying the trace map to the above identity, we get
$$
\mathrm{Tr}(\mathcal{O})=\langle \mathcal{O}\rangle,
$$
here use the fact that the trace map annihilates the exact term.

Here we give an informal discussion on $\beta\gamma$-system using physical terminology. Denote the path integral of $\beta\gamma$-system by $I_{\beta\gamma}$
$$
I_{\beta\gamma}=\int \mathcal{D}\mathcal{E}\cdot e^{\int_X\beta\bar{\partial}\gamma },\quad \mathcal{E}=\Omega^{0,\bullet}(X,E).
$$
Suppose there are no zero modes, it is known that $I_{\beta\gamma}$ is equal to the determinant of $\bar{\partial}$. As well-known to physicists, doing variation along some parameters (like moduli of curve or bundle) corresponds to the insertion of some operators. Roughly speaking,
$$
\frac{dI_{\beta\gamma}}{d\tau}=\int \mathcal{D}\mathcal{E}\cdot e^{\int_X\beta\bar{\partial}\gamma }\cdot T_{\mu}=\langle T_{\mu}\rangle.
$$
A similar thing happens when we do the variation along the moduli of bundles.

We calculate these operator insertions and find the precise match with Fay's calculations on variation of analytic torsion.

\subsection{Insertion of a modified current operator}

Let $F$ be a holomorphic bundle over $X$, then the direct sum $E=F\oplus F^{\vee}\otimes_{\mathcal{O}_X}\omega_X$ has a natural symplectic pairing
$$
E\otimes_{\mathcal{O}_X} E\rightarrow\omega
$$
induced by the pairing $F\otimes_{\mathcal{O}_X} (F^{\vee}\otimes_{\mathcal{O}_X} \omega_X)\rightarrow \omega_X$.

Consider a smooth differential form $\nu\in \Omega^{0,1}(X,\mathrm{End}(F))$. Choosing a local coordinate $z$ and a local holomorphic frame $\{e_i\}$ of $F$, we can write $\nu=\nu^i_{j,\bar{z}}d\bar{z}\cdot e_i\otimes e^j$. One may try to define a global section of chiral algebra as follows
$$
J_{\nu}\stackrel{?}{=}\nu^i_{j,\bar{z}}d\bar{z}\cdot \beta_{i,z}\gamma^jdz.
$$
Here $\gamma^j$ and $\beta_{i,z}dz$ correspond to $e^j$ and $e_i$, respectively. The problem is that under coordinate change, $\beta_{i,z}\gamma^j$ behaves differently from $e_i\otimes e^j$. However, we can correct this after choosing  Hermitian metrics on $E$ and $\omega_X^{\frac{1}{2}}$.  Following \cite{fay1992kernel}, we use the notation $h$ (resp. $\rho$) for the Hermitian metric on $E$ (resp. $\omega_X^{\frac{1}{2}}$). It will be clear that our later construction is independent of the spin structure $\omega_X^{\frac{1}{2}}$.

\begin{prop}
    The following expression is independent of the choice of local coordinate $z$
    $$    J_{\nu}:=\nu^i_{j,\bar{z}}d\bar{z}\cdot \beta_{i,z}\gamma^jdz-\mathrm{tr}\left(\nu\cdot\rho h^{-1}\partial_z(h\rho^{-1})dz\right)
    $$
Thus we have a globally well defined section $J_{\nu}\in \Omega^{0,1}(X,\mathcal{A}).
    $
\end{prop}
\begin{proof}
    If we change the local coordinate, we have
    $$
\nu^i_{j,\bar{w}}d\bar{w}\cdot \beta_{i,w}\gamma^jdw=\nu^i_{j,\bar{z}}d\bar{z}\cdot \beta_{i,z}\gamma^jdz+\mathrm{tr}(\frac{1}{2}\nu\theta')
    $$
Here $\theta(z)=\frac{dw}{dz}$ and $\theta'=\frac{d\theta}{dz}$. It is the same as $\mathrm{tr}(\nu\cdot\rho_z h^{-1}\partial_z(h\rho^{-1}_zdz)).$

    If we change the holomorphic frame, we get
    $$
    \nu^{\tilde{i}}_{\tilde{j},\bar{z}}d\bar{z}\cdot \tilde{\beta}_{\tilde{i},z}\tilde{\gamma}^{\tilde{j}}dz=\nu^i_{j,\bar{z}}d\bar{z}\cdot \beta_{i,z}\gamma^jdz+\mathrm{tr}(\nu\partial_z\log \varsigma).
    $$
    Here $\varsigma(z)$ is the local transition function of the frame change. Again it is cancelled by the second term in $J_{\nu}$.

\end{proof}

Now for the rest of this subsection, we assume that $H^0(X,F)=H^1(X,F)=0$. Following the notation in \cite{fay1992kernel}, we expand the Szeg\"{o} kernel near the diagonal
$$
P(z_1,z_2;F)=\frac{\mathrm{Id}\cdot dz}{z_1-z_2}+a_0(z_1;F)+a_1(z_1;F)(z_2-z_1)+O(|z_1-z_2|^2).
$$
The coefficients $a_0$ satisfies that
$$
a_0(z;F)-\rho h^{-1}\partial_z(h\rho^{-1})\in \Omega^{0,0}(X,\mathrm{End}(F)\otimes \omega_X).
$$
See \cite[pp26,(2.10)]{fay1992kernel} for details.

Now we can compute the expectation value $\langle J_{\nu}\rangle$ using the trace map we constructed before.

\begin{thm}\label{InsertionCurrent}
    The expectation value of $J_{\nu}$ is given by the following formula
    $$
    \langle J_{\nu}\rangle=\mathrm{Tr}(J_{\nu})=\frac{1}{\pi}\int_X\mathrm{tr}[\nu\cdot (a_0(z,F)-\rho h^{-1}\partial_z(h\rho^{-1})]dz
    $$
\end{thm}
\begin{proof}
    The term $\mathrm{tr}[\nu\cdot a_0(z,F)]$ comes from the self-loop diagram (see Fig. \ref{Jcurrent}).

\begin{figure}
    \centering

\tikzset{every picture/.style={line width=0.75pt}} 

\begin{tikzpicture}[x=0.75pt,y=0.75pt,yscale=-1,xscale=1]

\draw   (185.97,81.32) .. controls (223.32,66.35) and (262.9,105.13) .. (308.3,102.36) .. controls (353.7,99.59) and (444.51,26.19) .. (464.3,114.83) .. controls (484.09,203.47) and (333.91,166.07) .. (294.33,160.53) .. controls (254.75,154.99) and (197.61,189.35) .. (172.1,173) .. controls (146.59,156.64) and (148.62,96.29) .. (185.97,81.32) -- cycle ;
\draw  [draw opacity=0] (435.57,115.61) .. controls (430.54,125.46) and (418.82,132.72) .. (404.94,133.39) .. controls (389.77,134.14) and (376.62,126.79) .. (371.76,115.97) -- (403.99,105.77) -- cycle ; \draw   (435.57,115.61) .. controls (430.54,125.46) and (418.82,132.72) .. (404.94,133.39) .. controls (389.77,134.14) and (376.62,126.79) .. (371.76,115.97) ;
\draw  [draw opacity=0] (378.91,126.44) .. controls (383.04,120.53) and (391.52,116.36) .. (401.38,116.14) .. controls (412.3,115.89) and (421.68,120.55) .. (425.4,127.31) -- (401.67,134.16) -- cycle ; \draw   (378.91,126.44) .. controls (383.04,120.53) and (391.52,116.36) .. (401.38,116.14) .. controls (412.3,115.89) and (421.68,120.55) .. (425.4,127.31) ;
\draw  [draw opacity=0] (240.35,123.97) .. controls (235.87,134.81) and (224.52,142.89) .. (210.99,143.63) .. controls (196.34,144.43) and (183.69,136.36) .. (179.37,124.59) -- (210.09,114.7) -- cycle ; \draw   (240.35,123.97) .. controls (235.87,134.81) and (224.52,142.89) .. (210.99,143.63) .. controls (196.34,144.43) and (183.69,136.36) .. (179.37,124.59) ;
\draw  [draw opacity=0] (187.01,135.06) .. controls (190.72,128.5) and (198.98,123.8) .. (208.63,123.56) .. controls (219.21,123.29) and (228.27,128.45) .. (231.57,135.84) -- (208.9,142.43) -- cycle ; \draw   (187.01,135.06) .. controls (190.72,128.5) and (198.98,123.8) .. (208.63,123.56) .. controls (219.21,123.29) and (228.27,128.45) .. (231.57,135.84) ;
\draw  [fill={rgb, 255:red, 74; green, 144; blue, 226 }  ,fill opacity=1 ] (294.76,137.75) .. controls (294.76,136.47) and (295.75,135.43) .. (296.97,135.43) .. controls (298.2,135.43) and (299.19,136.47) .. (299.19,137.75) .. controls (299.19,139.03) and (298.2,140.07) .. (296.97,140.07) .. controls (295.75,140.07) and (294.76,139.03) .. (294.76,137.75) -- cycle ;
\draw [color={rgb, 255:red, 184; green, 233; blue, 134 }  ,draw opacity=1 ][fill={rgb, 255:red, 255; green, 255; blue, 255 }  ,fill opacity=1 ][line width=2.25]    (299.08,138) .. controls (380.08,171) and (422.08,99) .. (299.02,137.27) ;

\draw (317,109.4) node [anchor=north west][inner sep=0.75pt]  [font=\footnotesize]  {$\beta $};
\draw (344,152.4) node [anchor=north west][inner sep=0.75pt]  [font=\footnotesize]  {$\gamma $};
\draw (269,128.4) node [anchor=north west][inner sep=0.75pt]  [font=\footnotesize]  {$J_{\nu }$};

\end{tikzpicture}
    \caption{$\langle J_{\nu}\rangle$}
    \label{Jcurrent}
\end{figure}
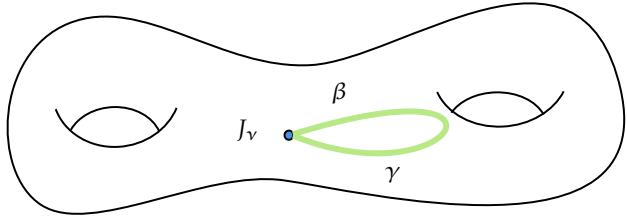

\end{proof}

\begin{rem}
If we choose $\nu$ to be a specific form related to the infinitesimal deformation of the holomorphic structure on $F$, the above formula matches the one in    \cite[pp79, Theorem 4.5]{fay1992kernel} (can be derived from \cite[pp60, (3.26)]{fay1992kernel}) which calculates the variation of analytic torsion.
\end{rem}

\subsection{Insertion of the modified energy-momentum tensor}

In this subsection, we turn to study the coupling between the energy-momentum tensor and Beltrami differentials.

We assume that $F$ is the trivial line bundle $X\times \mathbb{C}$. Let $\mu\in \Omega^{0,1}(X,\Theta_X)$ be a smooth (0,1)-form with valued in the tangent bundle $\Theta_X$. We can write $\mu$ locally as
$$
\mu=\mu^{z}_{\bar{z}}d\bar{z}\cdot \partial_z.
$$

Similar to the previous discussion, the naive guess
$$
T_{\mu}\stackrel{?}{=}\mu^z_{\bar{z}}d\bar{z}\cdot \beta_z\partial_z\gamma
$$
does not work either. But we can still correct it to a well-defined operator.

\begin{prop}
    We have the following well-defined global section
    $$
    \tilde{T}_{\mu}=\mu^{z}_{\bar{z}}d\bar{z}(\beta_z\cdot \partial_z\gamma-\frac{1}{3}\frac{\partial^2_z\rho}{\rho})dz\in \Omega^{0,1}(X,\mathcal{A})
    $$
\end{prop}
\begin{proof}
    Using the coordinate change formula for vertex algebra, we have
    $$
\beta_w\cdot \partial_w\gamma dw^{\otimes 2}=\beta_z\cdot \partial_z\gamma dz^{\otimes 2}+ \frac{1}{6}\frac{\theta''}{\theta}-\frac{1}{4} (\frac{\theta'}{\theta})^2.
    $$
    The second term is the Schwarzian derivative which is canceled out by the transformation formula of   $\frac{1}{3}\frac{\partial^2_z\rho}{\rho} dz^{\otimes 2}$.
\end{proof}

We have the following theorem which calculates the constant term of the trace map $\mathrm{Tr}(T_{\mu})$.

\begin{thm}
    The constant term of $\mathrm{Tr}(T_{\mu})$ is given by the following formula
    $$
    \frac{1}{\pi}\int_X\mathrm{tr}[\mu\cdot (a_1(z;E)-\partial_za_0(z;E)-\frac{1}{3}\rho^{-1}\partial^2_z\rho)]dz.
    $$
\end{thm}
\begin{proof}
    The proof is the same as the one in the previous section. We have the self-loop diagram (see Fig. \ref{Tcurrent}).
\begin{figure}
    \centering

\tikzset{every picture/.style={line width=0.75pt}} 

\begin{tikzpicture}[x=0.75pt,y=0.75pt,yscale=-1,xscale=1]

\draw   (185.97,81.32) .. controls (223.32,66.35) and (262.9,105.13) .. (308.3,102.36) .. controls (353.7,99.59) and (444.51,26.19) .. (464.3,114.83) .. controls (484.09,203.47) and (333.91,166.07) .. (294.33,160.53) .. controls (254.75,154.99) and (197.61,189.35) .. (172.1,173) .. controls (146.59,156.64) and (148.62,96.29) .. (185.97,81.32) -- cycle ;
\draw  [draw opacity=0] (435.57,115.61) .. controls (430.54,125.46) and (418.82,132.72) .. (404.94,133.39) .. controls (389.77,134.14) and (376.62,126.79) .. (371.76,115.97) -- (403.99,105.77) -- cycle ; \draw   (435.57,115.61) .. controls (430.54,125.46) and (418.82,132.72) .. (404.94,133.39) .. controls (389.77,134.14) and (376.62,126.79) .. (371.76,115.97) ;
\draw  [draw opacity=0] (378.91,126.44) .. controls (383.04,120.53) and (391.52,116.36) .. (401.38,116.14) .. controls (412.3,115.89) and (421.68,120.55) .. (425.4,127.31) -- (401.67,134.16) -- cycle ; \draw   (378.91,126.44) .. controls (383.04,120.53) and (391.52,116.36) .. (401.38,116.14) .. controls (412.3,115.89) and (421.68,120.55) .. (425.4,127.31) ;
\draw  [draw opacity=0] (240.35,123.97) .. controls (235.87,134.81) and (224.52,142.89) .. (210.99,143.63) .. controls (196.34,144.43) and (183.69,136.36) .. (179.37,124.59) -- (210.09,114.7) -- cycle ; \draw   (240.35,123.97) .. controls (235.87,134.81) and (224.52,142.89) .. (210.99,143.63) .. controls (196.34,144.43) and (183.69,136.36) .. (179.37,124.59) ;
\draw  [draw opacity=0] (187.01,135.06) .. controls (190.72,128.5) and (198.98,123.8) .. (208.63,123.56) .. controls (219.21,123.29) and (228.27,128.45) .. (231.57,135.84) -- (208.9,142.43) -- cycle ; \draw   (187.01,135.06) .. controls (190.72,128.5) and (198.98,123.8) .. (208.63,123.56) .. controls (219.21,123.29) and (228.27,128.45) .. (231.57,135.84) ;
\draw  [fill={rgb, 255:red, 74; green, 144; blue, 226 }  ,fill opacity=1 ] (294.76,137.75) .. controls (294.76,136.47) and (295.75,135.43) .. (296.97,135.43) .. controls (298.2,135.43) and (299.19,136.47) .. (299.19,137.75) .. controls (299.19,139.03) and (298.2,140.07) .. (296.97,140.07) .. controls (295.75,140.07) and (294.76,139.03) .. (294.76,137.75) -- cycle ;
\draw [color={rgb, 255:red, 184; green, 233; blue, 134 }  ,draw opacity=1 ][fill={rgb, 255:red, 255; green, 255; blue, 255 }  ,fill opacity=1 ][line width=2.25]    (299.08,138) .. controls (380.08,171) and (422.08,99) .. (299.02,137.27) ;

\draw (317,109.4) node [anchor=north west][inner sep=0.75pt]  [font=\footnotesize]  {$\beta $};
\draw (344,152.4) node [anchor=north west][inner sep=0.75pt]  [font=\footnotesize]  {$\partial \gamma $};
\draw (270,126.4) node [anchor=north west][inner sep=0.75pt]  [font=\footnotesize]  {$T_{\mu }$};

\end{tikzpicture}
    \caption{$\langle T_{\mu}\rangle$}
    \label{Tcurrent}
\end{figure}
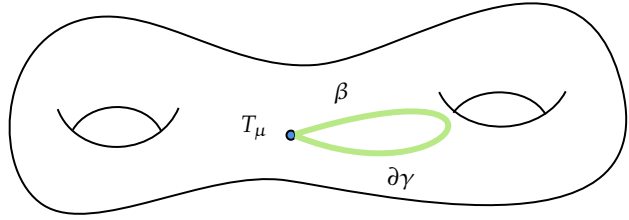
\end{proof}

\begin{rem}
The above formula recovers the first term of \cite[pp60,(3.27)]{fay1992kernel}. There Fay calculates the variation of the analytic torsion along the moduli space of curves.
\end{rem}
\section{Appendix}

\subsection{ Vertex algebra bundles}\label{VOAappendix}
We review the notion of vertex algebra and the construction of vertex algebra bundles over a smooth complex curve. See \cite{frenkel2004vertex,kac1998vertex}  for more details.

\begin{defn}

A vertex algebra is the following data:
\begin{itemize}
  \item the space of states: a vector space $\mathbf{V}$,
  \item the vacuum vector: a vector $|0\rangle \in \mathbf{V}$,
  \item the state-field correspondence: a linear map
of $\mathbf{V}$ to the space of fields, $a \mapsto Y(a, z)=\sum_{n \in \mathbb{Z}} a_{(n)} z^{-n-1}$,
satisfying the following axioms:
\end{itemize}

(translation covariance): $[T, Y(a, z)]=\partial Y(a, z)$,
where $T \in \mathrm{End} \mathbf{V}$ is defined by
$T(a)=a_{(-2)}|0\rangle$;

(vacuum): $Y(|0\rangle, z)=\mathrm{Id}_{\mathbf{V}},\left.Y(a, z)|0\rangle\right|_{z=0}=a$;

(locality): $(z-w)^{N} Y(a, z) Y(b, w)$
$$
=(z-w)^{N} Y(b, w) Y(a, z) \text { for } N \gg 0 .
$$

\end{defn}

Let $\mathcal{O}=\mathbb{C}[[z]]$. A conformal vertex algebra $\mathbf{V}$ is equipped with an action of $\mathrm{Der}_0\mathcal{O}=z\mathbb{C}[[z]]\partial_z$ under the correspondence $z^{j+1}\partial_z\mapsto-L_j$. Let us assume that the $\mathbb{Z}$-grading of $\mathbf{V}$ induced by the operator $L_0$ is bounded below. Then the Lie algebra $\mathrm{Der}_+\mathcal{O}=z^2\mathbb{C}[[z]]\partial_z$ is locally nilpotent as each of its elements has a negative degree.

Now suppose that we have an element in $\mathrm{Aut}(\mathcal{O})$ which is expressed as a power series
$$
z\mapsto f(z)=a_1z+a_2z^2+\dots,\quad a_1\neq 0.
$$

We can find $v_i\in \mathbb{C}, i\geq 0$, such that
$$
f(z)=\exp{(\sum_{i>0}v_iz^{i+1}\cdot \partial_z)}v_0^{z\partial_z}\cdot z.
$$

We define a linear operator $R(f)\in\mathrm{End}(\mathbf{V})$ by
$$
R(f)=\exp{(-\sum_{i>0}v_iL_i)}v_0^{-L_0}.
$$

We have a filtration $\mathbf{V}_{\leq m}$ of $\mathbf{V}$ be finite dimensional $\mathrm{Aut}(\mathcal{O})$-submodules. Hence we can define $\mathcal{V}$ as inductive limit of vector bundle $\mathcal{V}_i:=\mathcal{A}ut_X\times_{\mathrm{Aut}(\mathcal{O})}V_i.$

We can now describe the local coordinate change formula for $\mathcal{V}$ explicitly. When $D_x$ is a small analytic disc around $x$. Then the local coordinate $z$ on $D_x$ such that $z(x)=0$ induces a local coordinate $z_y:=z-z(y)$ at each point $y\in D_x$ satisfying $z_y(y)=0$. If we choose a new local coordinate $w=\rho(z)$. Then $w_y(t)=w(t+y)-w(y)=\rho(t+y)-\rho(y)=:\rho_y(t)$. We then have the coordinate change formula
$$
(z_y,v)\sim (w_y,R(\rho_y)^{-1}v).
$$

Now we turn to general twisting construction. Let $\mathfrak{g}$ be a simple Lie algebra. Suppose there is an injection $\alpha:\mathfrak{g}\rightarrow V$ such that the Fourier coefficients of the vertex operator $Y(\alpha(A),z), A\in \mathfrak{g}$, generate an action of $\widehat{\mathfrak{g}}$ on $V$. Furthermore, we assume that the conformal structure of $\mathbf{V}$ is compatible with the action of $\widehat{\mathfrak{g}}$. This means that the actions of $\mathrm{Vir}$ and $\widehat{\mathfrak{g}}$ combine into the action of the semi-direct product $\mathrm{Vir}\ltimes \widehat{\mathfrak{g}}$.

Suppose that the action of $\mathfrak{g}(\mathcal{O})$ on $\mathbf{V}$ can be exponentiate to a $G(\mathcal{O})-$action. Then we have an action of the semi-direct product $\mathrm{Aut}(\mathcal{O})\ltimes G(\mathcal{O}).$ Let $\mathcal{P}$ be a principal $G$-bundle on $X$. There is a principal $\mathrm{Aut}(\mathcal{O})\ltimes G(\mathcal{O})$-bundle $\widehat{P}$ over $X$ whose fiber at $X$ is $\widehat{\mathcal{P}}_x:=\{(z,s)| z\ \text{is a formal coordinate at } \ x, s \text{is a trivialization of }\ \mathcal{P}|_{D_x}\}.$ Then we can construct a vertex algebra bundle
$$
\mathcal{V}:=\widehat{P}\times_{\mathrm{Aut}(\mathcal{O})\ltimes G(\mathcal{O})}\mathbf{V}.
$$
For any element $\sigma(t)\in G(\mathcal{O})$, we can find
$$
\exp({\sum_{n\geq 0,i=1,\dots,\dim \mathfrak{g}} f^i_n(y)t^nE_i})=\sigma(y+t).
$$
Here $\{E_i\}_{i=1,\dots,\dim\mathfrak{g}}$ is a basis of $\mathfrak{g}$.
Define
$$
R(\sigma_y)=\exp(-{\sum_{n\geq 0,i=1,\dots,\dim \mathfrak{g}} f^i_n(y) J_{i,n}}),
$$
where $J_i$ is the affine Kac-Moody current corresponding to $E_i\in \mathfrak{g}$.

Similar to the discussion before,  we have
$$
(z_y,s;v)\sim (z_y,\tilde{s};R(\sigma_y)^{-1}v).
$$
Finally, we discuss a bit about the case of non-integral gradation following \cite[pp130, Section 7.3.9]{frenkel2004vertex}. Suppose that our vertex algebra is $\frac{1}{r}\mathbb{Z}$-graded by $L_0$ and we have a r-spin structure on $X$, that is, an r-cover of the frame bundle $\mathrm{Fr}_X$
$$
\pi:\mathcal{P}_{r-\mathrm{spin}}\rightarrow \mathrm{Fr}_X.
$$
First of all, we can still form a bundle over $\mathrm{Fr}_X$ by twisting
$$
\mathcal{V}_+=\mathcal{A}ut_X\times_{\mathrm{Aut}_+(\mathcal{O})}\mathbf{V}.
$$
Where $\mathrm{Aut}_+\mathcal{O}=\{z+a_2z^2+\dots\}\subset \mathrm{Aut}\mathcal{O}$. Using the r-spin structure we can pull back this bundle to $\mathcal{P}_{r-\mathrm{spin}}$ and do further $\mathbb{C}^{\times}$-twist
$$
\mathcal{V}:=\pi^*\mathcal{V}_+/\{(p\cdot t,v)\sim (p,t^{-rL_0}v)\}.
$$

\subsection{Coordinate change
formulas for chiral Weyl algebras}\label{VAbundleChiral}

We want to show that the chiral enveloping algebra is isomorphic to the chiral algebra constructed from the vertex algebra bundle
$$
\mathscr{U}(\mathcal{L})^{\flat}\simeq \mathcal{V}^r.
$$
To prove this, we first notice that locally their chiral algebra structure are the same. This can be seen from the Wick theorem that will be proved in the next section. Since both $\mathscr{U}(\mathcal{L})^{\flat}$ and $\mathcal{V}^r$ are generated by linear fields, we only need to show that they have the same coordinate change formula on linear fields. Furthermore,  the formula in \cite[pp118,(6.6.1)]{frenkel2004vertex}
$$
\rho'(z)(\partial_w+L_{-1})R(\rho_z)^{-1}(-)=R(\rho_z)^{-1}L_{-1}(-)
$$
matches of the coordinate change formula on $\mathscr{U}(\mathcal{L})^{\flat}$ side.

This means that we can further reduce to comparing the bundle $E\subset \mathscr{U}(\mathcal{L})$ and the bundle of primary fields in $\mathcal{V}^r$ which are isomorphic by construction.
\subsection{Proof of the Wick theorem}\label{WickThm}

We will prove Theorem \ref{WickTheoremChiral}. A choice of a local coordinate induces
 a local isomorphism
 $$
\tau^z_U:\mathcal{A}|_U\rightarrow \mathrm{Sym}\ \mathcal{L}|_{U}.
 $$
$$
\frac{n_1!\cdot e_{i_1}}{(t-z)^{n_1+1}}\cdots \frac{n_k!\cdot e_{i_k}}{(t-z)^{n_k+1}}|0\rangle\cdot dz\leftrightarrow    ( e_{i_1}\otimes \partial_z^{n_1}\otimes dz^{-1}\cdots e_{i_k}\otimes \partial_z^{n_k}\otimes dz^{-1})\cdot dz.
$$

Recall the diagram
 $$
 \mathbf{U}_{X}\boxtimes  \mathbf{U}_{X}\xrightarrow{c} U(\mathcal{L}^{\natural}_{X^2})/U(\mathcal{L}^{\natural}_{X^2})p^*_2\mathcal{L}^{\natural}_X\otimes U(\mathcal{L}^{\natural}_{X^2})/U(\mathcal{L}^{\natural}_{X^2})p^*_1\mathcal{L}^{\natural}_X\xleftarrow{\iota} \mathbf{U}_{X^2}.$$

First notice that
\begin{equation}
    \label{WickId1}
\iota(v_1\otimes v_2)=c(e^{-\mathfrak{P}_{\mathbf{sing}}}v_1\boxtimes v_2).
\end{equation}

 Here we write
$$
v_1\otimes v_2=(\frac{a_1}{(t-z_1)^{k_1+1}} \cdots  \frac{a_n}{(t-z_1)^{k_n+1}})\otimes (\frac{b_1}{(t-z_2)^{l_1+1}} \cdots  \frac{b_m}{(t-z_2)^{l_m+1}})\in \mathbf{U}_{X^2}.
$$

To prove (\ref{WickId1}), we need to compute
$$
\iota(v_1\otimes v_2)=
$$
$$
\sum_{I\subset \{1,\dots,n\}}(\prod_{i\notin I}\frac{a_i}{(t-z_1)^{k_i+1}})|0\rangle_{z_1}\boxtimes (\prod_{i \in I}\frac{a_i}{(t-z_1)^{k_i+1}})\cdot (\prod_{j=1}^m\frac{b_j}{(t-z_1)^{l_j+1}})|0\rangle_{z_2}.
$$

Exchanging $\frac{a_i}{(t-z_1)^{k_i+1}}$ and $\frac{b_j}{(t-z_1)^{l_j+1}}$ we get
$$
\frac{\langle a_i,b_j\rangle}{(t-z_1)^{k_i+1}(t-z_2)^{l_j+1}}|0\rangle_{z_2}=\frac{1}{l_j!}\cdot (-k_i-1)\cdot (-k_i-l_j)\cdot \frac{1}{(z_2-z_1)^{k_i+l_j+1}}
$$
$$
=\frac{(k_i+l_j)!}{k_i!l_j!}(-1)^{k_i+1}\frac{1}{(z_1-z_2)^{k_i+l_j+1}}
$$

While the Wick contraction gives the same result with a minus sign
$$
\frac{1}{k_i!l_j!}\cdot \partial^{k_i}_{z_1}\partial^{l_j}_{z_2}\frac{1}{z_1-z_2}=
\frac{(k_i+l_j)!}{k_i!l_j!}\cdot (-1)^{k_i}\cdot \frac{1}{(z_1-z_2)^{k_i+l_j+1}}
$$
and (\ref{WickId1}) is proved.

We can find $N$ sufficiently large such that
$$
\iota((z_1-z_2)^N\cdot e^{\mathfrak{P}^{\otimes}_{\mathbf{sing}}}v_1\otimes v_2)=c((z_1-z_2)^N \cdot v_1\boxtimes v_2).
$$
Here the operation $e^{\mathfrak{P}^{\otimes}_{\mathbf{sing}}}$ is defined in the same way as the operator $e^{\mathfrak{P}_{\mathbf{sing}}}$ by treating $\otimes$ as $\boxtimes$.

Then we have
$$
\mu(\eta\cdot v_1dz_1\boxtimes v_2 dz_2)=\mu_{{\omega}}(\frac{\eta}{(z_1-z_2)^N})\cdot ((z_1-z_2)^Ne^{\mathfrak{P}^{\otimes}_{\mathbf{sing}}}(v_1\otimes v_2))
$$
$$
=\mu_{\mathrm{Sym}}(\eta\cdot e^{\mathfrak{P}_{\mathbf{sing}}}v_1dz_1\boxtimes v_2dz_2)).
$$

\bibliographystyle{amsplain}
\providecommand{\bysame}{\leavevmode\hbox to3em{\hrulefill}\thinspace}
\providecommand{\MR}{\relax\ifhmode\unskip\space\fi MR }
\providecommand{\MRhref}[2]{%
  \href{http://www.ams.org/mathscinet-getitem?mr=#1}{#2}
}
\providecommand{\href}[2]{#2}

\end{document}